\documentclass[12pt,a4paper]{amsart}

\usepackage{pifont}
\usepackage{pgf,tikz,pgfplots}
\usetikzlibrary{arrows}

\makeatletter
\newcommand{\labeltext}[2]{%
	\@bsphack
	\def\@currentlabel{#1}{\label{#2}}%
	\@esphack
}
\makeatother

\def\step#1#2#3{\par \noindent{{\\ \bf Step~\labeltext{#1}{#3}#1. }{\bf #2. }}}

\makeatletter
\long\def\unmarkedfootnote#1{{\long\def\@makefntext##1{##1}\footnotetext{#1}}}
\makeatother

\usepackage{amsmath,amssymb}
\usepackage{enumerate,amssymb}
\usepackage{color}
\theoremstyle{plain}
\newtheorem{thm}{Theorem}[section]
\newtheorem{thmx}{Theorem}

\newtheorem{lemma}[thm]{Lemma}

\newtheorem{corollary}[thm]{Corollary}
\newtheorem{prop}[thm]{Proposition}
\newtoks\prt

\theoremstyle{definition}

\newtheorem{remark}[thm]{Remark}
\newtheorem{remarks}[thm]{Remarks}
\newtheorem{definition}[thm]{Definition}
\newtheorem{example}[thm]{Example}

\def\eqn#1$$#2$${\begin{equation}\label#1#2\end{equation}}

\numberwithin{equation}{section}

\def\diam{\operatorname{diam}}
\def\INV{\operatorname{INV}^+}
\def\dist{\operatorname{dist}}

\def\epsilon{\varepsilon}
\def\en{\mathbb N}
\def\er{\mathbb R}

\def\G{\mathcal{G}}

\def\id{\operatorname{id}}
\def\inter{\operatorname{int}}

\def\co{\operatorname{co}}

\def\loc{\operatorname{loc}}

\def\mir1{\mathcal L_1}

\def\oint{-\hskip -13pt \int}

\def\phi{\varphi}

\def\sto{\rightrightarrows}
\def\t{{\bf t}}

\def\rn{\mathbb R^n}
\def\er{\mathbb R}

\def\x{\widetilde{x}}

\def\zet{\mathbb Z}

\definecolor{ffqqqq}{rgb}{1.,0.,0.}
\definecolor{qqqqqq}{rgb}{0,0.,0.}
\definecolor{qqqqff}{rgb}{0.,0.,1.}
\definecolor{uuuuuu}{rgb}{0.26666666666666666,0.26666666666666666,0.26666666666666666}
\definecolor{ffwwqq}{rgb}{1.,0.4,0.}
\definecolor{ffqqqq}{rgb}{1.,0.,0.}
\definecolor{qqqqff}{rgb}{0.,0.,1.}
\definecolor{qqffqq}{rgb}{0.,1.,0.}
\definecolor{wwqqzz}{rgb}{0.4,0.,0.6}
\definecolor{zzttqq}{rgb}{0.6,0.2,0.}
\definecolor{qqqqff}{rgb}{0.,0.,1.}
\definecolor{yqqqyq}{rgb}{0.5019607843137255,0.,0.5019607843137255}
\definecolor{qqzzqq}{rgb}{0.,0.6,0.}
\definecolor{sqsqsq}{rgb}{0.12549019607843137,0.12549019607843137,0.12549019607843137}

\newtoks\by
\newtoks\paper
\newtoks\book
\newtoks\jour
\newtoks\yr
\newtoks\pages
\newtoks\vol
\newtoks\publ
\def\ota{{\hbox\vol{???}}}
\def\cLear{\by=\ota\paper=\ota\book=\ota\jour=\ota\yr=\ota
\pages=\ota\vol=\ota\publ=\ota}
\def\endpaper{\the\by, {\the\paper},
\textit{\the\jour} \textbf{\the\vol} (\the\yr), \the\pages.\cLear}
\def\endbook{\the\by, \textit{\the\book}, \the\publ.\cLear}
\def\endprep{\the\by, \textit{\the\paper}, \the\jour.\cLear}
\def\endyearprep{\the\by, \textit{\the\paper}, \the\jour, (\the\yr).\cLear}
\def\name#1#2{#2 #1}
\def\nom{ \rm no. }

\usepackage{hyperref}
\def\H{\mathcal{H}}
\def\de0#1{\rule[3pt]{#1}{0.4pt} \hspace{-0.1pt} \rule[3.05pt]{0.05pt}{0.4pt} \hspace{-0.1pt} \rule[3.1pt]{0.05pt}{0.4pt} \hspace{-0.1pt} \rule[3.15pt]{0.05pt}{0.4pt} \hspace{-0.1pt} \rule[3.2pt]{0.05pt}{0.4pt} \hspace{-0.1pt} \rule[3.25pt]{0.05pt}{0.4pt} \hspace{-0.1pt} \rule[3.3pt]{0.05pt}{0.4pt} \hspace{-0.1pt} \rule[3.35pt]{0.05pt}{0.4pt} \hspace{-0.1pt} \rule[3.4pt]{0.05pt}{0.4pt} \hspace{-0.1pt} \rule[3.45pt]{0.05pt}{0.4pt} \hspace{-0.1pt} \rule[3.5pt]{0.05pt}{0.4pt} \hspace{-0.1pt} \rule[3.55pt]{0.05pt}{0.4pt} \hspace{-0.1pt} \rule[3.6pt]{0.05pt}{0.4pt} \hspace{-0.1pt} \rule[3.65pt]{0.05pt}{0.4pt} \hspace{-0.1pt} \rule[3.7pt]{0.05pt}{0.4pt} \hspace{-0.1pt} \rule[3.75pt]{0.05pt}{0.4pt} \hspace{-0.1pt} \rule[3.8pt]{0.05pt}{0.4pt} \hspace{-0.1pt} \rule[3.85pt]{0.05pt}{0.4pt} \hspace{-0.1pt} \rule[3.9pt]{0.05pt}{0.4pt} \hspace{-0.1pt} \rule[3.95pt]{0.05pt}{0.4pt} \hspace{-0.1pt} \rule[4.0pt]{0.05pt}{0.4pt} \hspace{-0.1pt} \rule[4.05pt]{0.05pt}{0.4pt} \hspace{-0.1pt} \rule[4.1pt]{0.05pt}{0.4pt} \hspace{-0.1pt} \rule[4.15pt]{0.05pt}{0.4pt} \hspace{-0.1pt} \rule[4.2pt]{0.05pt}{0.4pt} \hspace{-0.1pt} \rule[4.25pt]{0.05pt}{0.4pt} \hspace{-0.1pt} \rule[4.3pt]{0.05pt}{0.4pt} \hspace{-0.1pt} \rule[4.35pt]{0.05pt}{0.4pt} \hspace{-0.1pt} \rule[4.4pt]{0.05pt}{0.4pt} \hspace{-0.1pt} \rule[4.45pt]{0.05pt}{0.4pt} \hspace{-0.1pt} \rule[4.5pt]{0.05pt}{0.4pt} \hspace{-0.1pt} \rule[4.55pt]{0.05pt}{0.4pt} \hspace{-0.1pt} \rule[4.6pt]{0.05pt}{0.4pt} \hspace{-0.1pt} \rule[4.65pt]{0.05pt}{0.4pt} \hspace{-0.1pt} \rule[4.7pt]{0.05pt}{0.4pt} \hspace{-0.1pt} \rule[4.75pt]{0.05pt}{0.4pt} \hspace{-0.1pt} \rule[4.8pt]{0.05pt}{0.4pt} \hspace{-0.1pt} \rule[4.85pt]{0.05pt}{0.4pt} \hspace{-0.1pt} \rule[4.9pt]{0.05pt}{0.4pt} \hspace{-0.1pt} \rule[4.95pt]{0.05pt}{0.4pt} \hspace{-0.1pt} \rule[5.0pt]{0.05pt}{0.4pt} \hspace{-0.1pt} \rule[5.05pt]{0.05pt}{0.4pt} \hspace{-0.1pt} \rule[5.1pt]{0.05pt}{0.4pt} \hspace{-0.1pt} \rule[5.15pt]{0.05pt}{0.4pt} \hspace{-0.1pt} \rule[5.2pt]{0.05pt}{0.4pt} \hspace{-0.1pt} \rule[5.25pt]{0.05pt}{0.4pt} \hspace{-0.1pt} \rule[5.3pt]{0.05pt}{0.4pt} \hspace{-0.1pt} \rule[5.35pt]{0.05pt}{0.4pt} \hspace{-0.1pt} \rule[5.4pt]{0.05pt}{0.4pt} \hspace{-0.1pt} \rule[5.45pt]{0.05pt}{0.4pt} \hspace{-0.1pt} \rule[5.5pt]{0.05pt}{0.4pt} \hspace{-0.1pt} \rule[5.55pt]{0.05pt}{0.4pt} \hspace{-0.1pt} \rule[5.6pt]{0.05pt}{0.4pt} \hspace{-0.1pt} \rule[5.65pt]{0.05pt}{0.4pt} \hspace{-0.1pt} \rule[5.7pt]{0.05pt}{0.4pt} \hspace{-0.1pt} \rule[5.75pt]{0.05pt}{0.4pt} \hspace{-0.1pt} \rule[5.8pt]{0.05pt}{0.4pt} \hspace{-0.1pt} \rule[5.85pt]{0.05pt}{0.4pt} \hspace{-0.1pt} \rule[5.9pt]{0.05pt}{0.4pt} \hspace{-0.1pt} \rule[5.95pt]{0.05pt}{0.4pt} \hspace{-0.1pt} \rule[6.0pt]{0.05pt}{0.4pt}}	
\def\deb{\mathop{\de0{16pt} ~}\limits}		
\def\xdeb#1{\mathop{\hspace{1pt} ^{#1} \hspace{-21pt} \de0{25pt} ~}\limits}	

\date{\today}

\begin{document}

\title[Characterizations of limits of planar homeomorphisms]{Point-wise characterizations of limits of planar Sobolev homeomorphisms and their quasi-monotonicity}

\author{Daniel Campbell}
\address{Department of Mathematical Analysis, Charles University,
	So\-ko\-lovsk\'a 83, 186~00 Prague 8, Czech Republic}
\email{\tt campbell@karlin.mff.cuni.cz}


\begin{abstract}
	We present three novel classifications of the weak sequential (and strong) limits in $W^{1,p}$ of planar diffeomorphisms. We introduce a concept called QM condition which is a kind of separation property for pre-images of closed connected sets and show that $u$ satisfies this property exactly when it is the limit of Sobolev homeomorphisms. Further, we prove that $u\in W^{1,p}_{\id}((-1,1)^2,\er^2)$ is the limit of a sequence of  homeomorphisms exactly when there are classically monotone mappings $g_{\delta}:[-1,1]^2\to \er^2$ and very small open sets $U_{\delta}$ such that $g_{\delta} = u$ on $[-1,1]^2 \setminus U_{\delta}$. Also, we introduce the so-called the three curve condition, which is in some sense reminiscent of the NCL condition of \cite{CPR} but for $u^{-1}$ instead of for $u$, and prove that a map is the $W^{1,p}$ limit of planar Sobolev homeomorphisms exactly when it satisfies this property. This improves on results in \cite{DPP} answering the question from \cite{IO2}.
\end{abstract} 

\maketitle

\section{Introduction}
	
	We are naturally led to the study of weak (and strong) limits of Sobolev homeomorphisms (or diffeomorphisms) from $\rn$ to $\rn$ in the context of variational models of bulk energy in non-linear elasticity, because we may consider all such mappings to be meaningful deformations. In the case $n=2$ and $p\geq2$, a succinct categorization was provided by Iwaniec and Onninen in \cite{IO2}. Their theorem says that  the weak closure of Sobolev homeomorphisms coincides with the strong closure of diffeomorphisms and is characterized by Sobolev monotone maps.
	\begin{definition}\label{defMonotone}
		Let $X,Y\subset \er^2$ be compact connected sets. We say that a continuous $u:X\to Y$ is a monotone map from $X$ onto $Y$ if $u^{-1}(y) = \{x\in X; \ u(x) = y\}$ is a closed connected set.
	\end{definition}
	
	A characterization of monotone maps is thanks to Youngs \cite{Y}
	\begin{thm}\label{Young}
		Let $X,Y \subset \er^2$ be Jordan domains of the same topological type. A map $h : X \xrightarrow{onto} Y$ is monotone if and only if it is the
		uniform limit of a sequence of homeomorphisms $u_k : X\xrightarrow{onto} Y$. 
	\end{thm}
	For $p>2=n$ the usual embedding of $W^{1,p}$ into continuous functions suffices to realize that limits in $W^{1,p}$ must also be uniform limits. The case $p=2$ is substantially more challenging. It was known (see \cite[Corollary 7.5.1]{IM}) that planar homeomorphisms in $W^{1,2}$ have a modulus of continuity. For analytical reasons proven in \cite[Lemma 4.3, Lemma 4.5]{IO2} the modulus of continuity is conveyed to their limits. This was instrumental in proving their theorem \cite[Theorem 1.2]{IO2} cited below. 
	\begin{thm}
		Let $X$ and $Y$ be bounded multiply connected Lipschitz domains in $\er^2$ and $h_j : X \xrightarrow{onto} Y$ homeomorphisms converging weakly to $h$ in the space $W^{1,p}(X, \er^2), p \geq 2$. Then there exists a sequence of homeomorphisms (actually $\mathcal{C}^{\infty}$-diffeomorphisms)
		$$
		h^*_j : X \xrightarrow{onto} Y, \ \ h^*_j \in h + W^{1,p}_0 (X, \er^2),
		$$
		converging to $h$ strongly in $W^{1,p}(X, \er^2)$. The same holds for simply connected domains, except for $p = 2$, in which case the mappings $h_j$ must be fixed either at one point in $X$ or at three points on $\partial X$.
	\end{thm}
	Moreover, the same authors show in \cite{IO} that even for $p\in [1,2)$ any $W^{1,p}$ limit of planar homeomorphisms is a monotone map, if we assume a priori that the limit is a continuous map.
	
	A key result of \cite{DPP} (found in Theorem~B) is that the weak sequential closure and strong closure of homeomorphisms and/or diffeomorphisms in $W^{1,p}_{\id}((-1,1)^2,\er^2)$ (for the definition see Section~\ref{SSS}) coincide. Therefore, in the following we will refer only to the closure of homeomorphisms in $W^{1,p}_{\id}((-1,1)^2,\er^2)$ and for short we will write this as $\overline{H\cap W^{1,p}_{\id}((-1,1)^2,\er^2)}$.
	
	In \cite[Question 2.4]{IO2} the authors pose the question, whether the class of $W^{1,p}$ ``monotone mappings'' coincides with the class of limits of Sobolev homeomorphisms.  It is clear that being the limit of homeomorphisms must convey some sense of monotonicity. What is less clear is what is the `correct' monotonicity condition to guarantee that a map is the limit of homeomorphisms. The reason for this is the lack of continuity of the mappings in the class of limits. For a continuous map the following conditions are all equivalent (see \cite[Proposition 3.4]{IO2} or \cite[Lemma 4.24]{DPP}):
	\begin{itemize}
		\item the preimage of every point is a closed connected set, equivalently the preimage of a closed connected set is a closed connected set, equivalently the map satisfies the $\INV$ condition (see Definition~\ref{defINV}, Lemma~\ref{IDidThisOne} and Lemma~\ref{StolenFromAldo})
		\item the preimage of every point is a closed connected set which does not disconnect the plane (we call this the very weak monotonicity (VWM) condition)
		\item the preimage of every closed connected set which does not disconnect the plane is a closed connected set which does not disconnect the plane (we call this the weak monotonicity (WM) condition)
		\item the map is the uniform limit of homeomorphisms (we call this Youngs monotone).
	\end{itemize}
	
	It is immediate that WM implies VWM implies $\INV$. It is easy to check that homeomorphisms satisfy the $\INV$ condition and it was already known that this class is closed under Sobolev limits. Thus the $\INV$ condition was known to be necessary for a map to a limit of homeomorphisms. For some time, many experts believed that the $\INV$ condition was precisely the correct monotonicity condition to guarantee that a Sobolev map can be approximated by homeomorphisms. One of the results in the ground-breaking paper for the sub-critical case \cite{DPP}, however, is that there are $W^{1,p}_{\id}((-1,1)^2,\er^2)$ mappings satisfying the $\INV$ condition which are not in $\overline{H\cap W^{1,p}_{\id}((-1,1)^2,\er^2)}$.
	
	With great attention to technical details it was shown in \cite[Theorem~C]{DPP} that if $u$ is in $W^{1,p}_{\id}((-1,1)^2,\er^2)$ and is weakly monotone, then $u\in\overline{H\cap W^{1,p}_{\id}((-1,1)^2,\er^2)}$. Although the proof is very intricate and insightful, there are reasons not to be satisfied with the sufficiency result since already the radial stretching map $x\to \frac{x}{|x|}$ fails the WM condition (see Lemma~\ref{HaHaJoke}). We are able to significantly simplify the proof using our novel QM condition and prove also Theorem~\ref{SpanishInquisition}. The authors of \cite{DPP} propose the question: ``is the VWM condition sufficient to guarantee that if $u$ is in $W^{1,p}_{\id}((-1,1)^2,\er^2)$ then $u$ is in $\overline{H\cap W^{1,p}_{\id}((-1,1)^2,\er^2)}$?'' This question is extremely challenging since none of the known counter examples can refute the claim but it is not obvious how to exploit the condition to find a limiting sequence of injective maps. The authors of this paper agree that the question is very interesting and would find a negative answer very surprising. On the other hand, the authors of \cite{DPP} already knew that not even the VWM condition is necessary for a map to be the limit of homeomorphisms (see the example in \cite[Section~5.4]{DPP}). 
	
	Another breakthrough result of \cite{DPP} was their characterization for the Sobolev limits of planar homeomorphisms. This result has been very useful, for example in showing the way how to approach strict limits in BV for example (see \cite{CKR}). To this point it is the only characterization known (to the author) for limits of planar homeomorphisms in $W^{1,p}$ for $p\in [1,2)$. In order to be able to approximate a planar map in $W^{1,p}$ by diffeomorphisms the authors of \cite{DPP} devised the so-called no-crossing condition. Roughly speaking, $u\in W^{1,p}_{\id}((-1,1)^2, \er^2)$ satisfies the no crossing condition if for (almost) every grid of lines $\Gamma\subset [-1,1]^2$ so that $u_{\rceil\Gamma}$ is continuous, there exist continuous injective $\phi_m:\Gamma\to\er^2$ with $\phi_m\sto u$ on $\Gamma$. For the precise definition see Definition~\ref{defNC}. Although this condition allows the authors to readily approximate their map, it is a rather complicated condition to check and there has been repeated interest in finding a simpler or more explicit classification, which comes closer to answering the question from \cite{IO2}, i.e. we would like to find the `correct' monotonicity condition. Note that \cite{CPR} was one such attempt to find such a characterization using a condition called NCL (see Definition~\ref{defNCL}). 
	
	We give three such equivalent conditions. The first of these we call \textit{quasi-monotonicity}, for the precise definition see Definition~\ref{DefQuasiMonotne}. Roughly speaking, $u$ is quasi-monotone when for every $\delta$ we can find a monotone $g_{\delta}$ such that $u = g_{\delta}$ outside a very small set, i.e. outside $U_{\delta}$, where $\H^1_{\delta}(U_{\delta})<\delta$. This is even stronger than saying that $u$ is Youngs monotone on $[-1,1]^2\setminus U_{\delta}$.
	
	The second condition is called the three-curve condition. Loosely speaking, we can say that, if for some reason we knew that the preimage of every injective Lipschitz curve $\psi$ is a Lipschitz curve $\gamma$, then the three curve condition is akin to saying that $\gamma$ is not self-crossing i.e. there are injective Lipschitz curves converging to $\gamma$ in the uniform metric (by curve we mean parameterized curve). Of course there is no reason to expect that the preimage of a curve should look like a curve, even for continuous monotone maps (unless $J_u>0$ a.e.) so the condition is somewhat more complicated and can be found at Definition~\ref{ThreeCurvePropertyDef} (refer also to Definition~\ref{AppropriateCouple}). As stated earlier, a related condition called the NCL condition was studied in \cite{CPR}. The NCL condition, loosely speaking, means that the image of an injective Lipschitz curve in $u$  is not self-crossing (i.e. can be approximated uniformly by continuous injective curves). It was proved that the NCL condition did not imply even the $\INV$ condition (unless $J_u>0$ a.e.). For this reason, we consider the fact that a similar condition on the inverse mapping is equivalent with the map being the Sobolev limit of homeomorphisms quite noteworthy.
	
	The third condition is called the QM condition and can be found at Definition~\ref{defQM}. It is a kind of separation property for the preimages of closed connected sets whose union does not disconnect the plane. It suffices to verify for the condition for injective Lipschitz curves whose union do not disconnect the plane. If for some reason we knew that the preimage of the pair of injective Lipschitz curves was a pair of injective Lipschitz curves, then the QM condition is akin to saying that the two preimage curves (although they may intersect each other) cannot cross over each other. We express this by saying that for the pair of closed connected sets $C_1,C_2$ and every $r,\delta>0$ there are disjoint Jordan domains $V_i\subset u^{-1}(C_i+B(0,r))\cup U_{\delta}$ containing $u^{-1}(C_i)\setminus U_{\delta}$, where $U_{\delta}$ are some small open sets as above. For another description see Remarks~\ref{QMRemarks}, point $(1)$.
	
	The characterization theorem is as follows.
	\begin{thmx}\label{MainTheorem}
		Let $p\in [1,2]$ and let $u\in W^{1,p}_{\id}((-1,1)^2, \er^2)$. Then the following are equivalent
		\begin{enumerate}
			\item[a)] $u \in \overline{H\cap W_{\id}^{1,p}((-1,1)^2, \er^2)}$,
			\item[b)] $u$ is quasi-monotone on $[-1,1]^2$ (see Definition~\ref{DefQuasiMonotne}),
			\item[c)] $u$ has the three-curve property (see Definition~\ref{ThreeCurvePropertyDef}),
			\item[d)] $u$ satisfies the QM property (see Definition~\ref{defQM}).
		\end{enumerate}
	\end{thmx}
	\begin{remark}
		Let $L$ be a segment in $(-1,1)^2$. It is easy to construct a map $u\in \overline{H\cap W_{\id}^{1,p}((-1,1)^2, \er^2)}$ for any $p\in[1,\infty]$ such that $u$ maps $L$ onto itself 3 times with $u$ parametrizing $L$  in some middle third in the opposite direction to how it parametrizes $L$ on its end thirds. Then the image of $L$ in any injective continuous approximation of $u$ will either be `S' shaped or `Z' shaped (see Figure~\ref{FiguroOne}). In the case that $p\geq 2$ and $u$ is necessarily continuous one can use Theorem~\ref{Young} to show that the geometry of the image of the segment  in every sufficiently accurate injective approximation $u$ is always of the type (i.e. always `S' or always `Z'). A key step and a very technical one in proving \cite[Theorem~C]{DPP} was showing that their maps have a unique order for separating curves that have the same image in $u$ (see \cite[Proposition~4.19]{DPP} and preparatory lemmas). Since the maps they consider are less general than maps in $\overline{H\cap W_{\id}^{1,p}((-1,1)^2, \er^2)}$ it was not clear that such a unique order must necessarily hold for the class of limits. A corollary of the fact that $a)$ is equivalent with $b)$ in Theorem~\ref{MainTheorem}, is that given any curve $\phi \subset \{u = g_{\delta}\}$ then the uniqueness of the geometry of approximations of $g_{\delta}$ implies the same uniqueness for $u$ and this is clearly independent of the choice of sufficiently small positive $\delta$. This observation is the key to proving the equivalence of $c)$ with $a)$. This observation is only made deeper with respect to Theorem~\ref{SpanishInquisition}.
	\end{remark}
	\begin{figure}[h]
		\begin{tikzpicture}[line cap=round,line join=round,>=triangle 45,x=1.0cm,y=0.5cm]
			\clip(0.5,1.5) rectangle (9,6);
			\draw [line width=1.pt, color=red] (1.,3.)-- (8.,3.);
			\draw [line width=1.pt, color=red] (8.,3.)-- (1.,2.);
			\draw [line width=1.pt, color=red] (1.,2.)-- (8.,2.);
			\draw [line width=1.pt,color=blue] (1.,4.)-- (8.,4.);
			\draw [line width=1.pt,color=blue] (8.,4.)-- (1.,5.);
			\draw [line width=1.pt, color=blue] (1.,5.)-- (8.,5.);
		\end{tikzpicture}
		\caption{If a mapping is not injective then there may be many ways to approximate its restriction to a curve by injective maps which may have very different geometry. Here we see an `S' shape in blue and a `Z' shape in red, both of which could be the image of an injective approximation of a map that maps a segment onto another 3 times.}\label{FiguroOne}
	\end{figure}
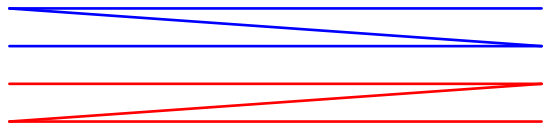
	
	We consider it appropriate to demonstrate the utility of Theorem~\ref{MainTheorem} and especially the QM condition. In order to do this let us demonstrate that it leads to a much simpler proof of \cite[Theorem C]{DPP}, i.e. let us prove that a $W^{1,p}_{\id}((-1,1)^2,\er^2)$ weakly monotone map satisfies the QM condition.
	
		\begin{figure}[h]
			\begin{tikzpicture}[line cap=round,line join=round,>=triangle 45,x=2.5cm,y=2.5cm]
			\clip(0.8,0.8) rectangle (6.2,3.2);
			\fill[line width=0.8pt,fill=black!10] (4.,3.) -- (4.,1.) -- (6.,1.) -- (5.8,1.2) -- (4.2,1.2) -- (4.2,2.8) -- cycle;
			\fill[line width=0.8pt,fill=black!10] (4.,3.) -- (6.,3.) -- (6.,1.) -- (5.8,1.2) -- (5.8,2.8) -- (4.2,2.8) -- cycle;
			\fill[line width=0.8pt,fill=black!10] (3.,3.) -- (3.,1.) -- (1.,1.) -- (1.2,1.2) -- (2.8,1.2) -- (2.8,2.8) -- cycle;
			\fill[line width=0.8pt,fill=black!10] (3.,3.) -- (1.,3.) -- (1.,1.) -- (1.2,1.2) -- (1.2,2.8) -- (2.8,2.8) -- cycle;
			\fill[line width=0.8pt,color=ffqqqq,fill=ffqqqq!20] (1.0795719622258892,2.3417235271215615) -- (1.0836934592638245,2.92285460947043) -- (2.855937185575979,2.9393405976221714) -- (2.888909161879461,2.284022568590468) -- (2.060488257254477,1.946059811479778) -- cycle;
			\draw [line width=0.8pt] (4.,3.)-- (4.,1.);
			\draw [line width=0.8pt] (4.,1.)-- (6.,1.);
			\draw [line width=0.8pt] (4.,3.)-- (6.,3.);
			\draw [line width=0.8pt] (6.,3.)-- (6.,1.);
			\draw [line width=0.8pt] (3.,3.)-- (3.,1.);
			\draw [line width=0.8pt] (3.,1.)-- (1.,1.);
			\draw [line width=0.8pt] (3.,3.)-- (1.,3.);
			\draw [line width=0.8pt] (1.,3.)-- (1.,1.);
			\draw [line width=1.5pt,color=qqqqff] (4.112993782146228,2.527190893828647)-- (4.455078036294854,2.098555201883382);
			\draw [line width=1.5pt,color=qqffqq] (5.4813307987407285,2.082069213731641)-- (5.9388169699515405,2.473611432335489);
			\draw [line width=0.8pt] (4.112993782146228,2.527190893828647)-- (4.1053733289327425,2.8946266710672575);
			\draw [line width=0.8pt] (4.1053733289327425,2.8946266710672575)-- (5.881116011420447,2.9187331124324953);
			\draw [line width=0.8pt] (5.881116011420447,2.9187331124324953)-- (5.9388169699515405,2.473611432335489);
			\draw [line width=1.5pt,color=qqqqff] (1.0795719622258892,2.3417235271215615)-- (2.060488257254477,1.946059811479778);
			\draw [line width=1.5pt,color=qqffqq] (2.060488257254477,1.946059811479778)-- (2.888909161879461,2.284022568590468);
			\draw [line width=0.8pt] (1.0795719622258892,2.3417235271215615)-- (1.0836934592638245,2.92285460947043);
			\draw [line width=0.8pt] (1.0836934592638245,2.92285460947043)-- (2.855937185575979,2.9393405976221714);
			\draw [line width=0.8pt] (2.855937185575979,2.9393405976221714)-- (2.888909161879461,2.284022568590468);
			\draw (3.4,2.2) node[anchor=north west] {$u$};
			\draw [->,line width=0.8pt] (3.2,2.) -- (3.8,2.);
			\begin{scriptsize}
				\draw [fill=black] (2.060488257254477,1.946059811479778) circle (2.5pt);
			\end{scriptsize}
		\end{tikzpicture}
		\caption{Let, for example $C_1$ be the blue set on the right and $C_2$ be the green set on the right. Let $\gamma$ be the black curve in the shaded area by the boundary where $u$ is identity. If there is some point $(x,y)\in u^{-1}(C_1) \cap u^{-1}(C_2)$ then necessarily $u^{-1}(C_1\cup C_2 \cup \gamma)$ must disconnect the plane (shaded red component).}\label{fig:Spain}
		\end{figure}
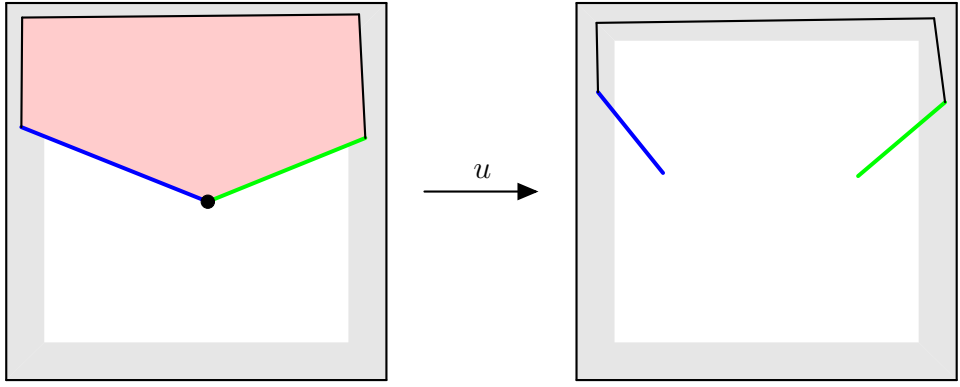
	
	\begin{proof}[Proof that $u$ satisfies WM $\Rightarrow$ $u\in \overline{H\cap W^{1,p}_{\id}((-1,1)^2, \er^2)}$]
		The reader may consult Figure~\ref{fig:Spain}. Let $C_1, C_2$ be a pair of disjoint closed sets such that $(-1,1)^2\setminus (C_1\cup C_2)$ is connected. Because $u\in W^{1,p}_{\id}((-1,1)^2,\er^2)$ we may assume that there is a $\xi>0$ such that $u(x,y) = (x,y)$ for all $\dist((x,y),\partial(-1,1)^2)<\xi$. Without loss of generality we may assume that $C_i \cap \{(x,y) : \dist((x,y),\partial(-1,1)^2)<\xi\} \neq \emptyset$. Then we can find a curve $\gamma$ inside $\{(x,y) : \dist((x,y),\partial(-1,1)^2)<\xi\}$ intersecting $C_1$ and $C_2$ such that $C_3:=C_1\cup C_2 \cup \gamma$ is closed connected and does not disconnect $(-1,1)^2$. Then $u^{-1}(C_3)$ is closed connected and does not disconnect $(-1,1)^2$. If $u^{-1}(C_1)\cap u^{-1}(C_2) \neq \emptyset$, then there is a bounded component of $\er^2\setminus u^{-1}(C_3)$ which is a contradiction. Then $\dist(u^{-1}(C_1),u^{-1}(C_2))>0$. Each component $G$ of $U_{\delta}$ has $\diam G\leq \delta$ and therefore, for $\delta\in (0, \tfrac{1}{2}\dist(u^{-1}(C_1),u^{-1}(C_2)))$, $G$ intersects at most one of $u^{-1}(C_i)$. Then define $V_i = u^{-1}(C_i) + B(0,\rho) \cup \bigcup_{G\cap u^{-1}(C_i)\neq \emptyset} G$ for sufficiently small $\rho$ since $u$ is uniformly continuous on $[-1,1]^2\setminus U_{\delta}$.
	\end{proof}
	The astute reader may have noticed the surprising implication of the above proof.
	\begin{thmx}\label{SpanishInquisition}
		Let $u \in W^{1,p}_{\id}((-1,1)^2,\er^2)$. Then $u$ is weakly monotone if and only if it is monotone.
	\end{thmx}
	\begin{proof}
		It is easy to prove that a monotone map is weakly monotone using Theorem~\ref{Young}.
		
		Let $C_1,C_2$ be as in the above proof then $u^{-1}(C_1)\cap u^{-1}(C_2) = \emptyset$. Thus we have that $u_{\operatorname{multi}}(x,y)$ (see Definiton~\ref{defIMG}) is a singleton for each $(x,y)$. This implies that $u$ is continuous at $(x,y)$. Since this holds for all $(x,y)$ we have that $u$ is continuous on $[-1,1]^2$ and therefore monotone.
	\end{proof}
	
\subsection{Plan for the paper and outline of the proof of Theorem~\ref{MainTheorem}}

In Section~\ref{Preliminaries} we give definitions we need, recall some known results and prove some preparatory lemmas.

We split the proof of Theorem~\ref{MainTheorem} into three parts, $a)$ is equivalent with $b)$ is proved in Theorem~\ref{QuasiMonotne},  $a)$ is equivalent with $c)$ is proved in Theorem~\ref{ThreeCurveChar} and  $a)$ is equivalent with $d)$ is proved in Theorem~\ref{PointwiseChar}.

$b)\Rightarrow a)$
If $U_{\delta}$ is an open set with small projection onto coordinate axes, then we can find a grid of horizontal and vertical lines avoiding it. If the restriction of $u$ on a grid is equal to the restriction of a monotone map on the grid, then it satisfies the no-crossing condition on this grid. This allows us to construct a diffeomorphism $v_{\delta}$ with $\int|Dv_{\delta}|^p$ bounded by $\int|Du|^p$. By sending $\delta \to 0$ and  so refining the grid we converge weakly to $u$ in $W^{1,p}$. 

$a)\Rightarrow b)$
The key to is to carefully choose a sequence of grids $\Gamma_k$, each of which are subsets of $[-1,1]^2\setminus U_{\delta}$, a set on which $u$ is uniformly continuous. Then by careful extension of an injective uniform approximation of $u$ on $\Gamma_k$ and an application of the Arzela-Ascoli theorem we find a monotone map $h_k$ on $[-1,1]^2$ equal to $u$ on $\Gamma_k$. On $[-1,1]^2\setminus U_{\delta}$ the grids $\Gamma_k$ become very fine meaning  $h_k\sto u$ on $[-1,1]^2\setminus U_{\delta}$. On the rest of the set, we are careful to make sure that the sequence is almost constant. This is enough to make sure that $h_k\sto$ on $[-1,1]^2$. The limit $g_{\delta}$ is monotone and is equal to $u$ on $[-1,1]^2\setminus U_{\delta}$.

$b)\Rightarrow c)$
We use $b)$ to find a homeomorphism $h:(-1,1)^2\to (-1,1)^2$ uniformly very close to $u$ except in $U_{\delta}$, a set not seen by $\phi([0,1])$. Then we modify $h^{-1}\circ\psi([0,1])$ slightly to get exactly the $\gamma$ we want.

$c) \Rightarrow a)$
Firstly we show that if we have a stronger property for grids $\Gamma$ and $\G$, i.e. there is a Lipschitz injective $g:\G \to \er^2$  such that $g(\G) \cap \Gamma = u^{-1}(\G)\cap \Gamma$, then the equivalence follows easily. Then it suffices to prove the stronger property follows from the weaker. We make $\phi([0,1])$ that looks like a grid $\Gamma$ in the preimage but has some `holes' by the vertexes (see Figure~\ref{fig:wavy}). We make $\psi([0,1/2])$ that looks like an arrival grid and $\psi([1/2,1])$ goes around the images of the vertexes of the grid in the preimage (see Figure~\ref{fig:arcs}). The second part of $\psi$ allows us to assume that $\gamma([0,1])$ does not `go through' the holes by the vertexes of $\Gamma$ and we can replace the image of $\phi$ with the grid $\Gamma$. By careful analysis of the curve $\gamma$ we show it is possible `plug the holes' near preimages of vertexes of $\G$ without intersecting the rest of $\gamma([0,1])$ or $\Gamma$. This allows us to construct a new injective Lipschitz $g$ from $\G$ intersecting $\Gamma$ exactly in $u^{-1}(\G)\cap \Gamma$, which is precisely the stronger property.

$b)\Rightarrow d)$
This is an easy exercise of Theorem~\ref{Young}.

$d) \Rightarrow c)$
We show that the QM condition implies the NCL condition. We may assume that $\phi([0,1])$ disconnects $[-1,1]^2$. Using the NCL condition, for every $(w,z)\in [-1,1]^2\setminus u(\phi([0,1]))$ we uniquely determine a `side' of $u(\phi([0,1]))$ which we say $(w,z)$ lies on. The QM condition allows us to construct parts of the curve $\gamma$ as curves in domains $V_i$ which contain $\psi(I_i)$, where $I_i\subset[0,1]$ are intervals. By a careful adjustment we make sure they intersect $\phi([0,1])$ exactly once exactly where we want them. Since $V_i$ are pairwise disjoint so are these curves. It remains to carefully connect the curves so that the final curve is injective.  

\section{Preliminaries}\label{Preliminaries}
Throughout the paper we call the projections $\pi_1(x,y)  = x, \pi_2(x,y)  = y$. By $\mathcal{L}^n$ we denote the standard $n$-dimensional Lebesgue measure and by $\H^s$ the $s$-dimensional Hausdorf measure.

\subsection{Special Sobolev spaces}\label{SSS}
Let $u\in W^{1,p}((-1,1)^2, \er^2)$. We say that $u \in W^{1,p}_{\id}((-1,1)^2, \er^2)$ if $[u(x) - x] \in W^{1,p}_0((-1,1)^2, \er^2)$.

Let $\Gamma = \bigcup_{i=1}^{N} \{x_i\}\times [-1,1]\ \cup\ \bigcup_{j=1}^{N} [-1,1] \times \{y_j\}$ be the finite union of horizontal and vertical lines. Throughout the course of this paper, we work with function spaces of the type $W^{1,p}(\Gamma)$ but since this is not standard we explain explicitly what we mean by this notation. We say that $u\in W^{1,p}(\Gamma)$ if $u$ is absolutely continuous on every segment $\{x_i\}\times[-1,1]$, $[-1,1]\times\{y_j\}$ and
\begin{equation}\label{Modulus}
	\sum_{j=1}^N \int_{-1}^{1} |\partial_1 u(x, y_i) |^p\ dx + \sum_{i=1}^N  \int_{-1}^{1} |\partial_2 u(x_i,y) |^p\ dy < \infty.
\end{equation}
We can define a norm on $W^{1,p}(\Gamma)$ as the $\tfrac{1}{p}$ power of $\int|u|^p$ plus the expression in \eqref{Modulus}. The space shares many well-known properties of $W^{1,p}((a,b))$ especially concepts such as weak convergence, which implies uniform convergence. For shorthand we write $D_{\tau} u$ to denote the function $\partial_2u$ on $\bigcup_{i=1}^{N}\{x_i\}\times[-1,1]$ and $\partial_1u$ on $\bigcup_{j=1}^{N}[-1,1]\times\{y_j\}$.

It is well known for any $u\in W^{1,p}((-1,1)^2)$ (for the correct representative of $u$) that $u_{\rceil \Gamma} \in W^{1,p}(\Gamma)$ for almost every choice of $x_i$ and almost every choice of $y_j$. Specifically, for $u\in W^{1,p}((-1,1)^2)$ we denote the sets $G_x = G_x(u)$ and $G_y = G_y(u)$ as
\begin{equation}\label{defGx}
	G_x = \{x\in (-1,1); \  u(x,\cdot) \in  W^{1,p}(-1,1) \}
\end{equation}
and
\begin{equation}\label{defGy}
	G_y = \{y\in (-1,1);  \ u(\cdot,y) \in  W^{1,p}(-1,1)  \} .
\end{equation}
For mappings into $\er^2$ the definitions are extended in the natural way by components.

We often work with rectangles like $K = [x,\tilde{x}]\times[y,\tilde{y}]$. When we write $\lambda K$ for $\lambda >0$ we mean the set
$$
\lambda K = \Big\{(a,b)\in \er^2: \ \Big(\tfrac{x+ \tilde{x}}{2}  + \tfrac{2a - x- \tilde{x}}{2\lambda}, \tfrac{y+ \tilde{y}}{2}  +\tfrac{2b - y- \tilde{y}}{2\lambda}\Big) \in K\Big\}.
$$

During the paper we often use sets that we call suitable $k$-grids for $u$. Let us now define what such a set is.
\begin{definition}\label{GAP}
	Let $p\in[1,\infty)$ and let $u\in W^{1,p}((-1,1)^2)$. Let $k\in \en$, let the sets $X_k = \{x_{k,i}\in [-1,1]; i=0,1,\dots, M_k^x\}$, $Y_k = \{y_{k,j}\in [-1,1]; \ j=0,1,\dots, M_k^y\}$ satisfy
	\begin{enumerate}
		\item  $x_{k,0}  = y_{k,0}=-1$, and $x_{k,M_k^x} =y_{k,M_k^y} = 1$
		\item $x_{k,i} - x_{k,i-1}, y_{k,j} - y_{k,j-1}  \in (2^{-k}, 2^{2-k})$ for $i\in \{1,2,\dots, M_{k}^x\}$ and $j\in \{1,2,\dots, M_{k}^y\}$.
	\end{enumerate}
	Then we call
	$$
	\Gamma_k = \bigcup_{i=0}^{M_{k}^x} \{x_{k,i}\}\times [-1,1]\ \cup\ \bigcup_{j=0}^{M_k^y} [-1,1] \times \{y_{k,j}\},
	$$
	a $k$-grid. We say $\Gamma_k$ is a suitable $k$-grid for $u$ if
	\begin{enumerate}
		\item $u_{\rceil \Gamma_k}\in W^{1,p}(\Gamma_k)$,
		\item there exists a $\hat{C}>0$ such that for all $0\leq j\leq M_k^y$ and all $q\in [1,p]$
		\begin{equation}\label{CkGrid1}
			\int_{-1}^{1} |\partial_1u(x,y_{j,k})|^q \ dx \leq \hat{C}2^{k}  \int_{y_{k,j}-2^{-k-4}}^{y_{k,j}+2^{-k-4}}\int_{-1}^1 |Du(x,y)|^q \ dx dy
		\end{equation}
		and for all $0\leq i\leq M_k^x$
		\begin{equation}\label{CkGrid2}
			\int_{-1}^{1} |\partial_2u(x_{k,i},y)|^q \ dy \leq \hat{C}2^{k} \int_{x_{k,i}-2^{-k-4}}^{x_{k,i}+2^{-k-4}}\int_{-1}^1 |Du(x,y)|^q \ dy dx.
		\end{equation} 
	\end{enumerate}
	If $\Gamma_1 \subset \Gamma_2, \subset \dots $ is an increasing collection of suitable $k$-grids for $u$ for each $k\in \en$ with $\hat{C}$ independent of $k$ then we say that that $\mathfrak{G} = \{\Gamma_k\}_{k\in \en}$ is a suitable grid system for $u$.
\end{definition}

\begin{prop}[The existence of many suitable grid systems]\label{GridExist}
	Let $p\in [1,\infty)$, let $u\in W^{1,p}((-1,1)^2,\er^2)$ and let $S_u\subset (-1,1)^2$ with $\H^1(S_u) = 0$ be such that $u_{\rceil [-1,1]^2\setminus S_u}$ is continuous on $[-1,1]^2\setminus S_u$. There exists $\mathfrak{G} =\{\Gamma_k\}_{k\in \en}$ a suitable grid system for $u$ satisfying \eqref{CkGrid1} and \eqref{CkGrid2} with the constant $\hat{C} = 32$.
	
	Moreover for each $k\in \en$ each $i\in\{0,1,\dots, M_k^x\}$ and each $j\in\{0,1,\dots, M_k^y\}$ such that $x_{k,i}\notin X_{k-1}$ (respectively $y_{k,j}\notin Y_{k-1}$) there is a set $E_{k,i}^1\subset \er$ (respectively $E_{k,j}^2\subset \er$) with $\mathcal{L}^1(E_{k,i}^1), \mathcal{L}^1(E_{k,j}^2) \geq 2^{-k-3}$ and every choice of $x_{k,i} \in E_{k,i}^1$, $y_{k,j} \in E_{k,j}^2$ constructs a suitable grid system for $u$ satisfying \eqref{CkGrid1} and \eqref{CkGrid2} with the constant $\hat{C} = 32$.
\end{prop}
\begin{proof}
	Firstly, assume that $u$ is a scalar function. By an interpolation argument it suffices to prove \eqref{CkGrid1} and \eqref{CkGrid2} for $q=1$ and $q=p$. Let $\Gamma_0 = \{-1\}\times[-1,1] \cup \{1\}\times[-1,1] \cup [-1,1]\times\{-1\} \cup [-1,1]\times\{1\}$, then $\Gamma_0$ is a suitable $0$-grid for $u$. Now we continue by induction and assume that we have $\Gamma_k$ a suitable $k$-grid for $u$ and construct $\Gamma_{k+1}$. We show how to construct the sets $E_{k+1,i}^1$, the construction of $E_{k+1,j}^2$ is similar. Any choice of $x_{k+1,i} \in E_{k+1,i}^1$, respectively $y_{k+1,j} \in E_{k+1,j}^2$ yields our $\Gamma_{k+1}$.

	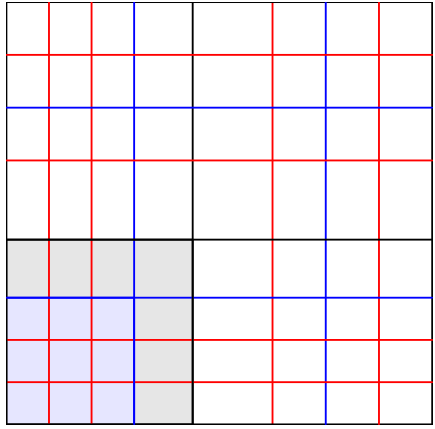
\begin{figure}
		\begin{tikzpicture}[line cap=round,line join=round,>=triangle 45,x=0.7cm,y=0.7cm]
			\clip(0,0) rectangle (8,8);
			\fill[line width=0.7pt,fill=black!10] (0.,3.5) -- (3.5,3.5) -- (3.5,0.) -- (0.,0.) -- cycle;
			\fill[line width=0.7pt,color=qqqqff,fill=qqqqff!10] (0.,2.4) -- (2.4,2.4) -- (2.4,0.) -- (0.,0.) -- cycle;
			\draw [line width=0.7pt,color=ffqqqq] (5.,8.)-- (5.,0.);
			\draw [line width=0.7pt] (8.,8.)-- (8.,0.);
			\draw [line width=0.7pt] (8.,0.)-- (0.,0.);
			\draw [line width=0.7pt] (0.,0.)-- (0.,8.);
			\draw [line width=0.7pt] (0.,8.)-- (8.,8.);
			\draw [line width=0.7pt] (3.5,8.)-- (3.5,0.);
			\draw [line width=0.7pt,color=qqqqff] (6.,8.)-- (6.,0.);
			\draw [line width=0.7pt,color=ffqqqq] (7.,8.)-- (7.,0.);
			\draw [line width=0.7pt,color=qqqqff] (2.4,8.)-- (2.4,0.);
			\draw [line width=0.7pt,color=ffqqqq] (0.8,8.)-- (0.8,0.);
			\draw [line width=0.7pt,color=ffqqqq] (1.6,8.)-- (1.6,0.);
			\draw [line width=0.7pt] (0.,3.5)-- (8.,3.5);
			\draw [line width=0.7pt,color=qqqqff] (0.,6.)-- (8.,6.);
			\draw [line width=0.7pt,color=qqqqff] (0.,2.4)-- (8.,2.4);
			\draw [line width=0.7pt,color=ffqqqq] (0.,0.8)-- (8.,0.8);
			\draw [line width=0.7pt,color=ffqqqq] (0.,1.6)-- (8.,1.6);
			\draw [line width=0.7pt,color=ffqqqq] (0.,5.)-- (8.,5.);
			\draw [line width=0.7pt,color=ffqqqq] (0.,7.)-- (8.,7.);
			\draw [line width=0.7pt] (0.,3.5)-- (3.5,3.5);
			\draw [line width=0.7pt] (3.5,3.5)-- (3.5,0.);
			\draw [line width=0.7pt] (3.5,0.)-- (0.,0.);
			\draw [line width=0.7pt] (0.,0.)-- (0.,3.5);
			\draw [line width=0.7pt,color=qqqqff] (0.,2.4)-- (2.4,2.4);
			\draw [line width=0.7pt,color=qqqqff] (2.4,2.4)-- (2.4,0.);
		\end{tikzpicture}
		\caption{We construct our grid $\Gamma_1$ by splitting the square by the black lines. The grid $\Gamma_2$ is made by including the blue lines. The grid $\Gamma_3$ is made by adding red lines. We make sure that neighboring lines in $\Gamma_k$ have distance close to $2^{-k}$ by adding the correct number of lines between neighbors, either 0, 1 or 2 lines are added.}\label{fig:MakeAGrid}
	\end{figure}

	How we construct the sets $E_{k+1,\tilde{i}}^1$ depends on the value of $x_{k,i+1}-x_{k,i}$. The reader should consult Figure~\ref{fig:MakeAGrid}. If $x_{k,i+1}-x_{k,i}\in (2^{-k}, 2^{1-k}]$ then we do not choose any new $x_{k+1,\tilde{i}} \in ( x_{k,i+1},x_{k,i})$. If $x_{k,i+1}-x_{k,i}\in (2^{1-k}, \tfrac{3}{2} 2^{1-k}]$ then we choose one new $x_{k+1,\tilde{i}} \in ( x_{k,i+1},x_{k,i})$. In this case call 
	$$
	F_{k+1,\tilde{i}}^1 =  \Big(\frac{x_{k,i+1}+x_{k,i}}{2} - 2^{-k-3},  \frac{x_{k,i+1}+x_{k,i}}{2} + 2^{-k-3}\Big) \cap G_x \setminus \pi_1(S_u)
	$$
	and then $\mathcal{L}^1(F_{\tilde{i}, k+1}^1 )  = 2^{-k-2}$. For all $x\in F_{\tilde{i}, k+1}^1$ it holds that
	$$
	[x_{k,i+1} - x], [x- x_{k,i}] \in (2^{-k}-2^{-k-3}, \tfrac{3}{2} 2^{-k} +2^{-k-3} ) \subset (2^{-k-1}, 2^{1-k}).
	$$
	If $x_{k,i+1}-x_{k,i}\in (\tfrac{3}{2} 2^{1-k}, 2^{2-k})$ then we choose two new points $x_{k+1,\tilde{i}} x_{k+1,\tilde{i}+1} \in ( x_{k,i+1},x_{k,i})$. In this case call 
	$$
	F_{k+1,\tilde{i}}^1 =  \Big(\frac{x_{k,i+1}+2x_{k,i}}{3} - 2^{-k-3},  \frac{x_{k,i+1}+2x_{k,i}}{3} + 2^{-k-3}\Big) \cap G_x \setminus \pi_1(S_u)
	$$
	and
	$$
	F_{k+1,\tilde{i}+1}^1 =  \Big(\frac{2x_{k,i+1}+x_{k,i}}{3} - 2^{-k-3},  \frac{2x_{k,i+1}+x_{k,i}}{3} + 2^{-k-3}\Big) \cap G_x \setminus \pi_1(S_u)
	$$
	and then $\mathcal{L}^1(F_{k+1,\tilde{i}}^1 ) = \mathcal{L}^1(F_{k+1,\tilde{i+1}}^1 )= 2^{-k-2}$. For all $s\in F_{k+1,\tilde{i}}^1$ and all $t\in F_{k+1,\tilde{i}}^1$ it holds that 
	$$
	[x_{k,i+1} - t], [t-s], [s- x_{k,i}] \in (2^{-k} -2^{-k-2},  \tfrac{2}{3}2^{1-k} + 2^{-k-2}) \subset (2^{-k-1}, 2^{1-k}).
	$$

	In both of the cases when we add some new points $x_{k+1,\tilde{i}}$ we have
	$$
	\oint_{F_{k,\tilde{i}}^1} \int_{-1}^1|\partial_2 u(x,y)| \ dy dx \leq  2^{k+2}\int_{x_{k,i}}^{x_{k,i+1}} \int_{-1}^1|\partial_2 u(x,y)| \ dy dx
	$$
	and
	$$
	\oint_{F_{k,\tilde{i}}^1} \int_{-1}^1|\partial_2 u(x,y)|^p \ dy dx \leq  2^{k+2}\int_{x_{k,i}}^{x_{k,i+1}} \int_{-1}^1|\partial_2 u(x,y)|^p \ dy dx
	$$
	Arguing that $\mathcal{L}^1(\{y\in I: f(y) > \lambda \oint_I f\}) \leq \lambda^{-1}\mathcal{L}^1(I)$ we find that the set of $x_{k+1,\tilde{i}} \in F_{k,\tilde{i}}^1$ satisfying
	$$
	\int_{-1}^1|\partial_1 u(x_{k,\tilde{i}},y)| \ dx \leq 32\cdot 2^{k}\int_{x_{k,i}}^{x_{k,i+1}} \int_{-1}^1|\partial_2 u(x,y)| \ dy dx
	$$
	and
	$$
	\int_{-1}^1|\partial_1 u(x,y_j)|^p \ dx \leq 32\cdot 2^{k}\int_{x_{k,i}}^{x_{k,i+1}} \int_{-1}^1|\partial_2 u(x,y)|^p \ dy dx
	$$
	each has measure at least $\tfrac{7}{8}2^{-k-2}$. Therefore, the measure of their intersection is at least $\tfrac{3}{4}\cdot2^{-k-2}$. We repeat the above estimates on both coordinate functions and the intersection of the sets on which the estimates hold for both coordinate functions simultaneously is at least $2^{k-3}>0$, that is the set $E_{k+1,\tilde{i}}^1$. The proof for $E_{k+1,\tilde{j}}^2$ and $\partial_2 u$ is similar. The choice of $E_{k+1,\tilde{i}}^1\subset G_x \setminus \pi_1(S_u)$ and $E_{k+1,\tilde{j}} \subset  G_y \setminus \pi_2(S_u)$ together with the estimates above guarantee that $u_{\rceil \Gamma_k} \in W^{1,p}(\Gamma_k)$ and that $\Gamma\cap S_u = \emptyset$.
\end{proof}

\subsection{Linearization of a map on a grid}
The main result we want to formulate here is Proposition~\ref{PiecewiseLinearization}, which is a summary of an approach used in \cite{DPP} and has been useful since. We include the details for the convenience of the reader even though they are not new. In the proof we use a version of the following area formula. Let $\phi:(0,1)\to\er$ be absolutely continuous and $A\subset (0,1)$ be Borel set, then (see e.g. \cite[Theorem 3.65]{Le} where $\psi=\chi_A$)
$$
\int_A|\phi'(t)|dt = \int_{\er} \H^0(\{t\in A:\ \phi(t) = z\}) \, dz. 
$$
The following corollary for $\phi:\Gamma_k\to\er$ absolutely continuous on a $k$-grid $\Gamma_k$ and $A\subset \Gamma_k$ follows simply
\begin{equation}\label{area}
	\int_{A}|D_{\tau}\phi|d\H^1 = \int_{\er} \H^0(\{t\in A:\ \phi(t) = z\}) \, dz. 
\end{equation}

We define the concept of the generalized segment, already introduced in \cite{DPP}, which is useful for proving Proposition~\ref{PiecewiseLinearization}.

\begin{definition}[generalized segments]\label{generalized segments}
	Let $R >0$, let  $-R= w_0 < w_1 < \dots < w_M = R,$  $-R= z_0 < z_1 < \dots < z_M = R,$ and let $\G = \bigcup_{n=0}^M \{w_n\}\times [-R,R] \cup [-R,R]\times\{z_n\}$. Let $R_{n,m} = [w_n, w_{n+1}]\times[z_m, z_{m+1}]$ be a rectangle of the grid $\G$. Let $X\neq Y$ and $X,Y\in\partial R \subset \G$. Given $\xi >0$ a small parameter, the \emph{generalized segment} $[XY]_{\xi}$ between $X$ and $Y$ in $R_{n,m}$ with parameter $\xi$ is defined as the standard segment $XY$ if the two points are not on the same side of $\partial R_{n,m}$; otherwise, $[XY]$ is the union of the two segments $XM$ and $MB$, where $M$ is the point inside $R_{n,m}$ whose distance from the side containing $X$ and $Y$ is $\xi|X-Y|/2$ and the projection of $M$ on the segment $XY$ is the mid-point of $XY$.
\end{definition}

The following claim from \cite{CKR} is straightforward.

\begin{prop}\label{EverythingIDo}
	Let $R\subset \er^2$ be a rectangle and let $a,b \in \partial R$. Let $[ab]_{\xi}$ be the generalized segment from $a$ to $b$ in $R$ with parameter $\xi >0$. Then
	$$
	\H^1([ab]) \leq (1+\xi)|b-a|.
	$$
\end{prop}

\begin{definition}
	Let $K\in \en$ and $\gamma_i:[0,1] \to \er^2, i=1,\dots, K$ be such that
	\begin{enumerate}
		\item $\gamma_i$ is Lipschitz continuous and injective,
		\item $\gamma_i([0,1]) \cap \gamma_j([0,1])$ contains at most one point for every $i\neq j$,
		\item if $\gamma_i(t) = \gamma_j(s)$ then both of the derivatives exist and are linearly independent,
		\item $\gamma_i([0,1]) \cap \gamma_j([0,1]) \cap \gamma_n([0,1]) = \emptyset$ if $i\neq j \neq n \neq i$.
	\end{enumerate}
	Then we call $\{\gamma_i\}_{i=1}^K$ a general grid.
\end{definition}

\begin{definition}
	Let $\{\sigma_i\}$ be a general grid, call $\Gamma = \bigcup_{i=1}^M\sigma([0,1])$ and let $\phi :\Gamma \to  \er^2$. We say that an injective Lipschitz continuous curve $g:A \to \er^2$, where $A = [0,1]$ or $A$ is the unit circle is a good arrival curve for $\phi$ if
	\begin{itemize}
		\item $F = \phi^{-1}(g(A))$ is finite and $F\cap \sigma_i([0,1])\cap \sigma_j([0,1]) =\emptyset$ for each $1\leq i<j\leq K$,
		\item for all $s\in A$ such that $g(s)\in \phi(F)$ there exists $g'(s)$ (or $D_{\tau}g(s)$ in case $A$ is the circle) 
		\item for every $(x,y)\in F$ there exists $D_{\tau}\phi(x,y)$ (here we implicitly understand that $\sigma_i'(t)\neq 0$ for the unique $t\in [0,1]$ and the unique $1\leq i\leq K$ such that $\sigma_i(t) = (x,y)$)
		\item the vectors $D_{\tau}\phi(x,y)$ and $g'(s)$ are linearly independent whenever $\phi(x,y) = g(s)$.
	\end{itemize}
	Further, for $M\in \en$ let $\{\gamma_j\}_{j=1}^M$ be a general grid such that $\gamma_j:[0,1] \to \er^2$ is a good arrival curve for $\phi$ for each $j=1,\dots M$ and also $\gamma_{i}([0,1])\cap \gamma_j([0,1])\cap \varphi(\Gamma) = \emptyset$, then we call $\{\gamma_j\}$ a good arrival grid for $\phi$.
\end{definition}

\begin{remark}\label{Convention}
	All the good arrival grids we actually use in this paper are collections of curves whose image is either a vertical or a horizontal segment. In this case we will assume that the curves have constant-speed parametrization, and it is enough to specify the image of the curves as in \eqref{SegmentsGrid}.
\end{remark}

The authors of \cite{DPP} proved that good arrival grids exist for suitable restrictions of $W^{1,p}$ maps. The following is a version of the proof given there for the convenience of the reader.

\begin{prop}\label{ArrivalGrids}
	Let $p\in [1,\infty)$, let $\Gamma_k$ be a $k$-grid, let $\phi \in W^{1,p}(\Gamma_k, \er^2)$. For every $M\in \en$ and for $\mathcal{L}^M$-almost every $(w_1, \dots, w_M), (z_1, \dots, z_M)  \in \er^M$ the set
	\begin{equation}\label{SegmentsGrid}
		\G= \bigcup_{n=0}^M \{w_n\}\times \er \cup \er\times\{z_n\}
	\end{equation}
	is a good arrival grid for $\phi$ (under the convention of Remark~\ref{Convention}).
\end{prop}

\begin{proof}
	Recall that we call the projections $\pi_1(x,y)  = x, \pi_2(x,y)  = y$. Since $\phi\in W^{1,1}(\Gamma_k, \er^2)$, it is absolutely continuous and so too is $\pi_i\circ \phi$ for $i=1,2$. Therefore, we can apply the area formula from \eqref{area} for $\pi_i\circ \phi$ for $i=1,2$. Denote
	$$
	\tilde{N}_i = \{(x,y)\in \Gamma_k; \text{ the derivative } D_{\tau}\phi(x,y)\text{ does not exist}\} \cup \{t\in \Gamma_k; D_{\tau}\pi_i\circ\phi(x,y) = 0\}.
	$$
	We can clearly fix a Borel set ${N_i}\supset \tilde{N}_i$ and $\mathcal{H}^1(N_i\setminus \tilde{N}_i) = 0$. We have $|D_{\tau}\pi_i\circ\phi| = 0$ $\H^1$-almost everywhere in $N_i$. From \eqref{area}
	$$
	\begin{aligned}
		\mathcal{L}^1\big(\pi_i\circ\phi(\tilde{N}_i)\big) &\leq\mathcal{L}^1\big(\pi_i\circ\phi({N}_i)\big) \leq \int_{\er} \H^0\big(\{t\in {N}_i:\ \pi_i(\phi(t)) = z\}\big) \, dz \\
		&=\int_{{N}_i}|D_{\tau}\pi_i\circ \phi|\; d\mathcal{H}^1=0.
	\end{aligned}
	$$ 
	Further, since $\phi\in W^{1,1}(\Gamma_k,\er^2)$, the area formula applied to $\pi_i\circ\phi$ again gives
	$$
	\infty > \int_{\Gamma_k}|D_{\tau}\phi|d\mathcal{H}^1 \geq \int_{\Gamma_k}|D_{\tau}\pi_i\circ\phi|d\mathcal{H}^1 = \int_{-\infty}^{\infty} \H^0(\{(x,y):\ \pi_i(\phi(x,y)) = z\}) \, d\mathcal{L}^1(z) 
	$$
	and so
	\begin{equation}\label{Finiteness}
		\H^0\Big(\phi\big(\Gamma_k\big)\cap (\{w\}\times \er)\Big) < \infty \ \text{and} \ \H^0\Big(\phi\big(\Gamma_k\big)\cap \big(\er\times\{z\}\big)\Big) < \infty
	\end{equation}
	for almost every $w,z\in \er$. We choose $-R= w_0 < w_1 < \dots < w_M = R,$ with $w_n \in \er\setminus N_1$ and  $-R= z_0 < z_1 < \dots < z_M = R,$ with $z_m \in \er\setminus N_2$ so that \eqref{Finiteness} holds. Since $\H^1\big(\{\pi_1\circ\phi(x_i,y_j); \ 0\leq i,j\leq 2^k  \}\big) = \H^1\big(\{\pi_2\circ\phi(x_i,y_j); \ 0\leq i,j\leq 2^k  \}\big) = 0$ we may assume that $(w_n,z_m) \in \er^2\setminus \phi(x_i, y_j) $, where $(x_i,y_j)$ are the crossing points of the $k$-grid $\Gamma_k$. By choosing the points $w_n$ first we get a finite set $\phi(\Gamma_k) \cap \{w_n\}\times\er$ and then choosing $z_m$ we may assume that $\phi(\Gamma_k) \cap \{(w_n, z_m)\} = \emptyset$.
\end{proof}

\begin{remark}\label{Alibistic}
	The arguments used above easily extend to the case, where we replace $\G$ with a single circle centered at $(a,b)$ of radius $r>0$. That is, let $\phi\in W^{1,p}(\Gamma_k, \er^2)$ and let $(a,b)\in \er^2$ then for $\mathcal{L}^1$-almost every $r>0$ the curve $g:\partial B(0,1) \to \er^2$, $g(\theta) = (a,b) + r\theta$ is a good arrival curve for $\phi$.
\end{remark}

The following proposition is the application of arrival grids. We hope that by formulating the following proposition we make the utility of the tool more readily accessible.

\begin{prop}\label{PiecewiseLinearization}
	Let $p\in [1,\infty)$, let $\Gamma_k$ be a $k$-grid, let $\phi \in W^{1,p}(\Gamma_k, \er^2)$ and let for every $\delta>0$ there exists an injective and continuous $\phi_{\delta}:\Gamma_k \to \er^2$ such that $|\phi_{\delta}(s) - \phi(s)| < \delta$  for all $s\in \Gamma_k$. Then there exist continuous piecewise linear injective $\gamma_l \sto \phi$ on $\Gamma_k$ such that
	\begin{equation}\label{WeakEstimate}
		\limsup_{l\to \infty}\int_{\Gamma_k}|D_{\tau}\gamma_l|^p \leq \int_{\Gamma_k}|D_{\tau}\phi|^p
	\end{equation}
	and $\{|D_{\tau}\gamma_l|\}_{l\in \en}$ is uniformly integrable on $\Gamma_k$.
\end{prop}
\begin{proof}
	Without loss of generality, we may assume that $\phi(\Gamma_k), \phi_{\delta}(\Gamma_k)\subset (-R,R)^2$ for all $\delta>0$ for some $R>0$. We start with a simple observation that an elementary piecewise linear approximation has all the properties we want except for possibly injectivity. Achieving the injectivity of the approximation forms the larger (second) part of the proof. 
	
	Let $\epsilon_l>0$ and let us assume that we have a decomposition of $\Gamma_k$ into essentially disjoint segments. For each of the  segments $X$ let it hold that $\int_{X}|D_{\tau}\phi|<\epsilon_l$. Let the function $\tilde{\gamma}_l:\Gamma_{k}\to \er^2$ be defined on each segment $X$ as the linear interpolation of the values of $\phi$ at the endpoints of $X$. We have
	$$
	\int_{\Gamma_k}|D_{\tau}\tilde{\gamma}_l|^q = \sum_{X\subset \Gamma_k}\Big|\oint_{X} D_{\tau} \phi\Big|^q\H^1(X) \leq \sum_{X\subset\Gamma_k}\int_{X} |D_{\tau}\phi|^q =\int_{\Gamma_k} |D_{\tau}\phi|^q
	$$
	holds for all $q\in[1,\infty)$. This proves \eqref{WeakEstimate} for $\tilde{\gamma}_l$ including the uniform integrability of $|D_{\tau}\tilde{\gamma}_l|$. Further $|\tilde{\gamma}_l(t) - \phi(t)|<2\epsilon_l$ for all $t\in X$.
	
	Moreover, for any $\gamma_l:\Gamma_k \to \er^2$ which is linear on each of the segments $X = t_1t_2$ and satisfies
	\begin{equation}\label{SizesXX}
		\max\big\{  |\gamma_l(t_1) - \tilde{\gamma}_l(t_1)|, |\gamma_l(t_2) - \tilde{\gamma}_l(t_2)| \big\}|< \epsilon_l\min\big\{1, |\phi(t_{1}) - \phi(t_{2})|\big\}
	\end{equation}
	we have $|\gamma_l(t) - \phi(t)|<4\epsilon_l$ for all $t \in \Gamma_k$ and
	$$
	\int_{\Gamma_k}|D_{\tau}\gamma_l|^p \leq (1+2\epsilon_l)^p\sum_{X\subset \Gamma_k}\Big(\oint_{X} |D_{\tau} \phi|\Big)^p\H^1(X) \leq (1+2\epsilon_l)^p\int_{X} |D_{\tau}\phi|^p
	$$
	thus proving \eqref{WeakEstimate}.
	
	With regards to the above it suffices to show that for each $\epsilon>0$ one can decompose $\Gamma_k$ into finitely many essentially disjoint segments, such that $\int_{X}|D_{\tau}\phi|<\epsilon$ for each of the segments $X$ and that there is a corresponding piecewise linear map $\gamma$ like $\gamma_l$ from the previous paragraph that is injective. In the following we describe how to achieve this.

	Using  Proposition~\ref{ArrivalGrids} we find a good arrival grid for $\phi$, precisely
	$$
	\G= \bigcup_{n=0}^M \{w_n\}\times [-R,R] \cup [-R,R]\times\{z_n\}
	$$
	with $M$ chosen large enough so that $0<w_{n+1}- w_n, z_{m+1}- z_m<\epsilon$. We have $F = \phi^{-1}(\G)$ is a finite set thanks to \eqref{Finiteness}. Let us choose $\eta>0$ so small that 
	\begin{itemize}
		\item[\ding{64}] the balls $B(s, \eta)$ are pairwise disjoint for $s\in F$,
		\item[\ding{64}] $\{(x_i,y_j): i,j=1,\dots 2^k-1 \} \cap B(s,\eta) = \emptyset$,
		\item[\ding{64}] for all $s \in F$ and all $t\in B(s,\eta)\cap \Gamma_k$
		\begin{equation}\label{Leave}
			|\phi(t) - \phi(s) - D_{\tau}\phi(s)(t-s)|< \tfrac{1}{4} |D_{\tau}\phi(s)|\cdot|t-s|,
		\end{equation}
		\item[\ding{64}] if $ s\in F, \ \pi_1\circ\phi(s)\in \{w_n\}_{n=0}^M$ then for all $t\in B(s,\eta)\cap \Gamma_k$ 
		\begin{equation}\label{Side1}
			|\pi_1\circ\phi(t) - \pi_1\circ\phi(s) - D_{\tau}\pi_1\circ\phi(s)(t-s)|< \tfrac{1}{4}|D_{\tau}\pi_1\circ\phi(s)|\cdot|t-s|,
		\end{equation}
		\item[\ding{64}] if $ s\in F, \ \pi_2\circ\phi(s)\in \{z_m\}_{m=0}^M$ then for all $t\in B(s,\eta)\cap \Gamma_k$
		\begin{equation}\label{Side2}
			|\pi_2\circ\phi(t) - \pi_2\circ\phi(s) - D_{\tau}\pi_2\circ\phi(s)(t-s)|< \tfrac{1}{4}|D_{\tau}\pi_2\circ\phi(s)|\cdot|t-s|.
		\end{equation} 
	\end{itemize}
	Let us define
	$$
	\begin{aligned}
		v_1 &= \min\Big\{\frac{|D_{\tau}\pi_1\circ \phi(s)|}{|D_{\tau}\phi(s)|}; \ s\in F, \ \pi_1\circ\phi(s)\in \{w_n\}_{n=0}^M  \Big\},\\
		v_2 &= \min\Big\{\frac{|D_{\tau}\pi_2\circ \phi(s)|}{|D_{\tau}\phi(s)|}; \ s\in F, \ \pi_2\circ\phi(s)\in \{z_m\}_{m=0}^M  \Big\}
	\end{aligned}
	$$
	and $v := \tfrac{3}{5}\min\{v_1,v_2\}<1$.
	
	To aid the understanding of the following we refer the reader to Figure~\ref{fig:CrossingAnArrivalGrid}. We find a $d\in (0,\tfrac{1}{3}\eta\min_{s\in F}|D_{\tau}\phi(s)|)$ so small that
	\begin{itemize}
		\item[\ding{64}] $B(a,2d)$ are pairwise disjoint for $a\in \phi(F)$,
		\item[\ding{64}] $B(a,2d)$ do not contain any cross points $(w_n,z_m)$ of $\G$.
	\end{itemize}
	It holds that $\phi(t)\in \er^2\setminus B(\phi(s),d)$ for all $t\in [B(s,\eta)\setminus B(s,\tfrac{1}{2}\eta)]\cap \Gamma_k$ because (using \eqref{Leave})
	\begin{equation}\label{alone}
		|\phi(t)  -\phi(s)| \geq \tfrac{3}{4} |D_{\tau}\phi(s)|\cdot|t-s| > \tfrac{3}{8}|D_{\tau}\phi(s)|\eta \geq\tfrac{9}{8}d.
	\end{equation}
	Then, for each $s\in F$, call $s_1,s_2$ the points of $\overline{B(s,\eta)}\cap \Gamma_k$ farthest points from $s$ in each direction such that $\phi(s_i)\in \partial B(\phi(s),d)$ (i.e. the `first' time $\phi(\Gamma_k \cap B(s,\eta))$ enters $B(\phi(s),d)$ and the `last' time it exits). Using the equations \eqref{Side1} and \eqref{Side2} we have
	$$
	|\pi_i\circ\phi(t)  -\pi_i\circ\phi(s)| \geq \tfrac{3}{4} |D_{\tau}\pi_i\circ\phi(s)|\cdot|t-s|
	$$
	and by \eqref{Leave}
	$$
	|\phi(t)  -\phi(s)| \leq \tfrac{5}{4} |D_{\tau}\phi(s)|\cdot|t-s|.
	$$
	Thus, using the definition of $v$,
	\begin{equation}\label{Angle}
		|\pi_i\circ\phi(t)  -\pi_i\circ\phi(s)| > v|\phi(t)  -\phi(s)|.
	\end{equation}
	Thus, we conclude that either
	$$
	\pi_i\circ\phi(s_1) - \pi_i\circ\phi(s) < -vd < vd< \pi_i\circ\phi(s_2) - \pi_i\circ\phi(s)
	$$
	or
	$$
	\pi_i\circ\phi(s_1) - \pi_i\circ\phi(s) > vd > - vd > \pi_i\circ\phi(s_2) - \pi_i\circ\phi(s). 
	$$

	\begin{figure}
		\begin{tikzpicture}[line cap=round,line join=round,>=triangle 45,x=7.0cm,y=7.0cm]
			\clip(1.5330169722110216,4.94) rectangle (2.4852766177242094,5.8801014311259285);
			\fill[line width=0.7pt,fill=black!10] (1.9,4.) -- (2.1,4.) -- (2.1,8.) -- (1.9,8.) -- cycle;
			\fill[line width=0.7pt,fill=black!20] (1.92,4.) -- (2.08,4.) -- (2.08,8.) -- (1.92,8.) -- cycle;
			\draw [line width=0.7pt] (6.,4.938142671201255) -- (6.,5.8801014311259285);
			\draw [line width=0.7pt] (2.,4.938142671201255) -- (2.,5.8801014311259285);
			\draw [line width=0.7pt,domain=1.5330169722110216:2.4852766177242094] plot(\x,{(--56.-0.*\x)/7.});
			\draw [line width=0.7pt,domain=1.5330169722110216:2.4852766177242094] plot(\x,{(--32.-0.*\x)/8.});
			\draw [line width=0.7pt] (2.,5.4)-- (2.250821945714846,6.178561045441247);
			\draw [line width=0.5pt,dash pattern=on 1pt off 1pt] (2.,5.4) circle (2.8cm);
			\draw [line width=0.7pt] (2.250821945714846,6.178561045441247)-- (2.192621759497969,6.813385804346);
			\draw [line width=0.7pt] (2.,5.4)-- (1.676452234981737,4.460689117710067);
			\draw [line width=0.7pt] (1.9,4.)-- (2.1,4.);
			\draw [line width=0.7pt] (1.92,4.)-- (2.08,4.);
			\draw [line width=0.7pt] (2.08,8.)-- (1.92,8.);
			\draw [line width=0.7pt,dotted,color=ffqqqq] (1.84335951359307,5.012500907545781)-- (1.8830179667591742,5.002766559950465);
			\draw [line width=0.7pt,dotted,color=ffqqqq] (1.8830179667591742,5.002766559950465)-- (1.88482062372127,5.0323301341288325);
			\draw [line width=0.7pt,dotted,color=ffqqqq] (1.88482062372127,5.0323301341288325)-- (1.851291204226291,5.021874723748678);
			\draw [line width=0.7pt,dotted,color=ffqqqq] (1.851291204226291,5.021874723748678)-- (1.9139795137578615,5.126653237955748);
			\draw [line width=0.7pt,dotted,color=ffqqqq] (1.9139795137578615,5.126653237955748)-- (1.910648215530878,5.171856059793091);
			\draw [line width=0.7pt,dotted,color=ffqqqq] (1.910648215530878,5.171856059793091)-- (1.9390280050672684,5.17925948315041);
			\draw [line width=0.7pt,dotted,color=ffqqqq] (1.9390280050672684,5.17925948315041)-- (1.9347093414421654,5.26624970759891);
			\draw [line width=0.7pt,dotted,color=ffqqqq] (1.9347093414421654,5.26624970759891)-- (1.987443213667146,5.335728116672842);
			\draw [line width=0.7pt,dotted,color=ffqqqq] (1.987443213667146,5.335728116672842)-- (1.9802187669606302,5.368639485002525);
			\draw [line width=0.7pt,dotted,color=ffqqqq] (1.9802187669606302,5.368639485002525)-- (2.010594182935993,5.3805525551520645);
			\draw [line width=0.7pt,dotted,color=ffqqqq] (2.010594182935993,5.3805525551520645)-- (1.9870089667148896,5.406438768077665);
			\draw [line width=0.7pt,dotted,color=ffqqqq] (1.9870089667148896,5.406438768077665)-- (2.0140456779927396,5.408739764782163);
			\draw [line width=0.7pt,dotted,color=ffqqqq] (2.0140456779927396,5.408739764782163)-- (1.9921862093000098,5.436926974412262);
			\draw [line width=0.7pt,dotted,color=ffqqqq] (1.9921862093000098,5.436926974412262)-- (2.0278516582197263,5.462237938161738);
			\draw [line width=0.7pt,dotted,color=ffqqqq] (2.0278516582197263,5.462237938161738)-- (2.0169219238733618,5.484672656030591);
			\draw [line width=0.7pt,dotted,color=ffqqqq] (2.0169219238733618,5.484672656030591)-- (2.041106951531704,5.489775709132981);
			\draw [line width=0.7pt,dotted,color=ffqqqq] (2.041106951531704,5.489775709132981)-- (2.0348657219641195,5.559964120131122);
			\draw [line width=0.7pt,dotted,color=ffqqqq] (2.0348657219641195,5.559964120131122)-- (2.072289287890644,5.60310517529642);
			\draw [line width=0.7pt,dotted,color=ffqqqq] (2.072289287890644,5.60310517529642)-- (2.070210200894726,5.662878926429064);
			\draw [line width=0.7pt,dotted,color=ffqqqq] (2.070210200894726,5.662878926429064)-- (2.1122473913025295,5.7907580873157185);
			\draw [line width=0.7pt,dotted,color=ffqqqq] (2.1122473913025295,5.7907580873157185)-- (2.1250026302828933,5.763786368506695);
			\draw [line width=0.7pt,dotted,color=ffqqqq] (2.1250026302828933,5.763786368506695)-- (2.128042626602505,5.81368665468665);
			\draw [line width=0.7pt,dotted,color=ffqqqq] (2.128042626602505,5.81368665468665)-- (2.0856081479822564,5.780661064257253);
			\draw [line width=0.7pt,dotted,color=ffqqqq] (2.0856081479822564,5.780661064257253)-- (2.08388189238078,5.810348619123719);
			\draw [line width=0.7pt,color=qqqqff] (2.08388189238078,5.810348619123719)-- (2.1569953885091864,5.938329849250318);
			\draw [line width=0.7pt,color=qqqqff] (2.1569953885091864,5.938329849250318)-- (2.23301989076101,6.051214716230298);
			\draw [line width=0.7pt,dotted,color=ffqqqq] (1.84335951359307,5.012500907545781)-- (1.8376710539561527,5.054815002535113);
			\draw [line width=0.7pt,dash pattern=on 1pt off 1pt,color=ffqqqq] (1.8376710539561527,5.054815002535113)-- (1.7605226675684065,4.6486929286010925);
			\draw [line width=0.7pt,color=qqqqff] (1.8341016546569613,5.03602508463844)-- (2.0850150386665174,5.7908611559115855);
			\draw [line width=0.7pt,color=qqqqff] (2.0850150386665174,5.7908611559115855)-- (2.08388189238078,5.810348619123719);
			\draw [line width=0.7pt,color=qqqqff] (1.8341016546569613,5.03602508463844)-- (1.7605226675684065,4.6486929286010925);
			\draw (1.8,5.86) node[anchor=north west] {$vd$};
			\draw (1.915,5.76) node[anchor=north west] {$\tfrac{7}{8}vd$};
			\draw (1.85,5.58) node[anchor=north west] {$\tilde{\phi}(s)$};
			\draw [color=ffqqqq](2.1041438509503014,5.709564547494366) node[anchor=north west] {$\phi_{\delta}$};
			\draw [color=blue](1.8,5.3) node[anchor=north west] {$\tilde{\phi}_{\delta}$};
			\draw (1.6,5.05) node[anchor=north west] {$\tilde{\phi}_{\delta}(s_{1,\delta})$};
			\draw (2.03,5.5) node[anchor=north west] {$\phi$};
			\begin{scriptsize}
				\draw [fill=black] (2.,5.535105895826601) circle (2.0pt);
				\draw [fill=uuuuuu] (1.8341016546569613,5.03602508463844) circle (1.5pt);
			\end{scriptsize}
		\end{tikzpicture}
		\caption{The definition of $\tilde{\phi}_{\delta}$ on $V$ using $\phi_{\delta}$.}\label{fig:CrossingAnArrivalGrid}
	\end{figure}
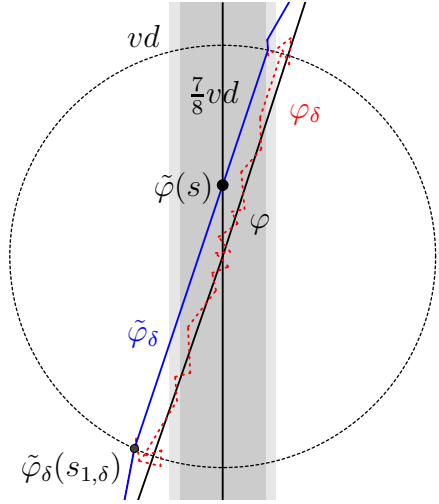

	We can make a similar conclusion for $\phi_{\delta}$ for $\delta \in (0,\tfrac{1}{8}vd)$, i.e. from \eqref{alone} we find $s_{1,\delta}, s_{2,\delta}$ such that $\phi_{\delta}^{-1}(B(\phi(s),d))\cap \Gamma_k\cap B(s,\eta) \subset s_{1,\delta}s_{2,\delta}$ and $\phi_{\delta}(s_{j,\delta})\in \partial B(\phi(s),d)$, $j=1,2$. Because $|\phi(s_{i,\delta}) - \phi(s)|\geq d-\delta \geq \tfrac{7}{8}d$,  the equation \eqref{Angle} with $t=s_{1,\delta}, s_{2,\delta}$ gives either
	$$
	\pi_i\circ\phi(s_{1,\delta}) - \pi_i\circ\phi(s) < -\tfrac{7}{8}vd <\tfrac{7}{8} vd< \pi_i\circ\phi(s_{2,\delta}) - \pi_i\circ\phi(s)
	$$
	and so
	$$
	\pi_i\circ\phi_{\delta}(s_{1,\delta}) - \pi_i\circ\phi(s) < -\tfrac{3}{4}vd <\tfrac{3}{4} vd< \pi_i\circ\phi_{\delta}(s_{2,\delta}) - \pi_i\circ\phi(s)
	$$
	or the same estimates with the roles of $s_{1,\delta}$ and $s_{2,\delta}$ reversed.
	In either case, the segment $\phi_{\delta}(s_{1,\delta}) \phi_{\delta}(s_{2,\delta} )$ intersects $\G$ exactly once at a point we call $\tilde{\phi}_{\delta}(s)$. Further we call $V = V_{\delta} = \bigcup_{s\in F}s_{1,\delta}s_{2,\delta}$ and we assume that
	\begin{equation}\label{PleaseSirCanIHaveSomeMore}
		\delta <\tfrac{1}{2} \dist\big(\G, \phi(\Gamma_k \setminus V)\big)\quad\text{and}\quad \delta <\tfrac{1}{2}\epsilon \min\{|a-b|; a\neq b, a,b \in \phi(F)\}.
	\end{equation}
	The first condition guarantees that $\phi_{\delta}(\Gamma_k\setminus V)\cap \G = \emptyset$ and we will later use the second condition to prove that \eqref{SizesXX} holds.

	Note that, because $\phi(\Gamma_k \setminus V_{\delta})\cap \phi(F) = \emptyset$ for all $\delta \in (0,\tfrac{1}{8}vd)$ (independently of the continuous uniform approximation we take), by choosing $d$ small enough we have $\phi_{\delta}^{-1}(B(a,d)) \setminus V = \emptyset$. We define a map $\tilde{\phi}_{\delta}:\Gamma_k \to \er^2$. For $t\in \Gamma_k\setminus V$ we define $\tilde{\phi}_{\delta}(t) = \phi_{\delta}(t)$. On the segment $s_{1,\delta}s$ we define $\tilde{\phi}_{\delta}$ as the constant speed parametrization of the segment $\phi_{\delta}(s_{1,\delta})\tilde{\phi}_{\delta}(s)$. Similarly on the segment $ss_{2,\delta}$ we define $\tilde{\phi}_{\delta}$ as the constant speed parametrization of the segment $\tilde{\phi}_{\delta}(s)\phi_{\delta}(s_{2,\delta})$. Clearly $\tilde{\phi}_{\delta}$ is continuous. Thanks to the injectivity of ${\phi}_{\delta}$ it is easily observed that the map $\tilde{\phi}_{\delta}$ is also injective. We have that $|\tilde{\phi}_{\delta}(s) - \phi_{\delta}(s)| < 2\delta + 2d$.

	Now we define the map $\gamma$ by first defining its image and then later defining its parametrization. We refer to pairs $s,s' \in F$ as neighboring pairs in $F$ if the segment $ss'\subset \Gamma_k$ and $ss' \cap F = \{s, s'\}$ (i.e. $ss'$ is a segment in $\Gamma_k$ between a neighboring pair of points of $F$). Let $s,s' \in F$ be a neighboring pair in $F$ such that $\tilde{\phi}_{\delta}(s)$ and $\tilde{\phi}_{\delta}(s')$ do not lie on a common segment of $\G$. Then we define the image of the segment $ss'$ in $\gamma$ to be the segment $\tilde{\phi}_{\delta}(s)\tilde{\phi}_{\delta}(s')$. By the choice of $w_n$ and $z_m$ we have $\int_{ss'}|D_{\tau}\phi|<2^{-k+1}\epsilon$.
	
	Let $s,s' \in F$ be a neighboring pair in $F$ with images on a common side of some $\partial R_{n,m} = \partial\big([w_n,w_{n+1}] \times[z_m, z_{m+1}]\big)$. The following observations follow simply from the injectivity of $\tilde{\phi}_{\delta}$ (see \cite[Step~II, Proposition~3.8]{DPP}). Firstly we notice that for any other neighboring pair $t,t' \in F$, either $ss'\cap tt' = \{x_i, y_j\} $ (for some $1\leq i,j\leq 2^k-1$ and $[\tilde{\phi}(s)\tilde{\phi}(s')]_{\xi} \cap [\tilde{\phi}(t)\tilde{\phi}(t')]_{\xi} \neq \emptyset$ (in this case the intersection is a single point which we denote as $A_{i,j} = A_{i,j}(\xi)$) for all $\xi>0$. Otherwise $ss'\cap tt' = \emptyset$ and there exists a $\xi_0>0$ such that $[\tilde{\phi}(s)\tilde{\phi}(s')]_{\xi} \cap [\tilde{\phi}(t)\tilde{\phi}(t')]_{\xi} = \emptyset$ for all $\xi\in(0,\xi_0)$.  Since there are a finite number of neighboring pairs we find a $0< \xi <\epsilon$ such that $[\tilde{\phi}(s)\tilde{\phi}(s')]_{\xi} \cap [\tilde{\phi}(t)\tilde{\phi}(t')]_{\xi} \neq \emptyset$ exactly when  $ss'\cap tt' \neq \emptyset$. We determine that the image of $ss'$ in $\gamma$ is $[\tilde{\phi}(s)\tilde{\phi}(s')]_{\xi}$.
	
	Having fixed $\xi$, we define $\gamma(x_i,y_i)$ as $A_{i,j}$ and $\gamma(s) = \tilde{\phi}_{\delta}(s)$ for $s\in F$. The remaining part of $\Gamma_k$ is made up of segments that have to go onto segments (or the union of two segments in the case when we use generalized segments). It is a simple observation that $\gamma$ is injective thanks to the fact that $\tilde{\phi}_{\delta}$ is injective and the two maps are obviously homotopically equivalent (see \cite[Step~II, Proposition~3.8]{DPP}).
	
	For segments with endpoints mapped onto distinct segments of $\G$ the estimate \eqref{SizesXX} is satisfied immediately from the second condition of \eqref{PleaseSirCanIHaveSomeMore}. For parts where we have to use a generalized segment comprised of 2 segments we use Proposition~\ref{EverythingIDo} to get the estimate in \eqref{SizesXX} only this time the right hand side is multiplied by $(1+\xi)$, but $\xi$ can be chosen as small as we like.
\end{proof}
\begin{remark}
	Let us note that Proposition~\ref{PiecewiseLinearization} we do not need $\Gamma_k$ to be a grid of horizontal and vertical segments. We only did so to simplify the notation. We can replace $\Gamma_k$ with a general grid of piecewise linear curves (so that the concept of piecewise linear curves makes clear sense). The generalization to general grids is a simple exercise which the reader can check once he knows the proof in the case for straight grids. We summarize this in the following claim.
\end{remark}
\begin{prop}
	Let $M\in \en$ and let $\{\sigma_i\}_{i=1}^M$ be a general grid of piecewise linear curves and let $\Gamma = \bigcup_{i=1}^M\sigma_i([0,1])$. Let $p\in [1,\infty)$, let $\phi \in W^{1,p}(\Gamma, \er^2)$ and let for every $\delta>0$ there exists an injective and continuous $\phi_{\delta}:\Gamma \to \er^2$ such that $|\phi_{\delta}(s) - \phi(s)| < \delta$  for all $s\in \Gamma$. Then there exist continuous piecewise linear injective $\gamma_l \sto \phi$ on $\Gamma$ such that
	$$
	\limsup_{l\to \infty}\int_{\Gamma}|D_{\tau}\gamma_l|^p \leq \int_{\Gamma}|D_{\tau}\phi|^p
	$$
	and $\{|D_{\tau}\gamma_l|\}_{l\in \en}$ is uniformly integrable on $\Gamma_k$.
\end{prop}
\subsection{Extension Results}
In the following $\tilde{Q}_r$ denotes the rotated square, specifically the convex envelope of the points $(r,0), (-r,0), (0,r), (0, -r)$. The following extension theorem is the reformulation of a known result of Hencl and Pratelli from \cite{HP}. We reformulate their theorem  by extracting some derivative estimates they show in the course of the proof. 	
\begin{thm}
	There exists a $C>0$ such that for any $r > 0$ and any finitely piece-wise linear and one-to-one function $v:\partial \tilde{Q}_r \to\er^2$ we can find a finitely piece-wise affine homeomorphism $h:Q(0,r)\to \er^2$ such that $h_{\rceil\partial \tilde{Q}_r}=v$ and for almost all $(x,y)\in \tilde{Q}_r$
	\begin{equation}\label{hope}
		\begin{aligned}
			|D_1h(x,y)| &\leq C\oint_{\partial \tilde{Q}_r}|D_{\tau}v| \ d\H^1 \text{ and }\\
			|D_2h(x,y)| &\leq C\big|D_{\tau}v(-r-|y|,y)\big|+ C\big|D_{\tau}v(r-|y|,y)\big| + C\oint_{\partial \tilde{Q}_r}|D_{\tau}v| \ d\H^1
		\end{aligned}
	\end{equation}
	where $D_{\tau}v$ is the tangential derivative of $v$ on $\partial\tilde{Q}_r$.
\end{thm}
\begin{proof}
	The extension theorem used here is exactly the `shortest curve extension', \cite[Theorem~2.1]{HP}. The estimates on the derivatives are already proved in their step 9 of their proof. Let $[A_1A_2]$ and $[B_1B_2]$ be a pair of disjoint segments in $\partial\tilde{Q}_r$ with the same projection onto the $y$-axis and $v$ is linear on each of the segments. The authors of \cite{HP} prove
	$$
	|D_1h(x,y)| \leq C\oint_{\partial\tilde{Q}_r} |D_{\tau}v| \ d\H^1
	$$	
	in their equation (2.11) (note that they prove it only for the case $r=1$ but a scaling argument immediately gives the above). The second to last estimate of step 9 is precisely 
	$$
	|D_2h(x,y)| \leq 5\big|D_{\tau}v(-r-|y|,y)\big|+ 5\big|D_{\tau}v(r-|y|,y)\big| + 3C\oint_{\partial\tilde{Q}_r} |D_{\tau}v| \ d\H^1
	$$
	for almost all $(x,y) \in \co\{A_1,A_2,B_1,B_2\}$. For $(x,y)$ in triangles near $(0,r)$ or $(0,-r)$ the estimates are proved in step 8.
\end{proof}
\begin{corollary}\label{SpecialExtension}
	Let $v_j:\partial \tilde{Q}_r \to\er^2$ be continuous injective and (for each $j$ but not necessarily uniformly) finitely piece-wise linear mappings such that $v_j$ converge weakly in $W^{1,1}(\partial\tilde{Q}_r, \er^2)$. Then there exist equicontinuous finitely piecewise affine homeomorphisms $h_j:\tilde{Q}_r\to \er^2$ with $|Dh_j|$ uniformly integrable. Further, for all $p\in [1,\infty)$, 
	\begin{equation}\label{UseMe}
		\int_{\tilde{Q}_r} |Dh_j|^p \ d\mathcal{L}^2 \leq C r \int_{\partial\tilde{Q}_r} |D_{\tau}v_j|^p \ d\H^1.
	\end{equation}
\end{corollary}
\begin{proof}
	Let $(x,y), (\tilde{x}, \tilde{y})\in\tilde{Q}_r$ be a pair of distinct points. Because $\int_{\partial\tilde{Q}_r}|D_{\tau}v_j|$ is bounded independently of $j$, the equation \eqref{hope} gives a uniform bound on $|D_1 h_j|$. Then it suffices to prove equicontinuity in the second variable.  Since $|D_{\tau}v_j|$ are uniformly integrable the estimate \eqref{hope} gives that for every $\epsilon>0$ there exists a $\delta>0$ such that when $|y-\tilde{y}|<\delta$ then $|h_j(x,y) - h_j(x,\tilde{y})|<\epsilon$ which proves the equicontinuity of $\{h_j\}$. The uniform integrability of $|Dh_j|$ follows immediately from \eqref{hope} and the uniform integrability of $|D_{\tau}v_j|$. The $L^p$ estimates follow immediately from $\eqref{hope}$.
\end{proof}

The following generalization of the extension from \cite{HP} was published in \cite{ER}. We provide the version of the theorem corresponding to Remark~2.1 from that paper.
\begin{thm}\label{PExtension}
	There exists a $C>0$ such that for any $r > 0$ and any finitely piece-wise linear and one-to-one function $v:\partial \tilde{Q}_r \to\er^2$ we can find a finitely piece-wise affine homeomorphism $h:Q(0,r)\to \er^2$ such that $h_{\rceil\partial \tilde{Q}_r}=v$ and for almost all $(x,y)\in \tilde{Q}_r$
	\begin{equation}\label{phope}
		\oint_{\tilde{Q}_r} |Dh|^p \leq C \oint_{\partial \tilde{Q}_r}|D_{\tau}v|^p \ d\H^1 
	\end{equation}
	where $D_{\tau}v$ is the tangential derivative of $v$ on $\partial\tilde{Q}_r$.
\end{thm}

\subsection{The $\INV$ condition, the no-crossing condition and other monotonicity-type conditions}
In the following we use the notation that $\mathbb{S}$ is the unit circle, i.e. $\mathbb{S} = \partial B(0,1)$ in $\er^2$.

We refer to the following as the $\INV$ condition. The so-called INV condition was introduced in \cite{MS} and can be recovered from the following definition by assuming that $\phi$ is affine. It was shown in \cite[Theorem 9.1]{MS} that the two concepts are the same if $\det Du >0$ a.e.. An example in \cite[Section 5.1]{DPP} shows that without this additional assumption the two conditions are not the same, not even for Lipschitz continuous maps.

\begin{definition}[$\INV$ functions]\label{defINV}
	Let us call $\theta(t) = (\cos(t),\sin(t))$ for $t\in \er$. Let $u\in W^{1,1}((-1,1)^2, \er^2)$, then we say that $u$ satisfies the $\INV$ condition if the following holds. Let $\phi: \mathbb{S} \mapsto (-1,1)^2$ be an injective Lipschitz continuous embedding of the unit circle with constant $|(\phi\circ\theta)'(t)| = C$ almost everywhere on $\er$ such that $u_{\rceil \phi(\mathbb{S})} \in W^{1,1}(\phi(\mathbb{S}), \er^2)$ (i.e. $u\circ\phi\circ \theta \in W^{1,1}_{\loc}(\er,\er^2)$) then
	\begin{enumerate}
		\item[a)] for almost every $(x,y)$ lying inside the interior of $\phi(\mathbb{S})$ (i.e. lying in the bounded component of $\er^2\setminus \phi(\mathbb{S})$) it holds that
		$$u(x,y)\in \{(w,z): \deg((w,z), u, \varphi(\mathbb{S}))\neq 0\}\cup u(\varphi(\mathbb{S}))$$
		\item[b)] for almost every $(x,y)$ lying outside $\phi(\mathbb{S})$ (i.e. lying in the unbounded component of $\er^2\setminus \phi(\mathbb{S})$) it holds that
		$$u(x,y)\in \{(w,z): \deg((w,z), u, \varphi(\mathbb{S}))= 0\} \cup u(\varphi(\mathbb{S})). $$
	\end{enumerate}
\end{definition}
A very useful result for INV maps is the following. A proof can be found for example in \cite[Lemma 2.5]{DPP}.
\begin{lemma}\label{BadSize}
	Let $1\leq p <2$ and let $u \in W^{1,p}((-1,1)^2,\er^2)$ satisfy the INV condition. Then, there exists a set $S_u\subset (-1,1)^2$ such
	that $\H^{2-p}(S_u) = 0$ and there is a representative of $u$ such that $u_{\rceil Q\setminus S_u}$ is continuous.
\end{lemma}
Note that the property that $u$ is continuous outside a set $S_u$ with $\H^{2-p}(S_u) = 0$ is stronger information than that a $W^{1,p}$ map is $p$-quasi-continuous, since already $\H^{2-p}(S_u) < \infty$ implies that $C_p(S_u)= 0$ where $C_p$ is the capacity defined by taking the infimum over integrals of the type $\int_{\er^2}|\nabla v|^p + |v|^p$. This is especially useful in our context when we consider \cite[Corollary 2.7]{DPP}, i.e.
\begin{thm}\label{HomImpINV}
	Let $u\in \overline{H\cap W_{\id}^{1,p}((-1,1)^2, \er^2)}$, then $u$ satisfies the $\INV$ condition.
\end{thm}

In the following we define a multi-function representative of $u\in W^{1,1}((-1,1)^2,\er^2)$ called $u_{\text{multi}}$. The concept makes sense for all $u$ satisfying the $\INV$ condition. Almost exclusively we work with $u$ on $(-1,1)^2\setminus S_u$ where $u_{\text{multi}}(x,y) = \{u(x,y)\}$ for the precise representative of $u$. We are mostly interested in the resulting set operation $u^{-1}$. 
\begin{definition}\label{defIMG}
	Let us call $\theta(t) = (\cos(t),\sin(t))$ for $t\in \er$. Let $u\in W^{1,1}((-1,1)^2, \er^2)$ and let $(x,y) \in (-1,1)^2$. We say $(w,y)\in u_{\text{multi}}(x,y)$ if for every $k\in \en$ there is an injective Lipschitz continuous embedding of the unit circle $\phi_k: \mathbb{S} \mapsto B((x,y), \tfrac{1}{k})\subset (-1,1)^2$ with constant $|(\phi_k\circ\theta)'(t)|$ almost everywhere on $\er$ such that $u\circ\phi_k\circ \theta \in W^{1,1}_{\loc}(\er,\er^2)$) and it holds that
	$$
	(w,z)\in \{(a,b): \deg((a,b), u, \varphi_k(\mathbb{S}))\neq 0\}\cup u(\varphi_k(\mathbb{S})).
	$$ 
	For each $C\subset \er^2$ we define
	$$
	u^{-1}(C) = \{(x,y): u_{\text{multi}}(x,y)\cap C \neq \emptyset\}.
	$$
\end{definition}

The following two lemmas provide a characterization for maps satisfying the $\INV$ condition. We believe that this characterization is already known to experts of the subject, but we provide an independent proof for the first half, for the lack of an easily accessible reference.
\begin{lemma}\label{IDidThisOne}
	Let $u\in W^{1,1}((-1,1)^2, \er^2)$ satisfy the $\INV$ condition. Then the preimage of a closed connected set is a closed connected set.
\end{lemma}
\begin{proof}
	First we prove that $u^{-1}(C)$ is closed. Let $C\subset (-1,1)^2$ be closed and connected. Let $x_k\in u^{-1}(y_k)$ for some $y_k\in C$ such that $x_k \to x \in (-1,1)^2$. We assume that $y_k\to y \in C$. For each $k$ there exists some curve $\phi_{k}(\mathbb{S})\subset B(x_k,\tfrac{1}{k})$ enclosing $x_k$ such that either $y_k\in u\circ\phi_k(\mathbb{S})$ or $\deg(y_k, u,\phi_k(\mathbb{S}))\neq 0$. Let $\phi(\mathbb{S})$ be any curve enclosing $x$, then there is some $k_0$ such that $\phi_k(\mathbb{S})$ is entirely enclosed by $\phi$ for all $k\geq k_0$. This means also that either $y\in u\circ\phi(\mathbb{S})$ or $\deg(y, u,\phi(\mathbb{S}))\neq 0$. Therefore $y\in u(x)$ and so $x\in u^{-1}(C)$. Thus, we see $u^{-1}(C)$ is closed.
	
	Let us assume that $u^{-1}(C)$ has two distinct components $E_1$ and $E_2$. Because $u^{-1}(C)$ is closed, necessarily $3\rho = \min\{\dist(E_1,u^{-1}(C)\setminus E_1), \dist(E_2,u^{-1}(C)\setminus E_2)\}>0$. Then it is possible to find injective Lipschitz curves $\phi_1,\phi_2:\mathbb{S}\to \er^2$ such that $u\circ\phi_i\circ\theta\in W^{1,1}_{\loc}(\er, \er^2)$, $\phi_i(\mathbb{S})\subset [E_i+B(0,\rho)]\setminus E_i$ and $E_i$ is in the interior of $\phi_i(\mathbb{S})$. We have that $\phi_1(\mathbb{S})\cap \phi_2(\mathbb{S}) = \emptyset$. Then either $\inter\phi_1(\mathbb{S})\cap \inter\phi_2(\mathbb{S}) = \emptyset$ or without loss of generality $\inter\phi_1(\mathbb{S})\subset \inter\phi_2(\mathbb{S})$. In either case we may assume that $E_2$ lies in the unbounded component of $\er^2\setminus \phi_1(\mathbb{S})$. Since $u\circ \phi_1$ is continuous and $u\circ\phi_1(\mathbb{S})\cap u^{-1}(C) = \emptyset$, the $\INV$ condition gives that $C \subset \{(w,z): \deg((w,z), u, \phi(\mathbb{S}))\neq 0\}$ and therefore $u_{\text{multi}}(E_2) \cap C = \emptyset$, which is a contradiction. 
\end{proof}
The second of the lemmas is \cite[Lemma 2.8]{DPP}. 
\begin{lemma}\label{StolenFromAldo}
	Let $u\in W^{1,1}_{\id}((-1,1)^2, \er^2)$ be such that the preimage of a point is a closed connected set. Then $u$ satisfies the $\INV$ condition.
\end{lemma}

The `no-crossing' condition was introduced in \cite{DPP}. The following is a formulation for Sobolev maps.
\begin{definition}[NC functions]\label{defNC}
	Let $u\in W^{1,p}_{\id}([-1,1]^2,\er^2)$. We say that $u$ satisfies the \emph{no-crossing condition} if the following holds: 
	
	Let $M\in \en$ and let $\{\gamma_i\}_{i=1}^M$ be a general grid called $\Gamma  = \bigcup_{i=1}^M \gamma_i([0,1])$ such that $u_{\rceil \Gamma} \in W^{1,p}(\Gamma, \er^2)$. Then, for every $\epsilon>0$ there exists a continuous and injective function $v:\Gamma\to\er^2$ such that $\|v-u\|_{L^\infty(\Gamma)}<\epsilon$.
\end{definition}
We want to use this concept, but we prefer to formulate a no crossing condition which is easier to prove. We now define our version of the no crossing condition and below we show that the two concepts are equivalent.
\begin{definition}\label{StayCalmPlease}
	Let $u\in W^{1,p}_{\id}([-1,1]^2,\er^2)$. We say that $u$ satisfies the no-crossing condition if there exists $\mathfrak{G} = \{\Gamma_k\}_{k\in \en},$ a suitable grid system for $u$ such that for each $k$ there is a sequence of continuous injective
	$\phi_{k,m} : \Gamma_k \to \er^2$ such that $\phi_{k,m} \sto u_{\rceil \Gamma_k}$ on $\Gamma_k$ as $m\to \infty$.
\end{definition}
\begin{remark}\label{WeakRemark}
	With respect to Proposition~\ref{PiecewiseLinearization} the condition that \emph{there are continuous, injective $\phi_{k,m} : \Gamma_k \to \er^2$ such that $\phi_{k,m} \sto u_{\rceil \Gamma_k}$ on $\Gamma_k$}, in fact, implies that \emph{there are continuous, piecewise linear, injective $\phi_{k,m} \in W^{1,p}(\Gamma_k, \er^2)$ with $\phi_{k,m} \deb u_{\rceil \Gamma_k}$ in $W^{1,p}(\Gamma, \er^2)$}.
\end{remark}
Formally the no-crossing condition of Definition~\ref{StayCalmPlease} appears to be much weaker than the condition with the same name from Definition~\ref{defNC}. In Proposition~\ref{NoCrossEquivalence} we use the equivalence of the weak and strong closures of homeomorphisms shown in \cite[Theorem~B]{DPP} to prove that the conditions are in fact equivalent.

\begin{prop}\label{NoCrossEquivalence}
	Let $u\in W^{1,p}_{\id}([-1,1]^2,\er^2)$  and $S_u \subset (-1,1)^2$ be such that $\H^1(S_u) = 0$ and $u_{\rceil [-1,1]^2 \setminus S_u}$ is continuous. Then $u\in \overline{H\cap W^{1,p}_{\id}((-1,1)^2, \er^2)}$ if and only if $u$ satisfies the no-crossing condition of Definition~\ref{StayCalmPlease}.
\end{prop} 
\begin{proof}
	The `only if' part of the claim follows straight forwardly. By Proposition~\ref{GridExist} we have the existence of suitable $k$-grids $\Gamma_k$ for $u$ for each $k\in \en$. Let us choose some fixed $k\in \en$. The no-crossing condition of Definition~\ref{defNC} gives injective continuous $\tilde{\phi}_{k,m} \sto u_{\rceil\Gamma_k}$ on $\Gamma_k$ as $m\to \infty$. Thus, we satisfy Definition~\ref{StayCalmPlease}.
	
	On the other hand, let us assume that the condition of Definition~\ref{StayCalmPlease} is satisfied. Let $\Gamma_k$ be a suitable $k$-grid for $u$. We have continuous injective $\phi_{k,m} \sto u_{\rceil \Gamma_k}$ on $\Gamma_k$ and by Remark~\ref{WeakRemark} also $\phi_{k,m} \deb u_{\rceil \Gamma_k}$ in $W^{1,p}(\Gamma_k ,\er^2)$. Employing a uniformly bi-Lipschitz linear change of variables we use Theorem~\ref{PExtension} on each $K_{k,i,j} = [x_{k,i},x_{k,i+1}]\times[y_{k,j}, y_{k,j+1}]$, rectangle of $\Gamma_k$. We get a finitely piecewise affine homeomorphism $h_{k,m}$ with
	$$
	\int_{(-1,1)^2}|Dh_{k,m}|^p \leq C2^{-k}\sum_{i,j=0}^{2^k-1}\int_{\partial K_{k,i,j}} |D_{\tau}\phi_{k,m}|^p \ d\H^1
	$$
	Now we apply the \eqref{CkGrid1}, \eqref{CkGrid2} from the definition of a $k$-grid to get
	$$
	\int_{(-1,1)^2}|Dh_{k,m}|^p \leq C\int_{(-1,1)^2}|Du|^p
	$$ 
	for all $k$ and $m$, and $C$ independent of parameters. This suffices to prove that $h_{k,m}$ is bounded in $W^{1,p}$ and (thanks to the boundedness of our maps) it suffices to find a sequence $m_k \to \infty$ such that $h_{k,m_k} \to u$ in $L^1$ in order to prove $h_{k,m_k} \deb u$ in $W^{1,p}$.
	
	Let $\epsilon >0$ be fixed and arbitrary. We have $\phi_{k,m} \sto u_{\rceil \Gamma_k}$ as $m\to \infty$. Therefore, for each $k\in \en$, we find an $m(k, \epsilon)$, such that $|\phi_{k,m(k, \epsilon)}(z) - u_{\rceil \Gamma_k}(z)|< \epsilon$ for all $z\in \Gamma_k$. By the integrability of  $|Du|$ we find a $k\in \en$ such that
	$$
	\int_{A} \int_{-1}^{1}|\partial_1u(x,y)| \, dx dy < \epsilon^2
	$$
	and
	$$
	\int_{A} \int_{-1}^{1}|\partial_2u(x,y)| \, dydx  < \epsilon^2
	$$
	for any $A\subset (-1,1)$ with $\mathcal{L}^1(A)\leq 2^{-k}$. By \eqref{CkGrid1} and \eqref{CkGrid2} with $q=1$, we may assume that
	$$
	\int_{-1}^1|\partial_1u(x,y_{k,j})| dx \leq C2^{k}\int_{y_{k,j}-2^{-k-4}}^{y_{k,j}+2^{-k-4}} \int_{-1}^{1}|Du|(x,y) \, dx dy  <C2^k\epsilon^2
	$$
	and 
	$$
	\int_{-1}^1|\partial_2u(x_{k,i},y) |dy \leq C2^{k}\int_{x_{k,i}-2^{-k-4}}^{x_{k,i}+2^{-k-4}} \int_{-1}^{1}|Du|(x,y) \, dy dx  <C2^k\epsilon^2.
	$$
	
	We have that
	$$
	\int_{-1}^1|\partial_2u(x_{k,i},y) |dy= \sum_{j=0}^{2^k-1} \int_{-y_{k,j}}^{y_{k,j+1}}|\partial_2u(x_{k,i},y) |dy<C2^k\epsilon^2
	$$
	it follows that there are at most $\left \lfloor {C2^k\epsilon}\right \rfloor$ intervals $[y_j, y_{j+1}]$ such that 
	$$
	\int_{y_{j}}^{y_{j+1}}|\partial_2u(x_{k,i},y) |> \epsilon.
	$$
	This means that for any fixed $i$ the number of rectangles $K_{k,i,j}$  with a vertical side whose image has length greater than $\epsilon$ is at most $\left \lfloor{C2^k\epsilon}\right \rfloor$, where $\left \lfloor{a}\right \rfloor$ denotes the integer part of $a$.
	The same holds for horizontal rows. The number of rows and columns is $2^k$ each having height resp. width approximately $2^{-k}$. Then  there exists a $C>0$ such that
	\begin{equation}\label{SmallOscillation}
		\text{ there exist at most $C2^{2k}\epsilon$ rectangles $K_{k,i,j}$ of $\Gamma_k$ such that }\diam u_{\rceil \partial K_{k,i,j}} \geq 4\epsilon.
	\end{equation}
	On the these \emph{bad} rectangles $K_{k,i,j}$ of $\Gamma_k$ we can use the following estimate which is always true thanks to $u\in W^{1,p}_{\id}((-1,1)^2, \er^2)$ and $u$ satisfies the $\INV$ condition
	\begin{equation}\label{BigOscillation}
		\diam u(\partial K_{k,i,j}) < 4.
	\end{equation}
	
	In the following we use $|h_{k,m(k, \epsilon)}(x,y) - u(x,y)|< 12\epsilon$ on \emph{good} rectangles, for which it holds that $\diam u(K_{k,i,j}) \leq 4\epsilon$. On the \emph{bad} rectangles $K_{k,i,j}$ where $\diam u(K_{k,i,j}) \geq \epsilon$  we use the fact that $\mathcal{L}^2(K_{k,i,j})\leq C2^{-2k}$, the estimate from \eqref{SmallOscillation} and \eqref{BigOscillation}. Altogether we get
	$$
	\|h_{k,m(k,\epsilon)} - u\|_{L^1((-1,1)^2)} \leq \sum_{i=0}^{M_{k}^x-1} \sum_{j=0}^{M_{k}^y-1}\|h_{k,m(k,\epsilon)} - u\|_{L^1(K_{k,i,j})} \leq C\epsilon + C\epsilon.
	$$
	Then choosing an infinitesimal sequence $\epsilon_n$ we find a sequence of homeomorphisms $h_{k_n, m_n} \to u$ in $L^1((-1,1)^2)$ with $\|Dh_{k_n,m_n}\|_p$ bounded and so $h_{k_n,m_n}\deb u$ in $W^{1,p}((-1,1)^2)$. By \cite[Theorem~B]{DPP} (i.e. the weak sequential closure of Sobolev homeomorphisms equals the strong closure) and \cite[Theorem~A]{DPP} (i.e. a limit of Sobolev homeomorphisms satisfies the NC condition) then $u$ satisfies the no-crossing condition of Definition~\ref{defNC}.
\end{proof}

\begin{remark}
	The convergence in $L^1$ argument in the proof above follows the arguments of \cite[Section 4.2]{CKR}.
\end{remark}

A strictly weaker condition than the no-crossing condition is the so-called NCL condition introduced in \cite{CPR}.
\begin{definition}[NCL functions]\label{defNCL}
	Let $u\in W^{1,1}_{\id}((-1,1)^2,\er^2)$. We say that the function $u$ satisfies the \emph{no-crossing on lines (NCL)} condition if for every injective and Lipschitz continuous $\phi:((0,1)) \to (-1,1)^2$ with $|\phi'|$constant almost everywhere on $(0,1)$ such that $u\circ\phi \in W^{1,1}((0,1),\er^2)$ and for every $\epsilon>0$ there exists a continuous and injective function $v:(0,1)\to(-1,1)^2$, such that $\|u\circ\phi-v\|_{L^\infty(0,1)}<\epsilon$.
\end{definition}

\begin{remark}
	The above condition is seemingly stronger than the condition NCLoop, where we assume that $\phi$ is an injective and Lipschitz mapping of $\mathbb{S}$. On the other hand, for mappings in $W^{1,1}_{\id}((-1,1)^2,\er^2)$ it is not hard to prove that these conditions are in fact equivalent. NCLoop implies NCL by simply using half of $\mathbb{S}$ to draw the original curve we were interested in and then joining up the ends with the other half. Given an injective Lipschitz loop $\phi:\mathbb{S} \to \er^2$ we prove that $u$ can be `injectified' on $\phi(\mathbb{S})$ by injectifying it on $\tilde{\phi}$ which is a curve that starts at $\partial(-1,1)^2$ goes to a point on $\phi(\mathbb{S})$ goes almost all the way around $\phi(\mathbb{S})$ and then follows the first part of $\tilde{\phi}$ back to the boundary. Since we may assume that $u$ is equal to the identity on a neighborhood of the boundary, the existence of injective approximations of $u_{\rceil \tilde{\phi}([0,1])}$ suffices to find injective approximations of $u_{\rceil \phi(\mathbb{S})}$.
\end{remark}

We refer to the following property as quasi-monotonicity. We prove that it is equivalent with the NC condition.
\begin{definition}\label{DefQuasiMonotne}
	Let $X,Y\subset \er^2$ be compact connected sets and let $u:X\to Y$ and $S_u \subset X$ with $\H^1(S_u) = 0$ be such that $u$ is continuous on $X\setminus S_u$. We say that $u$ is a quasi-monotone map from $X$ onto $Y$ if for every $\delta>0$ there exists an open set $U_{\delta}$ satisfying $\H^1_{\delta}(U_{\delta})<\delta$ and a monotone map $v_{\delta}$ (in the sense of Definition~\ref{defMonotone}) of $X$ onto $Y$ such that $u_{\rceil X\setminus U_{\delta}} = [v_{\delta}]_{\rceil X\setminus U_{\delta}}$.
\end{definition}

We call the following the QM condition. We prove that it is equivalent with the NC condition.
\begin{definition}[Quasi-monotone condition]\label{defQM}
	Let $u\in W^{1,1}_{\id}((-1,1)^2,\er^2)$ satisfy the $\INV$ condition and for $\delta >0$ let $U_{\delta}$ be open sets satisfying $\H^1_{\delta}(U_{\delta})<\delta$ such that $u$ is continuous on $(-1,1)^2\setminus U_{\delta}$. We say that $u$ satisfies the \emph{quasi-monotone (QM)} condition if for every $r>0$, for every $\delta>0$ and for every pair of disjoint closed connected sets $C_1,C_2\subset (-1,1)^2$ such that $(-1,1)^2\setminus (C_1\cup C_2)$ is simply connected, there exists a pair of Jordan domains $V_1,V_2 \subset \er^2$ such that 
	\begin{enumerate}
		\item $\overline{V_1}\cap\overline{V_2} = \emptyset$,
		\item $\overline{V_i}\subset u^{-1}(C_i + B(0,r))\cup U_{\delta}$ and
		\item $u^{-1}(C_i)\setminus U_{\delta}\subset V_i$.
	\end{enumerate}
\end{definition}
\begin{remarks}\label{QMRemarks}$\empty$
	
	\begin{enumerate}
		\item The QM condition is a kind of separation condition of the preimages of sets $C_1$ and $C_2$. Suppose that $u^{-1}(C_i)$ were both injective Lipschitz curves, then we could rephrase the QM condition as $u^{-1}(C_1)$ does not cross $u^{-1}(C_2)$. Let us be more specific, let $\delta>0$ be fixed, then there exists an $r_{\delta}>0$ such that 
		$$
			u^{-1}(C_1+B(0,r_{\delta}))\cap u^{-1}(C_2+B(0,r_{\delta})) \subset U_{\delta}
		$$
		because $u$ is uniformly continuous on $[-1,1]^2\setminus U_{\delta}$. Then the QM condition is equivalent with the existence of two curves $\gamma_i \subset u^{-1}(C_i+B(0,r))\cup U_{\delta}$, $i=1,2$, $r\in(0,r_{\delta})$ such that $\gamma_1\cap \gamma_2 = \emptyset$ and $\gamma_i$ intersects every component of $u^{-1}(C_i+B(0,r)) \setminus U_{{\delta}}$ for $r\in (0,r_{\delta})$ and $\delta>0$.
		\item It suffices to verify the above condition for a single $r>0$ such that $C_1+B(0,r) \cap C_2+B(0,r) = \emptyset$.
		\item The QM condition is also satisfied as soon as one can finds $V_1$ with $\overline{V_1}\cap u^{-1}(C_2) = \emptyset$ although this condition is not necessary.
		\item Suppose that $C_1,C_2, \dots, C_n$ are closed connected sets such that $(-1,1)^2\setminus \bigcup_{i=1}^n C_i$ is connected. Then there is a closed connected set $\tilde{C}_2\supset \bigcup_{i=2}^nC_i$. The QM condition used on the sets $C_1$ and $\tilde{C}_2$ means it is possible to disconnect $u^{-1}(C_1)$ from all $C_2,C_3,\dots C_n$ simultaneously. Thus we see the condition as defined above with pairs $C_1, C_2$ is equivalent with the seemingly stronger condition when we separate $u^{-1}(C_1)$, $u^{-1}(C_2)$, $\dots$ $u^{-1}(C_n)$ simultaneously with domains $V_1, V_2,\dots, V_n$. 
	\end{enumerate}
\end{remarks}

Another condition that characterizes $\overline{H\cap W_{\id}^{1,p}((-1,1)^2, \er^2)}$ is the \textit{three curve condition}. In order to formulate this property, let us make the following definitions.
\begin{definition}\label{AppropriateCouple}
	Let $u\in W_{\id}^{1,p}((-1,1)^2, \er^2)$ such that for all $\delta>0$ there is an open set $U_{\delta}$ with $\H^1_{\delta}(U_{\delta})< \delta$ and $u_{\rceil [-1,1]^2\setminus U_{\delta}}$ is continuous on $[-1,1]^2\setminus U_{\delta}$. Let $\varphi:[0,1] \to [-1,1]^2\setminus U_{\delta}$ for some $\delta>0$ and $\psi:[0,1] \to [-1,1]^2$ be Lipschitz continuous and injective such that
	\begin{enumerate}
		\item there exists $\psi'(t)$ at every $t\in \psi^{-1}\circ u\circ\varphi([0,1])$,
		\item there exists $\phi'(s)\neq 0$ and $(u\circ\varphi)'(s)$ at every $s\in(u\circ\varphi)^{-1}\circ\psi([0,1])$,
		\item the sets $\{t \in [0,1]:\psi(t) \in  u\circ\varphi([0,1])\}$ and $\{s \in [0,1]: u\circ\varphi(s) \in \psi([0,1])\}$ are finite and
		\item for any $s,t \in [0,1]$ such that $\psi(t) =  u\circ\varphi(s)$ the vectors $\psi'(t)$ and $(u\circ\varphi)'(s)$ are linearly independent.
	\end{enumerate}
	Then we say the couple $\varphi$ and $\psi$ are a \textit{pair of good test curves for $u$}.
\end{definition}	
We call the property that characterizes limits of homeomorphisms the three-curve property. Loosely speaking it means that whenever we have $\varphi, \psi$ a pair of good test curves for $u$ that we can construct an injective curve $\gamma$ that meets $\varphi$ exactly at the points that $u\circ\varphi$ meets $\psi$ and does so in the correct order.
\begin{definition}\label{ThreeCurvePropertyDef}
	Let $u\in W_{\id}^{1,p}((-1,1)^2, \er^2)$ be such that for all $\delta>0$ there is an open set $U_{\delta}$ with $\H^1_{\delta}(U_{\delta})< \delta$ and $u_{\rceil [-1,1]^2\setminus U_{\delta}}$ is continuous on $[-1,1]^2\setminus U_{\delta}$. We say that $u$ satisfies the \textit{three curve condition} if for every pair of good test curves $\varphi, \psi$ for $u$ there is an injective Lipschitz continuous $\gamma:[0,1]\to [-1,1]^2$ and there is a finite set of pairwise disjoint intervals $\{I_{t};t \in \psi^{-1}\circ u\circ\varphi([0,1])\}$ such that $\gamma(I_t) \cap \varphi([0,1]) = \varphi^{-1}\circ u^{-1}\circ\psi(t)$ satisfying 
	\begin{equation}\label{IntervalOrder}
		\text{$s<t$ for all $s\in I_{t_1}$ and all $t\in I_{t_2}$ if $t_1<t_2$.}
	\end{equation}
	Also, the derivative $\gamma'(t)$ exists at each $t \in \gamma^{-1}\circ u\circ\varphi([0,1])$
	\begin{equation}\label{GammaCross}
		\text{and the vectors $\gamma'(t)$ and $\phi'(\phi^{-1}\circ\gamma(t))$ are linearly independent.}
	\end{equation}
\end{definition}

\section{Proof of monotonicity theorems}
Theorem~\ref{partial} is an intermediary step towards proving Theorem~\ref{QuasiMonotne} which may be of independent interest since it is a different characterization of limits of Sobolev homeomorphisms.
\begin{thm}\label{partial}
	Let $p\in [1,\infty)$, let $u\in W^{1,p}_{\id}((-1,1)^2, \er^2)$. Then
	\begin{enumerate}
		\item[a)] if $u\in \overline{H\cap W_{\id}^{1,p}((-1,1)^2, \er^2)}$ then for any $\mathfrak{G} = \{\Gamma_k\}_{k\in \en}$ suitable grid system for $u$ there exists a sequence of monotone mappings $\{g_k\}_{k\in \en}$, $g_k:[-1,1]^2\to \er^2$ such that $[g_k]_{\rceil \Gamma_k} = u_{\rceil\Gamma_k}$.
		\item[b)] Conversely if there exists a $\mathfrak{G} = \{\Gamma_k\}_{k\in \en}$ suitable grid system for $u$ and a sequence of monotone mappings $\{g_k\}_{k\in \en}$, $g_k:[-1,1]^2\to \er^2$ such that $[g_k]_{\rceil \Gamma_k} = u_{\rceil\Gamma_k}$ then $u\in \overline{H\cap W_{\id}^{1,p}((-1,1)^2, \er^2)}$.
	\end{enumerate}
\end{thm}

\begin{proof}[Proof of Theorem~\ref{partial}]
	
	Let us have any suitable grid system $\mathfrak{G} = \{\Gamma_k\}_{k\in\en}$ for $u$. Let us assume that there are monotone $g_k:[-1,1]^2\to \er^2$ coinciding with $u$ on $\Gamma_k$. By Youngs' theorem (Theorem~\ref{Young}), there exist homeomorphisms $\phi_{k,m}:[-1,1]^2 \to \er^2$ converging uniformly to $g_k$ on $[-1,1]^2$. Then the continuous injective $[\phi_{k,m}]_{\rceil \Gamma_k} \sto u_{\rceil \Gamma_k}$ on $\Gamma_k$. Therefore $u$ satisfies the no-crossing condition of Definition~\ref{StayCalmPlease} and by Proposition~\ref{NoCrossEquivalence} and \cite[Theorem A]{DPP} (i.e. if $u$ satisfies the NC condition then $u$ is a limit of homeomorphisms) $u\in \overline{H\cap W^{1,p}_{\id}((-1,1)^2, \er^2)}$.

	Alternatively let $u\in \overline{H\cap W^{1,p}_{\id}((-1,1)^2, \er^2)}$. Then $u$ satisfies the no-crossing condition (of Definition~\ref{StayCalmPlease}); there exists a suitable grid system $\mathfrak{G} = \{\Gamma_k\}_{k\in \en}$ for $u$ and injective continuous $\phi_{k,m} \sto u_{\rceil \Gamma_k}$ on $\Gamma_k$. By Remark~\ref{WeakRemark} we may assume that $\phi_{k,m}$ are also piecewise linear and $\phi_{k,m} \xdeb{m\to \infty} u_{\rceil \Gamma_k}$ in $W^{1,p}(\Gamma_k)$. Therefore $|D_{\tau}\phi_{k,m}|$ is uniformly integrable on $\Gamma_k$ with respect to $\H^1_{\rceil \Gamma_k}$. Then using an affine change of variables and Corollary~\ref{SpecialExtension} on each rectangle $K_{k,i,j} = [x_{k,i},x_{k,i+1}]\times[y_{k,j},y_{k,j+1}]$ of $\Gamma_k$ separately we get a sequence of uniformly bounded equicontinuous homeomorphisms $h_{k,m} \in W^{1,p}_{\id}((-1,1)^2, \er^2)$ with bounded norm in $W^{1,p}$. Then by the Arzela-Ascoli theorem we may assume that there exists some $g_k$ such that $h_{k,m} \sto g_k$ on $[-1,1]^2$ as $m\to \infty$. Therefore, by Youngs' theorem $g_k$ is monotone. On the other hand because $[h_{k,m}]_{\rceil \Gamma_k} = \phi_{k,m} \sto u_{\rceil \Gamma_k}$ on $\Gamma_k$ we have that $[g_k]_{\rceil \Gamma_k} = u_{\rceil \Gamma_k}$.
	
\end{proof}

The proof of the equivalence between points $a)$ and $b)$ of Theorem~\ref{MainTheorem} are proved by the following theorem. 
\begin{thm}\label{QuasiMonotne}
	Let $p\in [1,2]$ and let $u\in W^{1,p}_{\id}((-1,1)^2, \er^2)$.
	\begin{enumerate}
		\item[a)] If for every $\delta>0$ there is a monotone map $g_{\delta}:[-1,1]^2 \to \er^2$ satisfying
		$$
		\mathcal{H}^{1}_{\delta}\Big( \big\{(x,y) \in [-1,1]^2: \ u(x,y)\neq g_{\delta}(x,y)\big\}\Big)<\delta
		$$
		then $u \in  \overline{H\cap W_{\id}^{1,p}((-1,1)^2, \er^2)}$.
		\item[b)] Conversely if $u \in  \overline{H\cap W_{\id}^{1,p}((-1,1)^2, \er^2)}$ then for every $\delta>0$ there is a monotone $g_{\delta}:[-1,1]^2 \to \er^2$ such that
		$$
		\mathcal{H}^{2-p}_{\delta}\Big( \big\{(x,y) \in [-1,1]^2: \ u(x,y)\neq g_{\delta}(x,y)\big\}\Big)<\delta.
		$$
	\end{enumerate}
	Especially a map $u\in W^{1,p}_{\id}((-1,1)^2, \er^2)$ lies in $\overline{H\cap W_{\id}^{1,p}((-1,1)^2, \er^2)}$ if and only if it is quasi-monotone on $[-1,1]^2$.
\end{thm}
\begin{proof}
	Let us assume that for each $k\in \en$ there exists a monotone $g_k:[-1,1]^2 \to \er^2$ such that $\H^{1}_{2^{-k-3}}(\{u\neq g_k\})< 2^{-k-3}$. It follows that $$\mathcal{L}^1(\pi_1(\{u\neq g_k\})),\mathcal{L}^1(\pi_2(\{u\neq g_k\}))<2^{-k-3}$$ where $\pi_1(x,y) =x$ and $\pi_2(x,y) = y$. Then for each $k$ and $i$ the set $E_{k,i}^1$ from Proposition~\ref{GridExist} has $E_{k,i}^1 \setminus \pi_1(\{u\neq g_k\}) \neq \emptyset$. Similarly, $E_{k,j}^1 \setminus \pi_2(\{u\neq g_k\}) \neq \emptyset$. Thus we construct $\mathfrak{G} = \{\Gamma_k\}_{k\in \en}$ a suitable grid system for $u$ and a sequence $g_k$ of monotone maps on $[-1,1]^2$ such that $u_{\rceil \Gamma_k} = [g_{k}]_{\rceil \Gamma_k}$. Thus, by Theorem~\ref{partial}, $u\in\overline{H\cap W^{1,p}_{\id}((-1,1)^2, \er^2)}$.

	Let $\delta >0$ be fixed and let us choose $k_0 \in \en$ such that $2^{-k_0}<2^{-5(2-p)}\delta$. Using Theorem~\ref{HomImpINV} and Proposition~\ref{BadSize} we find squares $\{Q_{l}\}_{l\in \en} = \{Q((a_{l},b_{l}), r_{l})\}_{l\in \en} $ with $2r_{l}<2^{-k_0-3}< \delta$ such that $U_{\delta} = \bigcup_l Q_l \supset S_u$ and such that $\sum_{l=1}^{\infty}(2r_{l})^{2-p}<2^{-k_0-3}$. Our aim now is to construct an increasing sequence of general grids $\{\Gamma_k\}_{k\in \en}$ each of which satisfies the hypothesis in Definition~\ref{defNC} but does not intersect $U_{\delta}$. Once we have constructed the grids we can apply our extension results (especially Corollary~\ref{SpecialExtension}) to construct the desired $g_{\delta}$.
	
	We construct $\Gamma_0, \subset \Gamma_1\subset \dots \subset\Gamma_{k_0}$ using Proposition~\ref{GridExist} and since
	$$
	\mathcal{L}^1\big(\pi_1\big(\bigcup_l Q_{l}\big)\big), \mathcal{L}^1\big(\pi_2\big(\bigcup_l Q_{l}\big)\big) < 2^{-k_0-3}<\mathcal{L}^1(E_{k_0,i}^1),\mathcal{L}^1(E_{k_0,j}^2),
	$$
	we can assume that $\Gamma_{k_0} \cap \bigcup_l Q_{l} = \emptyset$. For $m=1,\dots, k$, $\Gamma_m$ is (or will be) the finite union of horizontal and vertical segments.  By rectangles of $\Gamma_k$ we refer to the closure of the components of $[-1,1]^2 \setminus \Gamma_k$.

	\begin{figure}
		\begin{tikzpicture}[line cap=round,line join=round,>=triangle 45,x=0.6cm,y=0.6cm]
			\clip(0.46292119324500214,-3.348894937565878) rectangle (14.351973658953773,7.3097639794479425);
			\fill[line width=0.7pt,fill=black!10] (3.6,7.) -- (3.6,1.) -- (4.4,1.) -- (4.4,7.) -- cycle;
			\fill[line width=0.7pt,fill=black!10] (6.6,7.) -- (6.6,1.) -- (7.4,1.) -- (7.4,7.) -- cycle;
			\fill[line width=0.7pt,fill=black!10] (1.,4.4) -- (10.,4.4) -- (10.,3.6) -- (1.,3.6) -- cycle;
			\fill[line width=0.7pt,fill=black!10] (10.,4.4) -- (14.,4.4) -- (14.,3.6) -- (10.,3.6) -- cycle;
			\fill[line width=0.7pt,fill=black!10] (3.6,1.) -- (3.6,-3.) -- (4.4,-3.) -- (4.4,1.) -- cycle;
			\fill[line width=0.7pt,fill=black!10] (6.6,1.) -- (6.6,-3.) -- (7.4,-3.) -- (7.4,1.) -- cycle;
			\fill[line width=0.7pt,fill=black!20] (3.6,4.4) -- (3.6,3.6) -- (4.4,3.6) -- (4.4,4.4) -- cycle;
			\fill[line width=0.7pt,fill=black!20] (6.6,4.4) -- (6.6,3.6) -- (7.4,3.6) -- (7.4,4.4) -- cycle;
			\draw [line width=0.7pt] (1.,7.)-- (1.,1.);
			\draw [line width=0.7pt] (1.,1.)-- (10.,1.);
			\draw [line width=0.7pt] (10.,1.)-- (10.,7.);
			\draw [line width=0.7pt] (10.,7.)-- (1.,7.);
			\draw [line width=0.7pt,dash pattern=on 2pt off 2pt] (3.7999053315484126,1.)-- (3.7999053315484126,7.);
			\draw [line width=0.7pt,dash pattern=on 2pt off 2pt] (7.264037077386787,7.)-- (7.264037077386787,1.);
			\draw [line width=0.7pt,dash pattern=on 2pt off 2pt] (1.,4.2437479442183825)-- (10.,4.2437479442183825);
			\draw [line width=0.7pt] (1.,1.)-- (1.,-3.);
			\draw [line width=0.7pt] (1.,-3.)-- (10.,-3.);
			\draw [line width=0.7pt] (10.,-3.)-- (10.,1.);
			\draw [line width=0.7pt] (10.,1.)-- (14.,1.);
			\draw [line width=0.7pt] (14.,1.)-- (14.,-3.);
			\draw [line width=0.7pt] (14.,-3.)-- (10.,-3.);
			\draw [line width=0.7pt] (14.,7.)-- (14.,1.);
			\draw [line width=0.7pt] (14.,7.)-- (10.,7.);
			\draw [line width=0.7pt,dash pattern=on 2pt off 2pt] (4.242393457496349,1.)-- (4.242393457496349,-3.);
			\draw [line width=0.7pt,dash pattern=on 2pt off 2pt] (6.66492019031439,1.)-- (6.66492019031439,-3.);
			\draw [line width=0.7pt,dash pattern=on 2pt off 2pt] (10.,4.)-- (14.,4.);
		\end{tikzpicture}
		\caption{Given rectangles of $\Gamma_k$, we make $\Gamma_{k+1}$ by dividing them into smaller rectangles each of which have side length approximately $2^{-k-1}$. In some cases we split a side in three parts, in some cases we split it in two and if it is short, we do not split it.}\label{fig:Patchwork}
	\end{figure}
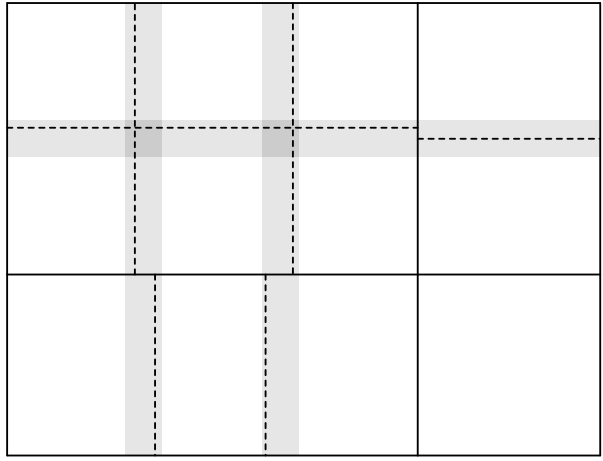

	We define $\Gamma_k$ inductively following the approach of the proof of Proposition~\ref{GridExist} but instead of working on the whole of $(-1,1)^2$ at once, we define $\Gamma_{k+1}$ by refining the grid $\Gamma_k$ in each of its rectangles $K$ individually. Let $K = [x,\tilde{x}]\times[y,\tilde{y}]$ be a rectangle in $\Gamma_{k}$ for some $k\geq k_0$. There are several cases of how we refine the grid inside $K$ depending on the dimensions of $K$ and the size of $\bigcup_l Q_{l}\cap K$, we either keep $K$ the same or we split it into 2,3,4,6 or 9 smaller rectangles as is depicted in Figure~\ref{fig:Patchwork}.
	
	Firstly, if $\sum_{Q_l\subset K}r_l \geq 2^{-k-3}$ then we call the rectangle $K$ a fixed rectangle. We do not refine it in any of the further steps and $K$ is a rectangle of $\Gamma_m$ for all $m\geq k$. In the following, then we can assume that $\sum_{Q_l\subset K}r_l < 2^{-k-3}$.

	If 
	\begin{equation}\label{ThreeCasesOfBeer}
		\tilde{x}-x\in
		\begin{cases}
			(2^{-k}, 2^{1-k}] \ & \text{ then we do not split } K \text{ by any vertical line}\\
			(2^{1-k}, \tfrac{3}{2} 2^{1-k}] \ & \text{ then we split } K \text{ by one vertical line}\\
			(\tfrac{3}{2} 2^{1-k}, 2^{2-k})\ & \text{  then we  split } K \text{ by two vertical lines}.
		\end{cases}
	\end{equation}
	In the second case we call 
	$$
	F =  \Big(\frac{\tilde{x}+x}{2} - 2^{-k-3},  \frac{\tilde{x} +x}{2} + 2^{-k-3}\Big) \cap G_x \setminus \pi_1(\bigcup_l Q_{l})
	$$
	and then $\mathcal{L}^1(F)  > 2^{-k-3}$. For all $t\in F$ it holds that 
	\begin{equation}\label{widther}
		[\tilde{x} - t], [t- x] \in (2^{-k}-2^{-k-3}, \tfrac{3}{2} 2^{-k} +2^{-k-3} ) \subset (2^{-k-1}, 2^{1-k}).
	\end{equation}
	Moreover, we find a set $E\subset F$ with $\mathcal{L}^1(E) > 2^{-k-4}$ such that the estimates
	\begin{equation}\label{Sobostimate1}
		\oint_{y}^{\tilde{y}}| \partial_2u(t,s)|  \ d\H^1(s)\leq \hat{C} \oint_{K}| Du|d\mathcal{L}^2
	\end{equation}
	and
	\begin{equation}\label{Sobostimate2}
		\oint_{y}^{\tilde{y}}| \partial_2u(t,s)|^p  \ d\H^1(s)\leq \hat{C} \oint_{K}| Du|^p \ d\mathcal{L}^2
	\end{equation}
	hold for all $t\in E$ with some absolute constant $\hat{C}$.
	
	In the third case of \eqref{ThreeCasesOfBeer} the approach is similar, and the details can be easily extrapolated from the proof of Proposition~\ref{GridExist}. We refine $K$ with horizontal lines with the corresponding properties and estimates. It is easy to observe that the grids $\Gamma_0, \Gamma_1,\dots $ are general grids for which $u_{\rceil \Gamma_k} \in W^{1,p}(\Gamma_{k}, \er^2)$.
	
	Let us call $\mathfrak{R}_k$ be the set of all rectangles of $\Gamma_k$. Let us call $\mathfrak{F}_k$ the set of all fixed rectangles of $\Gamma_k$ and $\mathfrak{F} = \bigcup_k \mathfrak{F}_k$. We have $\bigcup_l Q_{l} \subset\bigcup_{K\in \mathfrak{F}}K$. Further, using \eqref{widther} for $K\in\mathfrak{F}_k\setminus \mathfrak{F}_{k-1}$, we have
	\begin{equation}\label{ComparisionFord}
		\diam K < 2^{2-k} \text{ and } \sum_{Q_l \subset K}r_l \geq 2^{-k-3} \text{ and therefore } \diam K \leq 2^{5}\sum_{Q_l \subset K}r_l .
	\end{equation}
	Let us estimate $\sum_{K\in \mathfrak{F}_k}(\diam K)^{2-p}$. By \eqref{ComparisionFord} and the choice of $r_l$ we have
	$$
	\sum_{K\in \mathfrak{F}_k}(\diam K)^{2-p} \leq \sum_{{K\in \mathfrak{F}}}\Big(2^{5}\sum_{Q_l \subset K_i}r_l \Big)^{2-p} \leq 2^{5(2-p)}\sum_{l=1}^{\infty}r_l^{2-p} \leq 2^{5(2-p)-3-k_0}<\delta
	$$
	for all $k\in \en$ thanks to the choice of $k_0$. We denote $U = U_{\delta}: = \bigcup_{K\in \mathfrak{F}} \inter K$, where $\inter K$ is the topological interior of $K$. We have $\H^{2-p}_{\delta}(U)<\delta$.
	
	By the choice of $\Gamma_k$ especially \eqref{Sobostimate1} and \eqref{Sobostimate2} we have $u \in W^{1,p}(\Gamma_k, \er^2)$. Because $u \in  \overline{H\cap W_{\id}^{1,p}((-1,1)^2, \er^2)}$, $u$ satisfies the no-crossing condition of Definition~\ref{defNC} and there exist continuous injective $\tilde{\phi}_{k,m} \sto u_{\rceil \Gamma_k}$ on $\Gamma_k$ for each $k$. By Proposition~\ref{PiecewiseLinearization} we can replace $\tilde{\phi}_{k,m} $ with continuous injective piecewise linear $\phi_{k,m} \deb u_{\rceil \Gamma_k}$ in $W^{1,p}(\Gamma_k, \er^2)$. We may assume that $\|\phi_{k,m} - u_{\rceil\Gamma_k}\|_{L^{\infty}(\Gamma_k)}\leq m^{-1}$ and that $\phi_{k,m} = [\phi_{m,m}]_{\rceil \Gamma_k}$ for all $m\geq k$. We use Corollary~\ref{SpecialExtension} on each $K$ rectangle of $\Gamma_k$ to get a sequence of homeomorphisms $h_{k,m, K} : K\to \er^2$ which is equicontinuous and $h_{k,m,K}= \phi_{k,m}$ on $\partial K$. Note that the map defined by $h_{k,m}(x,y) = h_{k,m,K}(x,y)$ whenever $(x,y)\in K$ is a well-defined homeomorphism on $[-1,1]^2$ and for each $k$ the sequence of mappings $h_{k,m}$ is equicontinuous on $[-1,1]^2$. Then, for each $k$, (up to taking a subsequence) by the Arsela-Ascoli theorem we find a $h_k$ such that $h_{k,m} \sto h_k$ uniformly on $[-1,1]^2$. By the Young's theorem (Theorem~\ref{Young}) $h_k$ is monotone on $[-1,1]^2$. Further, without loss of generality we assume for each $K\in \mathfrak{F}_{o}$ that $h_{k} = h_{o}$ on $K$ for all $k\geq o$ because for every $k_1,k_2$ and every $m\geq \max\{k_1,k_2\}$ we have $\phi_{k_1,m} =\phi_{k_2,m} = \phi_{m,m}$.
	
	Because $S_u \subset U_{\delta}$ and because $U_{\delta}$  is open, $u$ is uniformly continuous on $[-1,1]^2\setminus U_{\delta}$. Therefore, because $\Gamma_k \cap U_{\delta} = \emptyset$ for each $k\in \en$, we have that $[h_k]_{\rceil \Gamma_m}$ is equicontinuous for each $m$. For every $K \in \mathfrak{R}_k \setminus \mathfrak{F}_k$ we have $\diam K < 2^{2-k}$. Therefore, for each $\epsilon>0$ we find a $k$ such that for all $m\geq k$ it holds that $\diam h_m(K) = \diam u(K) = \diam u(\partial K) < \epsilon$ for all $K \in \mathfrak{R}_k \setminus \mathfrak{F}_k$. Therefore, for all $x\in \bigcup_{K\in \mathfrak{R}_k \setminus \mathfrak{F}_k}K$, we have $|h_m(x) - h_n(x)|\leq 2\epsilon$ for all $m,n\geq k$. Simultaneously however, $h_m = h_n$ on $\bigcup_{K\in\mathfrak{F}_{k}}K$ for $m,n \geq k$. Therefore $h_k$ is a Cauchy sequence in $\mathcal{C}^0([-1,1]^2)$ and there is some $g_{\delta}$ such that $h_k \sto g_{\delta}$. By Youngs' theorem (Theorem~\ref{Young}) monotone maps are closed under uniform convergence and therefore we have that $g_{\delta}$ is monotone. Since $[h_k]_{\rceil [-1,1]^2\setminus U_{\delta}} \sto u_{\rceil [-1,1]^2\setminus U_{\delta}} $ we have $u_{\rceil [-1,1]^2\setminus U_{\delta}} = [g_{\delta}]_{\rceil [-1,1]^2\setminus U_{\delta}}$.	 	
\end{proof}
\begin{remark}
	Using the estimate \eqref{Sobostimate2} it is not difficult to prove that $[h_k]_{\rceil (-1,1)^2\setminus U_{\delta}}$ is bounded in $W^{1,p}((-1,1)^2\setminus U_{\delta},\er^2)$.
\end{remark}

\section{The three-curve condition characterization}
We start this section by proving the following lemma.
\begin{lemma}\label{DegreesAreUseful}
	Let $u\in W_{\id}^{1,p}((-1,1)^2, \er^2)$ and let $S_u \subset (-1,1)^2$ with $\H^1(S_u) = 0$ be such that $u_{\rceil (-1,1)^2\setminus S_u}$ is continuous on $(-1,1)^2\setminus S_u$ and let $u$ satisfy the three curve condition. Let $\Gamma_k$ be a suitable $k$-grid for $u$ avoiding $S_u$. Then for every $(a,b) \in (-1,1)^2 \setminus u(\Gamma_k)$ and every $R = [x_i,x_{i+1}]\times[y_j,y_{j+1}]$ it holds that 
	$$
	\deg((a,b), u , \partial R) \in \{0,1\}.
	$$
\end{lemma}
\begin{remark}
	The condition that the degree for a sense preserving mapping $u$ has values in $\{0,1\}$ is a necessary condition for $u$ to satisfy the $\INV$ condition. This can be observed from Lemma~\ref{IDidThisOne} and standard  degree arguments. The reader may also reference \cite[Lemma~2.8]{DPP}.
\end{remark}
\begin{proof}[Proof of Lemma~\ref{DegreesAreUseful}]
	Clearly $\deg((a,b), u , \partial(-1,1)^2) = \deg((a,b), \id, \partial(-1,1)^2)= 1$ for all $(a,b) \in (-1,1)^2$. Recall our notation $\theta(t) = (\cos(t),\sin(t))$. For any Jordan domain $G$ whose boundary is the image of an injective Lipschitz curve $\phi:\mathbb{S}\to \er^2$ with $|(\phi\circ \theta)'|$ constant, such that $u\circ\phi\circ\theta\in W^{1,1}_{\loc}(\er,\er^2)$ we know that $\deg((a,b), u, G) = \operatorname{Wind}((a,b), u\circ\phi)$ where $\operatorname{Wind}$ is the winding number. We have the additivity of the degree i.e. $\sum_R\deg((a,b), u , \partial R)=1$ where we sum over all rectangles of type $R = [x_i,x_{i+1}]\times[y_j, y_{j+1}]$. Therefore it suffices to prove that $\deg((a,b), u , \partial R) \geq 0$. Since we may assume that $(a,b)$ is a regular value of $u$ it suffices to prove that there is no $(x,y)\in (-1,1)^2$ and no $r>0$ such that $\deg(u(x,y), u, B((x,y), r) ) = -1$. We may assume that there is some $\xi >0$ such that $u(x,y) =(x,y)$ for all $(x,y)$ satisfying $\dist((x,y), \partial(-1,1)^2) < \xi$.
	
	Let us assume that we have an $(x_0,y_0)$ and $r>0$ such that $\deg(u(x_0,y_0), u, B((x_0,y_0), r) ) = \det Du(x_0,y_0) = -1$. In the following we construct a good test pair for $u$. In order to help our construction of $\phi$, we find a segment $T$ from $\partial(-1,1)^2$ to $\partial B((x_0,y_0), \rho)$. The image of the curve $\phi$ is made of a segment very close to $T$ an arc in $\partial B((x_0,y_0), \rho)$ and then back to the close to the boundary along another segment close to $T$. This can be seen in Figure~\ref{fig:ZadniPsi} in the lower part of the diagram. We may assume that $\phi$ maps $[0,1]$ onto its image at constant speed and we can choose $\rho$ and the segments in such a way that $u\circ\phi\in W^{1,1}((0,1),\er^2)$. The curve $\psi$ starts very close to $T$ inside $\{(s,t)\in (-1,1)^2; \dist((s,t), \partial(-1,1)^2) < \xi\}$ and stays very close to the boundary, then it goes through $u\big(B((x_0,y_0),\rho)\big)$ by a segment back into $\{(s,t)\in (-1,1)^2; \dist((s,t), \partial(-1,1)^2) < \xi\}$ and goes close to the boundary around to a point again close to the segment $T$. The reader can refer to Figure~\ref{fig:ZadniPsi} (especially the upper part of it). By choosing the segment of $\psi$ that goes through $u\big(B((x_0,y_0),\rho)\big)$ we can guarantee that $\phi$ and $\psi$ are a good test pair for $u$ by the arguments of Proposition~\ref{ArrivalGrids}. Because we may assume that $u\in W^{1,1}(T,\er^2)$ and that $\phi([0,1]) \cup T \subset (-1,1)^2\setminus S_u$ we may assume that the number of intersections of $\psi$ with the image in $u$ of the first and second segments of $\phi$ are both equal to the number of intersections of $\psi$ with $u(T)$, i.e. $\psi$ can cross over the image of both of the segments of $\phi$ but cannot cross only one of them.

	\begin{figure}[h]
		\begin{tikzpicture}[line cap=round,line join=round,>=triangle 45,x=0.9cm,y=0.9cm]
			\clip(0.0,0.) rectangle (16.,10.);
			\draw [line width=0.7pt] (0.,10.)-- (0.,6.);
			\draw [line width=0.7pt] (0.,6.)-- (4.,6.);
			\draw [line width=0.7pt] (4.,6.)-- (4.,10.);
			\draw [line width=0.7pt] (4.,10.)-- (0.,10.);
			\draw [line width=0.7pt] (6.,10.)-- (6.,6.);
			\draw [line width=0.7pt] (6.,6.)-- (10.,6.);
			\draw [line width=0.7pt] (10.,6.)-- (10.,10.);
			\draw [line width=0.7pt] (10.,10.)-- (6.,10.);
			\draw [line width=0.7pt,color=qqqqff] (6.2,7.95)-- (6.2,6.2);
			\draw [line width=0.7pt,color=qqqqff] (6.2,6.2)-- (8.,6.2);
			\draw [line width=0.7pt,color=qqqqff] (8.,6.2)-- (8.,9.8);
			\draw [line width=0.7pt,color=qqqqff] (8.,9.8)-- (6.2,9.8);
			\draw [line width=0.7pt,color=qqqqff] (6.2,9.8)-- (6.2,8.05);
			\draw [shift={(8.,8.)},line width=0.7pt,color=yqqqyq]  plot[domain=0.09966865249116451:6.181882701402558,variable=\t]({1.*0.5*cos(\t r)+0.*0.5*sin(\t r)},{0.*0.5*cos(\t r)+1.*0.5*sin(\t r)});
			\draw [line width=0.7pt,color=yqqqyq] (8.497436638788585,7.949435285121798)-- (9.05,8.45);
			\draw [line width=0.7pt,color=yqqqyq] (9.05,8.45)-- (7.6,8.6);
			\draw [line width=0.7pt,color=yqqqyq] (8.497518595104994,8.0497518595105)-- (9.14,8.48);
			\draw [line width=0.7pt,color=yqqqyq] (9.14,8.48)-- (7.6,8.65);
			\draw [line width=0.7pt,color=yqqqyq] (6.3,8.1)-- (7.6,8.65);
			\draw [line width=0.7pt,color=yqqqyq] (7.6,8.6)-- (6.3,7.9);
			\draw [line width=0.7pt] (12.,10.)-- (12.,6.);
			\draw [line width=0.7pt] (12.,6.)-- (16.,6.);
			\draw [line width=0.7pt] (16.,6.)-- (16.,10.);
			\draw [line width=0.7pt] (16.,10.)-- (12.,10.);
			\draw [line width=0.7pt] (0.,4.)-- (0.,0.);
			\draw [line width=0.7pt] (0.,0.)-- (4.,0.);
			\draw [line width=0.7pt] (4.,0.)-- (4.,4.);
			\draw [line width=0.7pt] (4.,4.)-- (0.,4.);
			\draw [line width=0.7pt] (6.,4.)-- (6.,0.);
			\draw [line width=0.7pt] (6.,0.)-- (10.,0.);
			\draw [line width=0.7pt] (10.,0.)-- (10.,4.);
			\draw [line width=0.7pt] (10.,4.)-- (6.,4.);
			\draw [line width=0.7pt] (12.,4.)-- (12.,0.);
			\draw [line width=0.7pt] (12.,0.)-- (16.,0.);
			\draw [line width=0.7pt] (16.,0.)-- (16.,4.);
			\draw [line width=0.7pt] (16.,4.)-- (12.,4.);
			\draw [line width=0.7pt,color=yqqqyq] (6.1,8.1)-- (6.3,8.1);
			\draw [line width=0.7pt,color=yqqqyq] (6.3,7.9)-- (6.1,7.9);
			\draw [line width=0.7pt,color=qqqqff] (0.2,7.95)-- (0.2,6.190763909584986);
			\draw [line width=0.7pt,color=qqqqff] (0.2,8.05)-- (0.2,9.8);
			\draw [line width=0.7pt,color=qqqqff] (0.2,9.8)-- (2.,9.8);
			\draw [line width=0.7pt,color=qqqqff] (0.2,6.190763909584986)-- (2.,6.2);
			\draw [line width=0.7pt,color=qqqqff] (2.,9.8)-- (2.,6.2);
			\draw [line width=0.7pt,color=yqqqyq] (12.1,8.1)-- (12.3,8.1);
			\draw [line width=0.7pt,color=yqqqyq] (12.3,8.1)-- (12.8,8.6);
			\draw [line width=0.7pt,color=yqqqyq] (12.1,7.9)-- (12.3,7.9);
			\draw [line width=0.7pt,color=yqqqyq] (12.3,7.9)-- (12.8,8.5);
			\draw [line width=0.7pt,color=yqqqyq] (12.8,8.6)-- (14.4,9.2);
			\draw [line width=0.7pt,color=yqqqyq] (12.8,8.5)-- (14.4,9.1);
			\draw [line width=0.7pt,color=yqqqyq] (15.1,8.4)-- (14.2,7.3);
			\draw [line width=0.7pt,color=yqqqyq] (14.2,7.3)-- (13.3,7.6);
			\draw [line width=0.7pt,color=yqqqyq] (13.3,7.6)-- (13.4,7.9);
			\draw [line width=0.7pt,color=yqqqyq] (15.2,8.4)-- (14.2,7.2);
			\draw [line width=0.7pt,color=yqqqyq] (14.2,7.2)-- (13.2,7.6);
			\draw [line width=0.7pt,color=yqqqyq] (13.2,7.6)-- (13.4,7.94);
			\draw [line width=0.7pt,color=yqqqyq] (14.4,9.2)-- (15.2,8.4);
			\draw [line width=0.7pt,color=yqqqyq] (14.4,9.1)-- (15.1,8.4);
			\draw [shift={(14.,8.)},line width=0.7pt,color=yqqqyq]  plot[domain=-2.9398459506537016:2.940206476849828,variable=\t]({1.*0.5*cos(\t r)+0.*0.5*sin(\t r)},{0.*0.5*cos(\t r)+1.*0.5*sin(\t r)});
			\draw [line width=0.7pt,color=yqqqyq] (13.4,7.94)-- (13.51014096668138,7.899809543986732);
			\draw [line width=0.7pt,color=yqqqyq] (13.4,7.9)-- (13.510104877233895,8.100013842491839);
			\draw [line width=0.7pt,color=qqqqff] (12.2,7.95)-- (12.2,6.2);
			\draw [line width=0.7pt,color=qqqqff] (12.2,6.2)-- (14.,6.2);
			\draw [line width=0.7pt,color=qqqqff] (14.,6.2)-- (14.,9.8);
			\draw [line width=0.7pt,color=qqqqff] (14.,9.8)-- (12.2,9.8);
			\draw [line width=0.7pt,color=qqqqff] (12.2,9.8)-- (12.2,8.05);
			\draw [line width=0.7pt,color=yqqqyq] (0.1,2.1)-- (1.5101515620487123,2.1002422457683556);
			\draw [line width=0.7pt,color=yqqqyq] (0.1,1.9)-- (1.5101140868294851,1.8999410569859414);
			\draw [shift={(2.,2.)},line width=0.7pt,color=yqqqyq]  plot[domain=-2.9401144143980478:2.9397402257173133,variable=\t]({1.*0.5*cos(\t r)+0.*0.5*sin(\t r)},{0.*0.5*cos(\t r)+1.*0.5*sin(\t r)});
			\draw [line width=0.7pt,color=qqqqff] (0.2,2.05)-- (0.2,3.8);
			\draw [line width=0.7pt,color=qqqqff] (0.2,3.8)-- (2.,3.8);
			\draw [line width=0.7pt,color=qqqqff] (0.2,1.95)-- (0.2,0.2);
			\draw [line width=0.7pt,color=qqqqff] (0.2,0.2)-- (2.,0.2);
			\draw [shift={(2.,1.6)},line width=0.7pt,color=ffqqqq]  plot[domain=-1.5707963267948966:1.5707963267948966,variable=\t]({1.*1.4*cos(\t r)+0.*1.4*sin(\t r)},{0.*1.4*cos(\t r)+1.*1.4*sin(\t r)});
			\draw [shift={(2.,2.4)},line width=0.7pt,color=qqqqff]  plot[domain=-1.5707963267948966:1.5707963267948966,variable=\t]({1.*1.4*cos(\t r)+0.*1.4*sin(\t r)},{0.*1.4*cos(\t r)+1.*1.4*sin(\t r)});
			\draw [line width=0.7pt,color=qqqqff] (2.,3.)-- (2.,1.);
			\draw [line width=0.7pt,color=yqqqyq] (6.1,2.1)-- (7.51012262772888,2.1001007499309696);
			\draw [line width=0.7pt,color=yqqqyq] (6.1,1.9)-- (7.510124941305595,1.8998879284543921);
			\draw [line width=0.7pt,color=qqqqff] (6.2,0.2)-- (6.2,1.95);
			\draw [line width=0.7pt,color=qqqqff] (6.2,2.05)-- (6.2,3.8);
			\draw [line width=0.7pt,color=qqqqff] (6.2,3.8)-- (7.,3.8);
			\draw [line width=0.7pt,color=qqqqff] (6.2,0.2)-- (8.,0.2);
			\draw [line width=0.7pt,color=qqqqff] (7.,3.8)-- (7.,1.);
			\draw [line width=0.7pt,color=qqqqff] (8.,0.2)-- (8.,3.);
			\draw [line width=0.7pt,color=qqqqff] (7.,1.)-- (7.5,1.);
			\draw [shift={(7.75,2.)},line width=0.7pt,color=ffqqqq]  plot[domain=-1.8157749899217608:1.3258176636680323,variable=\t]({1.*1.0307764064044151*cos(\t r)+0.*1.0307764064044151*sin(\t r)},{0.*1.0307764064044151*cos(\t r)+1.*1.0307764064044151*sin(\t r)});
			\draw [line width=0.7pt,color=yqqqyq] (12.1,1.9)-- (13.516777279764293,1.9015340667235952);
			\draw [line width=0.7pt,color=yqqqyq] (12.1,2.1)-- (13.516785283840777,2.0985052052185695);
			\draw [line width=0.7pt,color=qqqqff] (12.2,1.95)-- (12.2,0.2);
			\draw [line width=0.7pt,color=qqqqff] (12.2,0.2)-- (14.,0.2);
			\draw [line width=0.7pt,color=qqqqff] (12.2,2.05)-- (12.2,3.8);
			\draw [line width=0.7pt,color=qqqqff] (12.2,3.8)-- (14.,3.8);
			\draw [line width=0.7pt,color=qqqqff] (14.,3.)-- (14.,1.);
			\draw [shift={(13.5,2.)},line width=0.7pt,color=qqqqff]  plot[domain=1.5707963267948966:4.71238898038469,variable=\t]({1.*0.5*cos(\t r)+0.*0.5*sin(\t r)},{0.*0.5*cos(\t r)+1.*0.5*sin(\t r)});
			\draw [line width=0.7pt,color=qqqqff] (14.,1.)-- (13.5,1.5);
			\draw [line width=0.7pt,color=ffqqqq] (13.5,2.5)-- (14.,3.8);
			\draw [line width=0.7pt,color=qqqqff] (14.,0.2)-- (13.5,1.);
			\draw [shift={(13.75,2.)},line width=0.7pt,color=qqqqff]  plot[domain=1.3258176636680326:4.4674103172578254,variable=\t]({1.*1.0307764064044151*cos(\t r)+0.*1.0307764064044151*sin(\t r)},{0.*1.0307764064044151*cos(\t r)+1.*1.0307764064044151*sin(\t r)});
			\draw [line width=0.7pt,color=yqqqyq] (1.4,8.1)-- (1.5101618536016237,7.8997074761858945);
			\draw [line width=0.7pt,color=yqqqyq] (1.4,7.9)-- (1.5099982782216863,8.099490264118094);
			\draw [line width=0.7pt,color=yqqqyq] (0.1,8.1)-- (1.4,8.1);
			\draw [line width=0.7pt,color=yqqqyq] (0.1,7.9)-- (1.4,7.9);
			\draw [shift={(2.,8.)},line width=0.7pt,color=yqqqyq]  plot[domain=-2.9396375846349327:2.9412751165272244,variable=\t]({1.*0.5*cos(\t r)+0.*0.5*sin(\t r)},{0.*0.5*cos(\t r)+1.*0.5*sin(\t r)});
			\draw [shift={(8.,2.)},line width=0.7pt,color=yqqqyq]  plot[domain=-2.940005962371435:2.9400290735458094,variable=\t]({1.*0.5*cos(\t r)+0.*0.5*sin(\t r)},{0.*0.5*cos(\t r)+1.*0.5*sin(\t r)});
			\draw [shift={(14.,2.)},line width=0.7pt,color=yqqqyq]  plot[domain=-2.940575451948491:2.9404941803820885,variable=\t]({1.*0.4931528539590842*cos(\t r)+0.*0.4931528539590842*sin(\t r)},{0.*0.4931528539590842*cos(\t r)+1.*0.4931528539590842*sin(\t r)});
			\begin{scriptsize}
				\draw [fill=black] (2.,8.5) circle (2.pt);
				\draw[color=black] (2.3,8.7) node {$N$};
				\draw[color=yqqqyq] (0.7,8.3) node {$u\circ\phi$};
				\draw[color=yqqqyq] (6.7,8.6) node {$u\circ\phi$};
				\draw[color=yqqqyq] (12.7,8.8) node {$u\circ\phi$};
				\draw[color=yqqqyq] (0.5,2.3) node {$\phi$};
				\draw[color=yqqqyq] (6.5,2.3) node {$\phi$};
				\draw[color=yqqqyq] (12.5,2.3) node {$\phi$};
				\draw[color=blue] (0.7,6.4) node {$\psi$};
				\draw[color=blue] (6.7,6.4) node {$\psi$};
				\draw[color=blue] (12.7,6.4) node {$\psi$};
				\draw[color=blue] (0.7,0.4) node {$\gamma$};
				\draw[color=blue] (6.7,0.4) node {$\gamma$};
				\draw[color=blue] (12.7,0.4) node {$\gamma$};
				\draw [fill=black] (8.,8.5) circle (2.pt);
				\draw[color=black] (8.2,8.8) node {$N$};
				\draw [fill=black] (14.,8.5) circle (2.pt);
				\draw[color=black] (14.3,8.7) node {$N$};
				\draw [fill=black] (2.,2.5) circle (2.pt);
				\draw[color=black] (2.3,2.7) node {$S$};
				\draw [fill=black] (8.,2.5) circle (2.pt);
				\draw[color=black] (8.2,2.7) node {$N$};
				\draw [fill=black] (14,2.5) circle (2.pt);
				\draw[color=black] (14.2,2.8) node {$S$};
				\draw [fill=black] (2.,1.5) circle (2.pt);
				\draw[color=black] (2.2,1.75) node {$N$};
				\draw [fill=black] (8.,7.5) circle (2.0pt);
				\draw[color=black] (8.2,7.75) node {$S$};
				\draw [fill=uuuuuu] (14.,7.5) circle (2.0pt);
				\draw[color=uuuuuu] (14.2,7.75) node {$S$};
				\draw [fill=uuuuuu] (14.,1.5) circle (2.0pt);
				\draw[color=uuuuuu] (14.2,1.8) node {$N$};
				\draw [fill=uuuuuu] (8.,1.5) circle (2.0pt);
				\draw[color=uuuuuu] (8.1,1.8) node {$S$};
				\draw [fill=uuuuuu] (2.,7.5) circle (2.0pt);
				\draw[color=uuuuuu] (2.1,7.2) node {$S$};
			\end{scriptsize}
		\end{tikzpicture}
		\caption{Three cases are depicted in this figure. Above is the case in the image and below is the case in the preimage. In all cases the good test pair $\phi$ and $\psi$ are the same. The case on the left-hand side corresponds to an even number of intersections of $u\circ\phi$ with $\psi$ and an even number of intersections of $u\circ\phi$ with $\psi$. The case in the middle corresponds to an odd number of intersections of $u\circ\phi$ with $\psi$ and an even number of intersections of $u\circ\phi$ with $\psi$. The case on the right-hand side corresponded to an odd number of intersections of $u\circ\phi$ with $\psi$ and an odd number of intersections of $u\circ\phi$ with $\psi$. In every case it is impossible to construct an injective $\gamma$ which visits the points $N$ and $S$ in the correct order.}\label{fig:ZadniPsi}
	\end{figure}
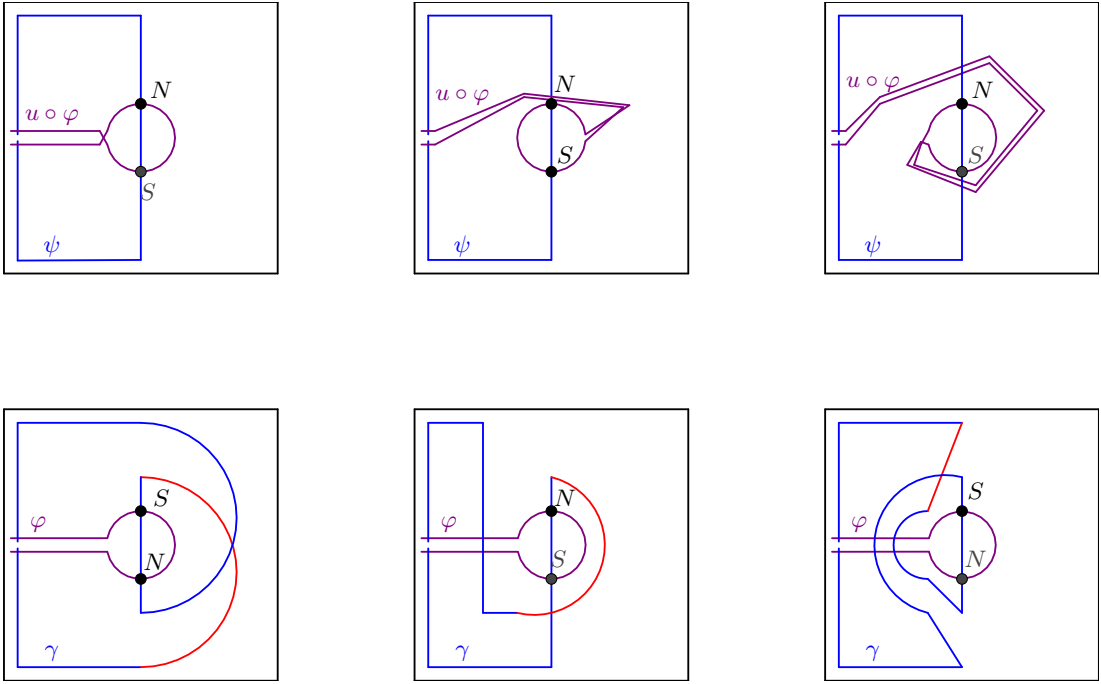

	Using the three-curve property of $u$ we find a continuous injective $\gamma$ which especially satisfies \eqref{IntervalOrder}. Especially, by assuming that $\rho$ is small enough and $\psi$ has been chosen well, we may assume that $\psi([0,1]) \cap u(\partial B((x_0,y_0), \rho)) = \{N,S\}$ consists of exactly two points. Our aim is to prove a contradiction with the three-curve property, i.e. that $\gamma$ cannot be injective. Let us call the part of the segment of $\psi$ above $N$ the north segment of $\psi$ and we call the part bellow $S$ the south segment of $\psi$ (for reference see the upper part of Figure~\ref{fig:ZadniPsi}). We consider three cases, the first cases is where the number of intersections of $u(T)$ with the north and south part of the segment of $\psi$ is even. The second case is either when the number of intersections of $\psi$ with the north segment of $\psi$ is odd and with the south segment of $\psi$ is even or visa versa. The third case is when the number of intersections of $u(T)$ with both segments of $\psi$ is odd.
	
	As can be observed in Figure~\ref{fig:ZadniPsi} we get a contradiction in each case. The first case is equivalent to saying that the curve $\gamma$ cannot intersect the segments of $\phi$. Then the part of the curve $\gamma$ until it gets to $u^{-1}(N)\cap \partial B((x_0,y_0),\rho)$ disconnects the point $u^{-1}(S)\cap \partial B((x_0,y_0),\rho)$ from $\gamma(1)$ in $(-1,1)^2\setminus\phi([0,1])$ and so either $\gamma$ is not injective or it must intersect $\phi([0,1])$ at a point that it may not. Either of these is a contradiction. The other two cases lead similarly to contradictions as can be observed from Figure~\ref{fig:ZadniPsi}.
\end{proof}

Now we prove that condition $a)$ from Theorem~\ref{MainTheorem} is equivalent with condition $c)$.
\begin{thm}\label{ThreeCurveChar}
	Let $u\in W_{\id}^{1,p}((-1,1)^2, \er^2)$ and let $S_u \subset (-1,1)^2$ with $\H^1(S_u) = 0$ be such that $u_{\rceil (-1,1)^2\setminus S_u}$ is continuous on $(-1,1)^2\setminus S_u$. Then $u \in \overline{H\cap W_{\id}^{1,p}((-1,1)^2, \er^2)}$ if and only if $u$ has the three-curve property.
\end{thm}

\begin{proof}[Proof that limits of homeomorphisms satisfy the three-curve property]
	Let us assume that $u \in \overline{H\cap W_{\id}^{1,p}((-1,1)^2, \er^2)}$, then by Theorem~\ref{QuasiMonotne} $u$ is quasi-monotone. Let $\delta>0$ and let $\phi$, $\psi$ be a good test pair for $u$ with $\varphi([0,1]) \subset (-1,1)^2\setminus U_{\delta}$. There is a monotone map $g_{\delta}:[-1,1]^2 \to [-1,1]^2$ such that $u=g_{\delta}$ on $\varphi([0,1])$. By the Young's theorem, for every $\epsilon>0$ we find a homeomorphism $h_{\epsilon}:[-1,1]^2 \to [-1,1]^2$ with $\|h_{\epsilon} - g_{\delta}\|_{\infty}< \epsilon$. Let us call the set $\tilde{\gamma}_{\epsilon} = h_{\epsilon}^{-1}(\psi([0,1]))$.
	
	There are now two operations necessary to define the curve $\gamma$. The first is to prove that by choosing a very small $\epsilon>0$ and then doing a slight alteration we can construct a curve $\hat{\gamma}$ such that $\hat{\gamma} \cap \phi([0,1]) =u^{-1}(\psi([0,1]))\cap \phi([0,1]) =:F$ and the property \eqref{IntervalOrder} holds. The second operation is to make $\gamma([0,1])$ by modifying $\hat{\gamma}$ so that it can be parameterized by a Lipschitz continuous curve.

	We have the existence of two linearly independent vectors $(u\circ\phi)'(t)$ and $\psi'(s_t)$ for each $t\in T = \phi^{-1}(F)$, i.e. $u\circ\phi(t)\in \psi((0,1))$ and where $s_t = \psi^{-1}(u\circ\phi(t))$. Therefore we can find a $\rho>0$ such that the balls $B((w,z), \rho)$ are pairwise disjoint for $(w,z) \in u(F)$ and for each $t\in T$ there is an interval $I_{t,\rho} = (a_{t,\rho}, b_{t,\rho})$ of $\phi^{-1}\circ u^{-1}\big(B((w,z),\rho)\big)$ containing $t$ and no other elements of $T$. Similarly we may assume that for each $s\in \psi^{-1}(u(F))$ there is an interval $s\in J_{s,\rho} \subset \psi^{-1}\big(B(\psi(s),\rho)\big)$, the intervals are pairwise disjoint and each interval contains exactly one point $s\in \psi^{-1}(u(F))$. The blow up of $u\circ\phi([0,1])$ at the points $(w,z) \in u(F)$ i.e.
	$$
	\frac{u\circ\phi([0,1]) - (w,z)}{\rho}\cap B(0,1)
	$$
	converges (for example in Haussdorf metric) to a segment parallel to $(u\circ\phi)'(t)$. Similarly, the blow up of $\psi([0,1])$ converges to a segment parallel to $\psi'(s_t)$. Because these segments are non-parallel and intersect at $(0,0)$ they must cross. This means we can find a $\rho$ such that the two points $u\circ\phi(a_{t,\rho})$ and $u\circ\phi(b_{t,\rho})$ lie in different components of $B((w,z),\rho)\setminus\psi(J_{s_t,\rho})$.
	
	It is possible to find an $\epsilon>0$ such that for any continuous $h_{\epsilon}$ such that $|h_{\epsilon}\circ\phi(t)- u\circ\phi(t)|<\epsilon$ for all $t\in[0,1]$ we have that
	$$
	\phi^{-1}\circ h_{\epsilon}^{-1}(\psi([0,1]))\subset \bigcup_{t\in T}I_{t,\rho}
	$$
	and $h_{\epsilon}\circ\phi(a_{t,\rho})$ lies in the same component of $B((w,z),\rho)\setminus\psi(J_{s_t,\rho})$ as does $u\circ\phi(a_{t,\rho})$ and $h_{\epsilon}\circ\phi(b_{t,\rho})$ lies in the same component of $B((w,z),\rho)\setminus\psi(J_{s_t,\rho})$ as does $u\circ\phi(b_{t,\rho})$. Further, according to the first paragraph it is possible to find such a homeomorphic $h_{\epsilon}$.
	
	There is a $\rho_0>0$ and a $\delta>0$ (depending on $Du(x,y)$, $(x,y)\in F$) such that for all $\rho\in (0,\rho_0)$ there are $d,\epsilon>0$ sufficiently small that the inclusions
	$$
	h_{\epsilon}^{-1}(\psi([0,1]))\cap \phi(I_{t,\rho})\subset B(\phi(t), \delta^{-1}\epsilon) \subset B(\phi(t), d) \subset h_{\epsilon}^{-1}(B(u\circ\phi(t), \rho))
	$$
	hold for all $t\in T$. By choosing $\epsilon\ll d$ we may assume that there is only one component of $h_{\epsilon}^{-1}(\psi([0,1])) \cap B((x,y), d)$ intersecting $\phi([0,1])$. In order to construct $\hat{\gamma}$ we alter $\tilde{\gamma}_{\epsilon}$ in the following way. We keep all of $\tilde{\gamma}_{\epsilon}$ outside of $\bigcup_{(x,y)\in F}B((x,y), d)$ without doing anything. We replace each component of $\tilde{\gamma}_{\epsilon} \cap B((x,y), d)$, $(x,y)\in F$, which does not intersect $\phi([0,1])$ with a segment which has the same endpoints as the replaced part of $\tilde{\gamma}_{\epsilon}$. We replace the component of $\tilde{\gamma}_{\epsilon} \cap B((x,y), d)$ intersecting $\phi([0,1])$ with an injective Lipschitz curve inside $ B((x,y), d)$ which has the same end points as the replaced component. We may assume that this curve intersects $\phi([0,1])$ exactly once at $(x,y)$ perpendicular to $\phi'(\phi^{-1}(x,y))$ and we may assume that $\hat{\gamma}$ is an injective curve. The existence of this curve is easy to observe especially since we have proved that $h_{\epsilon}\circ\phi(I_{t,\rho})$ crosses $\psi([0,1])$ and $h_{\epsilon}$ is a homeomorphism. This is very reminiscent of a similar step in the proof of Proposition~\ref{PiecewiseLinearization} where we did something similar with the union of two segments, here however $\phi$ is just an injective Lipschitz curve close to a segment so we must be a bit more careful.
	
	It not difficult to observe that the injective and mutually disjoint curves in $\hat{\gamma}\setminus \bigcup_{(x,y)\in F}B((x,y), d)$ which may have infinite length can be replaced by injective Lipschitz curves of finite length while maintaining the fact that they are disjoint. We call the resulting set $\gamma([0,1])$ and the mapping $\gamma$ may be the constant-speed parametrization of $\gamma([0,1])$. The ordering property \eqref{IntervalOrder} is easily checked since $\psi$ intersects each of the sets $h_{\epsilon}\circ\phi(I_{t,\rho})$ inside $B((w,z), \rho)$ for $(w,z)\in u(F)$ exactly when $u\circ\phi(t) = (w,z)$ and the balls $B((w,z), \rho)$ are pairwise disjoint.  The \eqref{GammaCross} property holds because we constructed $\gamma$ so that $\gamma'$ is non-zero perpendicular to $\phi'$ at their mutual intersection points. Thus $u$ has the three-curve property.
\end{proof}

\begin{proof}[Proof that if $u$ satisfies the three-curve property then it is the limit of homeomorphisms]
	\step{1}{Verifying the NC condition if we have a stronger version of the three-curve property for grids}{Step1of3Cproof}
	
	Firstly, let us assume that we have a following improved version of the three curve property where we replace $\varphi([0,1])$ with $\Gamma_k$ a suitable $k$-grid for $u$ and replace $\psi(0,1)$ with $\G$, a good arrival grid for $\Gamma_k$ from Proposition~\ref{ArrivalGrids}. Specifically let us assume for every $\Gamma_k$ a suitable $k$-grid for $u$ and every good arrival grid for $u_{\rceil \Gamma_k}$
	$$
	\G= \bigcup_{n=0}^M \{w_n\}\times [-1,1] \cup [-1,1]\times\{z_n\}.
	$$
	that there exists an injective continuous map $g: \G \to [-1,1]^2$ such that $g(\G)\cap \Gamma_k = \Gamma_k \cap u^{-1}(\G)$. Then thanks to the Jordan Sch\"onflies theorem there is a homeomorphism between the bounded components of $\er^2 \setminus g(\G)$ and rectangles of $\G$, i.e. bounded components of $\er^2\setminus \G$. Thus, for each $A$ bounded component of $\er^2 \setminus g(\G)$ there is a continuous injective embedding of $A \cap\Gamma_k$ into the rectangle of $\G$ whose boundary is $g^{-1}(\partial A)$. It is not difficult to observe that these piecewise homeomorphisms with disjoint images can be adapted to be continuous on their common boundaries. This yields $\Phi$, an injective continuous mapping of $\Gamma_k$ into $\er^2$ such that $\Phi(x,y)$ lies in the same rectangle of $\G$ as does $u(x,y)$. Assuming that the rectangles of $\G$ have diameter at most $ \epsilon /4$ we have $\|\Phi - u_{\rceil \Gamma_k}\|_{\infty} < \epsilon$. Thus, we have exactly verified the no-crossing condition of Definition~\ref{StayCalmPlease}.
	
	\step{2}{Replace $\Gamma_k$ with $\phi([0,1])$ and $\G$ with $\psi([0,1])$ and use the three-curve property}{Step2of3Cproof}
	
	Without loss of generality, we may assume that there is a neighborhood of $\partial [-1,1]^2$ called $P$ such that $u(x,y) =(x,y)$ for $(x,y)\in P$. Also without loss of generality we may assume that every $R=[x_i,x_{i+1}]\times[y_j,y_{j+1}]$ rectangle of $\Gamma_k$ such that $R\cap \partial[-1,1]^2\neq \emptyset$ we have $R \subset P$. Similarly, for every $\tilde{R}=[w_n,w_{n+1}]\times[z_m,z_{m+1}]$ rectangle of $\G$ such that $\tilde{R}\cap \partial[-1,1]^2$ we have $\tilde{R} \subset P$.
	
	It now remains to show how the three-curve property implies the stronger property from step~\ref{Step1of3Cproof}. Let $\Gamma_k$ be a suitable $k$-grid for $u$ and let $\G$ be a corresponding good arrival grid as in step~\ref{Step1of3Cproof}. We start by constructing a curve $\psi$. Recall that the set $F = \Gamma_k \cap u^{-1}(\G)$ is finite, that $(x,y)\in F$ lies on a side of $\Gamma_k$ never on a vertex, $u(x,y)$ lies on sides of $\G$, never on vertexes of $\G$ and that $D_{\tau}u(x,y)$ exists and has non-zero component perpendicular to $\G$ at $u(x,y)$ for all $(x,y)\in F$.

	\begin{figure}[h]
		\begin{tikzpicture}[line cap=round,line join=round,>=triangle 45,x=1.5cm,y=1.5cm]
			\clip(-0.1,-0.1) rectangle (5.1,5.1);
			\draw [line width=0.5pt,dotted] (1.,5.)-- (1.,0.);
			\draw [line width=0.5pt,dotted] (2.,0.)-- (2.,5.);
			\draw [line width=0.5pt,dotted] (3.,5.)-- (3.,0.);
			\draw [line width=0.5pt,dotted] (4.,0.)-- (4.,5.);
			\draw [line width=0.5pt,dotted] (5.,4.)-- (0.,4.);
			\draw [line width=0.5pt,dotted] (0.,3.)-- (5.,3.);
			\draw [line width=0.5pt,dotted] (5.,2.)-- (0.,2.);
			\draw [line width=0.5pt,dotted] (0.,1.)-- (5.,1.);
			\draw [->,line width=0.5pt] (1.,3.8) -- (1.,3.2);
			\draw [line width=0.5pt] (1.,3.2) -- (1.2,3.);
			\draw [line width=0.5pt] (1.2,3.) -- (1.8,3.);
			\draw [line width=0.5pt] (1.8,3.) -- (2.,2.8);
			\draw [->,line width=0.5pt] (2.,2.8) -- (2.,2.2);
			\draw [line width=0.5pt] (2.,2.2) -- (2.2,2.);
			\draw [->,line width=0.5pt] (2.2,2.) -- (2.8,2.);
			\draw [line width=0.5pt] (2.8,2.) -- (3.,1.8);
			\draw [line width=0.5pt] (3.,1.8) -- (3.,1.2);
			\draw [line width=0.5pt] (3.,1.2) -- (3.2,1.);
			\draw [->,line width=0.5pt] (3.2,1.) -- (3.8,1.);
			\draw [line width=0.5pt] (3.8,1.) -- (4.,0.8);
			\draw [->,line width=0.5pt] (4.,0.8) -- (4.,0.2);
			\draw [line width=0.5pt, color=blue] (3.,0.2) -- (3.,0.8);
			\draw [line width=0.5pt, color=blue] (3.,0.8) -- (2.8,1.);
			\draw [->,line width=0.5pt, color=blue] (2.8,1.) -- (2.2,1.);
			\draw [line width=0.5pt, color=blue] (2.2,1.) -- (2.,1.2);
			\draw [line width=0.5pt, color=blue] (2.,1.2) -- (2.,1.8);
			\draw [line width=0.5pt, color=blue] (2.,1.8) -- (1.8,2.);
			\draw [->,line width=0.5pt, color=blue] (1.8,2.) -- (1.2,2.);
			\draw [line width=0.5pt, color=blue] (1.2,2.) -- (1.,2.2);
			\draw [->,line width=0.5pt, color=blue] (1.,2.2) -- (1.,2.8);
			\draw [line width=0.5pt, color=blue] (1.,2.8) -- (0.8,3.);
			\draw [->,line width=0.5pt, color=blue] (0.8,3.) -- (0.2,3.);
			\draw [line width=0.5pt] (0.2,3.) -- (0.2,2.);
			\draw [line width=0.5pt, color=red] (0.2,2.) -- (0.8,2.);
			\draw [line width=0.5pt, color=red] (0.8,2.) -- (1.,1.8);
			\draw [->,line width=0.5pt, color=red] (1.,1.8) -- (1.,1.2);
			\draw [line width=0.5pt, color=red] (1.,1.2) -- (1.2,1.);
			\draw [->,line width=0.5pt, color=red] (1.2,1.) -- (1.8,1.);
			\draw [line width=0.5pt, color=red] (1.8,1.) -- (2.,0.8);
			\draw [->,line width=0.5pt, color=red] (2.,0.8) -- (2.,0.2);
			\draw [line width=0.5pt] (2.,0.2) -- (1.,0.2);
			\draw [->,line width=0.5pt] (1.,0.2) -- (1.,0.8);
			\draw [line width=0.5pt] (1.,0.8) -- (0.8,1.);
			\draw [line width=0.5pt] (0.8,1.) -- (0.2,1.);
			\draw [->,line width=0.5pt] (4.,0.2) -- (3.,0.2);
			\draw [->,line width=0.5pt] (4.,5.) -- (4.,4.2);
			\draw [line width=0.5pt] (4.,4.2) -- (4.2,4.);
			\draw [line width=0.5pt] (4.2,4.) -- (4.8,4.);
			\draw [line width=0.5pt] (4.8,4.) -- (4.8,3.);
			\draw [->,line width=0.5pt] (4.8,3.) -- (4.2,3.);
			\draw [line width=0.5pt] (4.2,3.) -- (4.,3.2);
			\draw [->,line width=0.5pt] (4.,3.2) -- (4.,3.8);
			\draw [line width=0.5pt] (4.,3.8) -- (3.8,4.);
			\draw [->,line width=0.5pt] (3.8,4.) -- (3.2,4.);
			\draw [line width=0.5pt] (3.2,4.) -- (3.,4.2);
			\draw [->,line width=0.5pt] (3.,4.2) -- (3.,4.8);
			\draw [->,line width=0.5pt] (3.,4.8) -- (2.,4.8);
			\draw [->,line width=0.5pt] (2.,4.8) -- (2.,4.2);
			\draw [line width=0.5pt] (2.,4.2) -- (2.2,4.);
			\draw [line width=0.5pt] (2.2,4.) -- (2.8,4.);
			\draw [line width=0.5pt] (2.8,4.) -- (3.,3.8);
			\draw [line width=0.5pt] (3.,3.8) -- (3.,3.2);
			\draw [line width=0.5pt] (3.,3.2) -- (3.2,3.);
			\draw [line width=0.5pt] (3.2,3.) -- (3.8,3.);
			\draw [line width=0.5pt] (3.8,3.) -- (4.,2.8);
			\draw [->,line width=0.5pt] (4.,2.8) -- (4.,2.2);
			\draw [line width=0.5pt] (4.,2.2) -- (4.2,2.);
			\draw [line width=0.5pt] (4.2,2.) -- (4.8,2.);
			\draw [line width=0.5pt] (4.8,2.) -- (4.8,1.);
			\draw [->,line width=0.5pt] (4.8,1.) -- (4.2,1.);
			\draw [line width=0.5pt] (4.2,1.) -- (4.,1.2);
			\draw [->,line width=0.5pt] (4.,1.2) -- (4.,1.8);
			\draw [line width=0.5pt] (4.,1.8) -- (3.8,2.);
			\draw [line width=0.5pt] (3.8,2.) -- (3.2,2.);
			\draw [line width=0.5pt] (3.2,2.) -- (3.,2.2);
			\draw [->,line width=0.5pt] (3.,2.2) -- (3.,2.8);
			\draw [line width=0.5pt] (3.,2.8) -- (2.8,3.);
			\draw [line width=0.5pt] (2.8,3.) -- (2.2,3.);
			\draw [line width=0.5pt] (2.2,3.) -- (2.,3.2);
			\draw [line width=0.5pt] (2.,3.2) -- (2.,3.8);
			\draw [line width=0.5pt] (2.,3.8) -- (1.8,4.);
			\draw [line width=0.5pt] (1.8,4.) -- (1.2,4.);
			\draw [line width=0.5pt] (1.2,4.) -- (1.,4.2);
			\draw [->,line width=0.5pt] (1.,4.2) -- (1.,4.8);
			\draw [line width=0.5pt] (1.,4.8) -- (0.2,4.8);
			\draw [line width=0.5pt] (0.2,4.8) -- (0.2,4.);
			\draw [->,line width=0.5pt] (0.2,4.) -- (0.8,4.);
			\draw [line width=0.5pt] (0.8,4.) -- (1.,3.8);
		\end{tikzpicture}
		\caption{A curve made from $\G$ by replacing crosses at its vertices with $NE$, $SW$ segments. We make curves that go left and up or go down and right (for example the red and the blue part). These are then connected  together close to $\partial(-1,1)^2$ inside $P$. The set depicted above is called $\tilde{\Psi}$.}\label{fig:wavy}
	\end{figure}

	For an illustration of the following paragraph see Figure~\ref{fig:wavy}. Since the set $\{(w_n,z_m)\}$ of vertexes of $\G$ is finite (and therefore compact), $u(\Gamma_k)$ is a compact set and the two sets are disjoint, we find a $\rho_0$ such that
	$$
	B((w_n,z_m),\rho_0)\cap [u(\Gamma_k)+B(0, \rho_0)] = \emptyset \text{ for every  $(w_n,z_m)$ vertex of }\G.
	$$
	Assuming that $\rho\in (0,\rho_0)$ is small enough so that $B((w_n,z_m),\rho)$ are pairwise disjoint, there are exactly four segments of $\G$ intersecting $B((w_n,z_m),\rho)$ each ending in $(w_n,z_m)$. Let us denote the vertical segment above $(w_n,z_m)$ the north segment, and the vertical segment below $(w_n,z_m)$ the south segment. Similarly, we denote the left segment west and the right segment east. For each $(w_n,z_m)$ we call $NE(w_n,z_m)$ the segment joining the intersection of the north segment and $\partial B((w_m,z_n),\rho)$ and the intersection of the east segment and $\partial B((w_m,z_n),\rho)$. Similarly, $SW(w_n,z_m)$ is the corresponding segment for the south and west segments. The set
	$$
	\Big[\G \setminus \bigcup_{n,m=1}^M B((w_n,z_m), \rho)\Big] \cup  \bigcup_{n,m=1}^M NE(w_n,z_m)\cup SW(w_n,z_m)
	$$
	has $2M-1$ components, each of which is (the image of) an injective Lipschitz curve. We connect these components inside the neighborhood of the boundary $P$, to get (the image of) a single injective Lipschitz curve as shown in Figure~\ref{fig:wavy}. Let us call this set $\tilde{\Psi}$.

	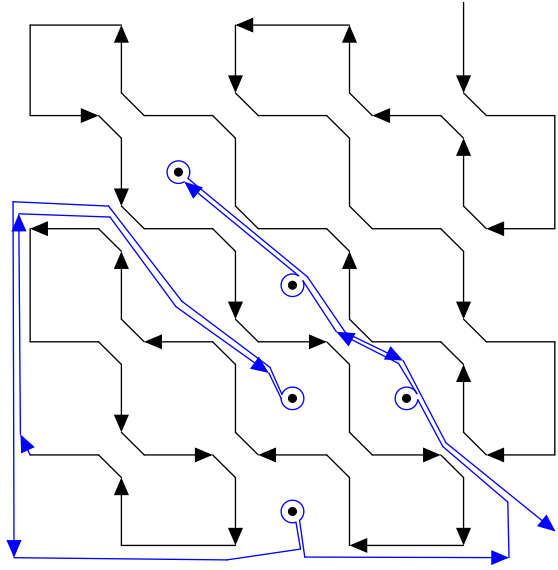
\begin{figure}[h]
		\begin{tikzpicture}[line cap=round,line join=round,>=triangle 45,x=1.5cm,y=1.5cm]
			\draw [->,line width=0.5pt] (1.,3.8) -- (1.,3.2);
			\draw [line width=0.5pt] (1.,3.2) -- (1.2,3.);
			\draw [line width=0.5pt] (1.2,3.) -- (1.8,3.);
			\draw [line width=0.5pt] (1.8,3.) -- (2.,2.8);
			\draw [->,line width=0.5pt] (2.,2.8) -- (2.,2.2);
			\draw [line width=0.5pt] (2.,2.2) -- (2.2,2.);
			\draw [->,line width=0.5pt] (2.2,2.) -- (2.8,2.);
			\draw [line width=0.5pt] (2.8,2.) -- (3.,1.8);
			\draw [line width=0.5pt] (3.,1.8) -- (3.,1.2);
			\draw [line width=0.5pt] (3.,1.2) -- (3.2,1.);
			\draw [->,line width=0.5pt] (3.2,1.) -- (3.8,1.);
			\draw [line width=0.5pt] (3.8,1.) -- (4.,0.8);
			\draw [->,line width=0.5pt] (4.,0.8) -- (4.,0.2);
			\draw [line width=0.5pt] (3.,0.2) -- (3.,0.8);
			\draw [line width=0.5pt] (3.,0.8) -- (2.8,1.);
			\draw [->,line width=0.5pt] (2.8,1.) -- (2.2,1.);
			\draw [line width=0.5pt] (2.2,1.) -- (2.,1.2);
			\draw [line width=0.5pt] (2.,1.2) -- (2.,1.8);
			\draw [line width=0.5pt] (2.,1.8) -- (1.8,2.);
			\draw [->,line width=0.5pt] (1.8,2.) -- (1.2,2.);
			\draw [line width=0.5pt] (1.2,2.) -- (1.,2.2);
			\draw [->,line width=0.5pt] (1.,2.2) -- (1.,2.8);
			\draw [line width=0.5pt] (1.,2.8) -- (0.8,3.);
			\draw [->,line width=0.5pt] (0.8,3.) -- (0.2,3.);
			\draw [line width=0.5pt] (0.2,3.) -- (0.2,2.);
			\draw [line width=0.5pt] (0.2,2.) -- (0.8,2.);
			\draw [line width=0.5pt] (0.8,2.) -- (1.,1.8);
			\draw [->,line width=0.5pt] (1.,1.8) -- (1.,1.2);
			\draw [line width=0.5pt] (1.,1.2) -- (1.2,1.);
			\draw [->,line width=0.5pt] (1.2,1.) -- (1.8,1.);
			\draw [line width=0.5pt] (1.8,1.) -- (2.,0.8);
			\draw [->,line width=0.5pt] (2.,0.8) -- (2.,0.2);
			\draw [line width=0.5pt] (2.,0.2) -- (1.,0.2);
			\draw [->,line width=0.5pt] (1.,0.2) -- (1.,0.8);
			\draw [line width=0.5pt] (1.,0.8) -- (0.8,1.);
			\draw [line width=0.5pt] (0.8,1.) -- (0.2,1.);
			\draw [->,line width=0.5pt] (4.,0.2) -- (3.,0.2);
			\draw [->,line width=0.5pt] (4.,5.) -- (4.,4.2);
			\draw [line width=0.5pt] (4.,4.2) -- (4.2,4.);
			\draw [line width=0.5pt] (4.2,4.) -- (4.8,4.);
			\draw [line width=0.5pt] (4.8,4.) -- (4.8,3.);
			\draw [->,line width=0.5pt] (4.8,3.) -- (4.2,3.);
			\draw [line width=0.5pt] (4.2,3.) -- (4.,3.2);
			\draw [->,line width=0.5pt] (4.,3.2) -- (4.,3.8);
			\draw [line width=0.5pt] (4.,3.8) -- (3.8,4.);
			\draw [->,line width=0.5pt] (3.8,4.) -- (3.2,4.);
			\draw [line width=0.5pt] (3.2,4.) -- (3.,4.2);
			\draw [->,line width=0.5pt] (3.,4.2) -- (3.,4.8);
			\draw [->,line width=0.5pt] (3.,4.8) -- (2.,4.8);
			\draw [->,line width=0.5pt] (2.,4.8) -- (2.,4.2);
			\draw [line width=0.5pt] (2.,4.2) -- (2.2,4.);
			\draw [line width=0.5pt] (2.2,4.) -- (2.8,4.);
			\draw [line width=0.5pt] (2.8,4.) -- (3.,3.8);
			\draw [line width=0.5pt] (3.,3.8) -- (3.,3.2);
			\draw [line width=0.5pt] (3.,3.2) -- (3.2,3.);
			\draw [line width=0.5pt] (3.2,3.) -- (3.8,3.);
			\draw [line width=0.5pt] (3.8,3.) -- (4.,2.8);
			\draw [->,line width=0.5pt] (4.,2.8) -- (4.,2.2);
			\draw [line width=0.5pt] (4.,2.2) -- (4.2,2.);
			\draw [line width=0.5pt] (4.2,2.) -- (4.8,2.);
			\draw [line width=0.5pt] (4.8,2.) -- (4.8,1.);
			\draw [->,line width=0.5pt] (4.8,1.) -- (4.2,1.);
			\draw [line width=0.5pt] (4.2,1.) -- (4.,1.2);
			\draw [->,line width=0.5pt] (4.,1.2) -- (4.,1.8);
			\draw [line width=0.5pt] (4.,1.8) -- (3.8,2.);
			\draw [line width=0.5pt] (3.8,2.) -- (3.2,2.);
			\draw [line width=0.5pt] (3.2,2.) -- (3.,2.2);
			\draw [->,line width=0.5pt] (3.,2.2) -- (3.,2.8);
			\draw [line width=0.5pt] (3.,2.8) -- (2.8,3.);
			\draw [line width=0.5pt] (2.8,3.) -- (2.2,3.);
			\draw [line width=0.5pt] (2.2,3.) -- (2.,3.2);
			\draw [line width=0.5pt] (2.,3.2) -- (2.,3.8);
			\draw [line width=0.5pt] (2.,3.8) -- (1.8,4.);
			\draw [line width=0.5pt] (1.8,4.) -- (1.2,4.);
			\draw [line width=0.5pt] (1.2,4.) -- (1.,4.2);
			\draw [->,line width=0.5pt] (1.,4.2) -- (1.,4.8);
			\draw [line width=0.5pt] (1.,4.8) -- (0.2,4.8);
			\draw [line width=0.5pt] (0.2,4.8) -- (0.2,4.);
			\draw [->,line width=0.5pt] (0.2,4.) -- (0.8,4.);
			\draw [line width=0.5pt] (0.8,4.) -- (1.,3.8);
			\draw [shift={(1.5,3.5)},line width=0.5pt, color=blue]  plot[domain=-0.5886745790152181:5.25649370279071,variable=\t]({1.*0.1*cos(\t r)+0.*0.1*sin(\t r)},{0.*0.1*cos(\t r)+1.*0.1*sin(\t r)});
			\draw [shift={(2.5,2.5)},line width=0.5pt, color=blue]  plot[domain=-5.288660486959367:0.39849164124891656,variable=\t]({1.*0.1*cos(\t r)+0.*0.1*sin(\t r)},{0.*0.1*cos(\t r)+1.*0.1*sin(\t r)});
			\draw [shift={(3.5,1.5)},line width=0.5pt, color=blue]  plot[domain=0.43027365812548035:6.154820193529883,variable=\t]({1.*0.1*cos(\t r)+0.*0.1*sin(\t r)},{0.*0.1*cos(\t r)+1.*0.1*sin(\t r)});
			\draw [shift={(2.5,0.5)},line width=0.5pt, color=blue]  plot[domain=-0.901381083619305:5.0218835872048855,variable=\t]({1.*0.1*cos(\t r)+0.*0.1*sin(\t r)},{0.*0.1*cos(\t r)+1.*0.1*sin(\t r)});
			\draw [->,line width=0.5pt, color=blue] (0.2,1.) -- (0.11460599170345026,1.1820617505145035);
			\draw [->,line width=0.5pt, color=blue] (0.11460599170345026,1.1820617505145035) -- (0.10032373952821229,3.131589172434544);
			\draw [line width=0.5pt, color=blue] (0.10032373952821229,3.131589172434544) -- (0.9001298613415384,3.103024668084067);
			\draw [line width=0.5pt, color=blue] (0.9001298613415384,3.103024668084067) -- (1.478561074438676,2.3103596723583366);
			\draw [->,line width=0.5pt, color=blue] (1.478561074438676,2.3103596723583366) -- (2.29264944842724,1.7247873331735624);
			\draw [line width=0.5pt, color=blue] (2.29264944842724,1.7247873331735624) -- (2.4000109451242446,1.5014794948641883);
			\draw [line width=0.5pt, color=blue] (2.406831039697851,1.5363255397236797) -- (2.3019029534778697,1.7751373811545426);
			\draw [line width=0.5pt, color=blue] (2.3019029534778697,1.7751373811545426) -- (1.53085061248044,2.3581417834686835);
			\draw [line width=0.5pt, color=blue] (1.53085061248044,2.3581417834686835) -- (0.8874908141818282,3.1998873754181862);
			\draw [line width=0.5pt, color=blue] (0.8874908141818282,3.1998873754181862) -- (0.049858128126479384,3.2384395074123664);
			\draw [->,line width=0.5pt, color=blue] (0.049858128126479336,3.2384395074123664) -- (0.058605040238803745,0.09137705534830287);
			\draw [line width=0.5pt, color=blue] (0.058605040238803745,0.09137705534830273) -- (1.9244299435731755,0.07202589762603781);
			\draw [line width=0.5pt, color=blue] (1.9244299435731755,0.07202589762603781) -- (2.570838829306321,0.16766802867839356);
			\draw [line width=0.5pt, color=blue] (2.570838829306321,0.16766802867839356) -- (2.530457729461343,0.4047512377190148);
			\draw [->,line width=0.5pt, color=blue] (2.6076225802273396,0.09704023519994667) -- (4.3984788854570995,0.09193808048419223);
			\draw [line width=0.5pt, color=blue] (4.3984788854570995,0.09193808048419223) -- (4.388047625296427,0.5830515481808669);
			\draw [line width=0.5pt, color=blue] (4.388047625296427,0.5830515481808669) -- (3.818121499715109,1.0741199708515758);
			\draw [line width=0.5pt, color=blue] (3.818121499715109,1.0741199708515758) -- (3.5991772505561195,1.4871987120910146);
			\draw [line width=0.5pt, color=blue] (3.5908851635558285,1.5417119532679817) -- (3.429228009592026,1.8092181559670302);
			\draw [->,line width=0.5pt, color=blue] (3.429228009592026,1.8092181559670302) -- (2.8842693803406245,2.088678135929872);
			\draw [line width=0.5pt, color=blue] (2.8842693803406245,2.088678135929872) -- (2.5921647328570634,2.538802860943314);
			\draw [->,line width=0.5pt, color=blue] (2.55449013893094,2.5838500134721927) -- (1.551765230654162,3.414440892388236);
			\draw [->,line width=0.5pt, color=blue] (2.965572571177941,2.0805478168461398) -- (3.470421221058876,1.8341908879399245);
			\draw [->,line width=0.5pt, color=blue] (3.8449953128957284,1.1081488092088159) -- (4.806247589134377,0.3233764430608477);
			\draw [line width=0.5pt, color=blue] (3.8449953128957284,1.1081488092088159)-- (3.470421221058876,1.8341908879399245);
			\draw [line width=0.5pt, color=blue] (2.965572571177941,2.0805478168461398)-- (2.6314034107306408,2.574167713495625);
			\draw [line width=0.5pt, color=blue] (1.5831677361585557,3.444474081166892)-- (2.6314034107306408,2.574167713495625);
			\draw [line width=0.5pt, color=blue] (2.6076225802273396,0.09704023519994667)-- (2.56205275358243,0.42158153423562633);
			\draw [shift={(2.5,1.5)},line width=0.5pt, color=blue]  plot[domain=-3.156388142030479:2.7698330506528097,variable=\t]({1.*0.1*cos(\t r)+0.*0.1*sin(\t r)},{0.*0.1*cos(\t r)+1.*0.1*sin(\t r)});
			\begin{scriptsize}
				\draw [fill=black] (2.5,2.5) circle (1.5pt);
				\draw [fill=black] (1.5,3.5) circle (1.5pt);
				\draw [fill=black] (3.5,1.5) circle (1.5pt);
				\draw [fill=black] (2.5,0.5) circle (1.5pt);
				\draw [fill=black] (2.5,1.5) circle (1.5pt);
			\end{scriptsize}
		\end{tikzpicture}
		\caption{The construction of $\Psi = \tilde{\Psi} \cup \hat{\Psi}$, whose parametrization is called $\psi$. The curve $\hat{\Psi}$ (depicted in blue) goes close to all images of crossing points of $\Gamma_k$ denoted by $\{(x_i,y_j)\}$.}\label{fig:arcs}
	\end{figure}

	We index the vertex points of the grid $\Gamma_k$ as $(x_i,y_j)$. The set $u(x_i,y_j)$ is finite and disjoint from $\tilde{\Psi}$. By Remark~\ref{Alibistic} we find an $\eta>0$ such that
	\begin{enumerate}
		\item $\overline{B((X,Y), \eta)}$ are pairwise disjoint for $(X,Y)\in \{u(x_i,y_j)\}$,
		\item $B((X,Y), \eta) \cap \tilde{\Psi} = \emptyset$ for $(X,Y)\in \{u(x_i,y_j)\}$,
		\item the set of points in $\Gamma_k$ mapped by $u$ onto $\partial B((X,Y), \eta)$ is finite for each $(X,Y)\in \{u(x_i,y_j)\}$
		\item for every $(x,y)\in \Gamma_k$ with $u(x,y)\in \partial B((X,Y), \eta)$ for some $(X,Y)\in \{u(x_i,y_j)\}$ there exists $D_{\tau}u(x,y)$ and has non-zero normal component to $\partial B((X,Y),\eta)$ at $u(x,y)$.
	\end{enumerate}
	For each of these circles $\partial B((X,Y), \eta)$ we subtract a small arc in $\partial B((X,Y), \eta)$ that does not intersect the finite set $u(\Gamma_k) \cap \partial B((X,Y), \eta)$. It is possible to connect all of these arcs with an injective Lipschitz curve that does not intersect $\tilde{\Psi}$ and this can be seen in Figure~\ref{fig:arcs}. Call this set $\hat{\Psi}$. Note that it is possible that $\hat{\Psi} \cap u(\Gamma_k)\neq \emptyset$. On the other hand, the image in $u$ of the subset of $\Gamma_k$ where $D_{\tau}u$ does not exist or the derivative is zero has $\H^1$ measure zero. Therefore, it cannot disconnect the plane and in fact $\hat{\Psi}$ can avoid it completely. Therefore we may assume that there are a finite number of points in $(x,y)\in\Gamma_k$ such that $u(x,y) \in \hat{\Psi}$, that $D_{\tau}u(x,y)$ exists and has non-zero component perpendicular the tangent set of $\hat{\Psi}$ at $u(x,y)$, which we may assume to exist. Since we may assume that $\hat{\Psi}$ starts at an endpoint of $\tilde{\Psi}$ we have that $\Psi = \tilde{\Psi}\cup \hat{\Psi}$ is the image of an injective Lipschitz curve. By $\psi:[0,1]\to \er^2$ we denote a Lipschitz parametrization of $\Psi$ such that $\psi([0,\tfrac{1}{2}]) = \tilde{\Psi}$ and $\psi([\tfrac{1}{2},1]) = \hat{\Psi}$.

	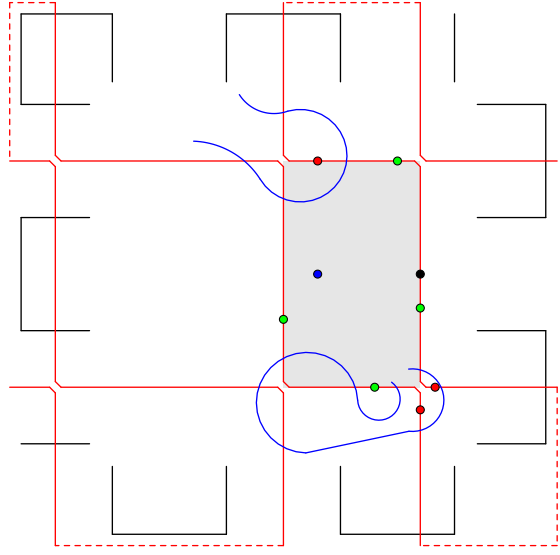
\begin{figure}[h]
		\begin{tikzpicture}[line cap=round,line join=round,>=triangle 45,x=1.5cm,y=1.5cm]
			\clip(-0.1,-0.1) rectangle (5.1,5.1);
			\fill[line width=0.pt,fill=black!10] (2.5,3.5) -- (3.7,3.5) -- (3.7,1.5) -- (2.5,1.5) -- cycle;
			\draw [line width=0.5pt] (4.,0.8) -- (4.,0.2);
			\draw [line width=0.5pt] (3.,0.2) -- (3.,0.8);
			\draw [line width=0.5pt] (0.8,3.) -- (0.2,3.);
			\draw [line width=0.5pt] (0.2,3.) -- (0.2,2.);
			\draw [line width=0.5pt] (0.2,2.) -- (0.8,2.);
			\draw [line width=0.5pt] (2.,0.8) -- (2.,0.2);
			\draw [line width=0.5pt] (2.,0.2) -- (1.,0.2);
			\draw [line width=0.5pt] (1.,0.2) -- (1.,0.8);
			\draw [line width=0.5pt] (0.8,1.) -- (0.2,1.);
			\draw [line width=0.5pt] (4.,0.2) -- (3.,0.2);
			\draw [line width=0.5pt] (4.2,4.) -- (4.8,4.);
			\draw [line width=0.5pt] (4.8,4.) -- (4.8,3.);
			\draw [line width=0.5pt] (4.8,3.) -- (4.2,3.);
			\draw [line width=0.5pt] (3.,4.2) -- (3.,4.8);
			\draw [line width=0.5pt] (3.,4.8) -- (2.,4.8);
			\draw [line width=0.5pt] (2.,4.8) -- (2.,4.2);
			\draw [line width=0.5pt] (4.2,2.) -- (4.8,2.);
			\draw [line width=0.5pt] (4.8,2.) -- (4.8,1.);
			\draw [line width=0.5pt] (4.8,1.) -- (4.2,1.);
			\draw [line width=0.5pt] (1.,4.2) -- (1.,4.8);
			\draw [line width=0.5pt] (1.,4.8) -- (0.2,4.8);
			\draw [line width=0.5pt] (0.2,4.8) -- (0.2,4.);
			\draw [line width=0.5pt] (0.2,4.) -- (0.8,4.);
			\draw [line width=0.5pt] (4.,4.8) -- (4.,4.2);
			\draw [line width=0.5pt,color=ffqqqq] (4.9,3.5)-- (3.75,3.5);
			\draw [line width=0.5pt,color=ffqqqq] (3.75,3.5)-- (3.7,3.55);
			\draw [line width=0.5pt,color=ffqqqq] (3.7,4.9)-- (3.7,3.55);
			\draw [line width=0.5pt,dash pattern=on 2pt off 2pt,color=ffqqqq] (3.7,4.9)-- (2.5,4.9);
			\draw [line width=0.5pt,color=ffqqqq] (2.5,4.9)-- (2.5,3.55);
			\draw [line width=0.5pt,color=ffqqqq] (2.5,3.55)-- (2.55,3.5);
			\draw [line width=0.5pt,color=ffqqqq] (2.55,3.5)-- (3.65,3.5);
			\draw [line width=0.5pt,color=ffqqqq] (3.65,3.5)-- (3.7,3.45);
			\draw [line width=0.5pt,color=ffqqqq] (3.7,3.45)-- (3.7,1.55);
			\draw [line width=0.5pt,color=ffqqqq] (3.7,1.55)-- (3.75,1.5);
			\draw [line width=0.5pt,color=ffqqqq] (3.75,1.5)-- (4.9,1.5);
			\draw [line width=0.5pt,dash pattern=on 2pt off 2pt,color=ffqqqq] (4.9,1.5)-- (4.9,0.1);
			\draw [line width=0.5pt,dash pattern=on 2pt off 2pt,color=ffqqqq] (4.9,0.1)-- (3.7,0.1);
			\draw [line width=0.5pt,color=ffqqqq] (3.7,0.1)-- (3.7,1.45);
			\draw [line width=0.5pt,color=ffqqqq] (3.7,1.45)-- (3.65,1.5);
			\draw [line width=0.5pt,color=ffqqqq] (3.65,1.5)-- (2.55,1.5);
			\draw [line width=0.5pt,color=ffqqqq] (2.55,1.5)-- (2.5,1.55);
			\draw [line width=0.5pt,color=ffqqqq] (2.5,1.55)-- (2.5,3.45);
			\draw [line width=0.5pt,color=ffqqqq] (2.5,3.45)-- (2.45,3.5);
			\draw [line width=0.5pt,color=ffqqqq] (2.45,3.5)-- (0.55,3.5);
			\draw [line width=0.5pt,color=ffqqqq] (0.55,3.5)-- (0.5,3.55);
			\draw [line width=0.5pt,color=ffqqqq] (0.5,3.55)-- (0.5,4.9);
			\draw [line width=0.5pt,dash pattern=on 2pt off 2pt,color=ffqqqq] (0.5,4.9)-- (0.1,4.9);
			\draw [line width=0.5pt,dash pattern=on 2pt off 2pt,color=ffqqqq] (0.1,4.9)-- (0.1,3.5);
			\draw [line width=0.5pt,color=ffqqqq] (0.1,3.5)-- (0.45,3.5);
			\draw [line width=0.5pt,color=ffqqqq] (0.45,3.5)-- (0.5,3.45);
			\draw [line width=0.5pt,color=ffqqqq] (0.5,3.45)-- (0.5,1.55);
			\draw [line width=0.5pt,color=ffqqqq] (0.5,1.55)-- (0.55,1.5);
			\draw [line width=0.5pt,color=ffqqqq] (0.55,1.5)-- (2.45,1.5);
			\draw [line width=0.5pt,color=ffqqqq] (2.45,1.5)-- (2.5,1.45);
			\draw [line width=0.5pt,color=ffqqqq] (2.5,1.45)-- (2.5,0.1);
			\draw [line width=0.5pt,dash pattern=on 2pt off 2pt,color=ffqqqq] (2.5,0.1)-- (0.5,0.1);
			\draw [line width=0.5pt,color=ffqqqq] (0.5,0.1)-- (0.5,1.45);
			\draw [line width=0.5pt,color=ffqqqq] (0.5,1.45)-- (0.45,1.5);
			\draw [line width=0.5pt,color=ffqqqq] (0.45,1.5)-- (0.1,1.5);
			\draw [shift={(1.6844373276990798,2.9516976138375233)},line width=0.5pt,color=qqqqff]  plot[domain=0.5630813197284342:1.5318219204773191,variable=\t]({1.*0.7240650298397869*cos(\t r)+0.*0.7240650298397869*sin(\t r)},{0.*0.7240650298397869*cos(\t r)+1.*0.7240650298397869*sin(\t r)});
			\draw [shift={(2.648360634429539,3.546867018815714)},line width=0.5pt,color=qqqqff]  plot[domain=-2.6060353516319794:1.8412324519478873,variable=\t]({1.*0.4088952693236433*cos(\t r)+0.*0.4088952693236433*sin(\t r)},{0.*0.4088952693236433*cos(\t r)+1.*0.4088952693236433*sin(\t r)});
			\draw [shift={(2.416507753145919,4.277203594859147)},line width=0.5pt,color=qqqqff]  plot[domain=3.7024934608421822:5.062010270154613,variable=\t]({1.*0.3579584064972491*cos(\t r)+0.*0.3579584064972491*sin(\t r)},{0.*0.3579584064972491*cos(\t r)+1.*0.3579584064972491*sin(\t r)});
			\draw [shift={(3.3369231821453145,1.3949923475916708)},line width=0.5pt,color=qqqqff]  plot[domain=-3.091634257867849:0.9307081740331696,variable=\t]({1.*0.18647567841209603*cos(\t r)+0.*0.18647567841209603*sin(\t r)},{0.*0.18647567841209603*cos(\t r)+1.*0.18647567841209603*sin(\t r)});
			\draw [shift={(2.706929320808355,1.3634926545248218)},line width=0.5pt,color=qqqqff]  plot[domain=0.0499583957219444:4.688043738766803,variable=\t]({1.*0.4443051836830991*cos(\t r)+0.*0.4443051836830991*sin(\t r)},{0.*0.4443051836830991*cos(\t r)+1.*0.4443051836830991*sin(\t r)});
			\draw [shift={(3.6302559385790283,1.3856801965937748)},line width=0.5pt,color=qqqqff]  plot[domain=-1.6784861976277128:1.6888883552139853,variable=\t]({1.*0.27663515296003466*cos(\t r)+0.*0.27663515296003466*sin(\t r)},{0.*0.27663515296003466*cos(\t r)+1.*0.27663515296003466*sin(\t r)});
			\draw [line width=0.5pt,color=qqqqff] (2.696113672222608,0.9193191321336935)-- (3.6005226825058054,1.1106475779164515);
			\begin{scriptsize}
				\draw [fill=black] (3.7,2.5) circle (1.5pt);
				\draw [fill=green] (3.7,2.2) circle (1.5pt);
				\draw [fill=green] (3.3,1.5) circle (1.5pt);
				\draw [fill=green] (2.5,2.1) circle (1.5pt);
				\draw [fill=green] (3.5,3.5) circle (1.5pt);
				\draw [fill=red] (3.7,1.3) circle (1.5pt);
				\draw [fill=red] (3.83,1.5) circle (1.5pt);
				\draw [fill=red] (2.8,3.5) circle (1.5pt);
				\draw [fill=blue] (2.8,2.5) circle (1.5pt);
			\end{scriptsize}
		\end{tikzpicture}
		\caption{The curve $\varphi$ (the curve depicted in red) is constructed by replacing crosses of $\Gamma_k$ with $NE$ and $SW$ segments. Then these curves are connected with each other close to the boundary of $(-1,1)^2$ inside $P$ and these segments are showed as dashed. The black curves are (the preimage in $u$ of) $\tilde{\Psi} \cap P$ and the dashed red curves do not intersect the black curves. The blue curves are the part of  $\gamma([1/2,1])$ each of which corresponds to a circular arc around some $u(x_i,y_j)$. The black point corresponds to some $\gamma(s)$ and the following intersection $\gamma(t) \in \phi([0,1])$ can be for example one of the green points but cannot be a red point. The blue point represents $\gamma(s/2+t/2)$. The point $\gamma(t)$ cannot be the upper-left red point because it is disconnected from $\gamma(s/2+ t/2)$ by $\gamma([1/2,1])$ in $[-1,1]^2\setminus \phi([0,1])$ and $\gamma$ is injective. The point $\gamma(t)$ cannot be one of the lower-right red points because these points are mapped by $u$ into $B(u(x_{i+1},y_j), \eta)$ but $|u(\gamma(t)) - u(x_{i+1},y_j)|\geq2\eta$.}\label{fig:MakePhi}
	\end{figure}

	Now it is the turn of $\Gamma_k$ to be replaced with a $\varphi$ such that $\varphi$ and $\psi$ are a good test pair for $u$. The construction of the curve $\varphi$ has two stages. The first one is similar to the first stage for the construction of $\psi$, i.e. we replace crosses in $\Gamma_k$ with pairs of $NE$ and $SW$ segments. We can find a $\delta_0>0$ so small that if the $NE$, $SW$ segments are made in $B((x_i,y_j), \delta)$ for any $\delta\in (0, \delta_0)$ then the image in $u$ of the added segments do not intersect $\partial B(u(x_i,y_j), \eta)$ and therefore there are no points on the added segments whose image intersects $\Psi$ at all (because $\dist(u(x_i,y_j), \Psi)> \eta$). The second stage is to connect up the curves resulting from this process. Because $\tilde{\Psi} \Subset (-1,1)^2$ we can join our components using segments very close to the boundary and so (because $u(x,y) = (x,y)$ in $P$) they do not intersect $\tilde{\Psi}$. This can be seen in Figure~\ref{fig:MakePhi}. Any intersection between $u\circ\varphi([0,1])$ and $\psi([0,1])$ inside $(-1,1)^2\setminus P$ have already been chosen so that they satisfy the properties for a good test pair. Any intersections between $\hat{\Psi}$ and $u\circ\varphi([0,1])$ inside $P$ can easily be made to satisfy the conditions. Finally, we get a single curve which we call $\varphi$ such that $u\circ \varphi (t) \in \tilde{\Psi}$ exactly when $\varphi(t) \in F$.
	
	By construction  $\varphi$ and $\psi$ are a good test pair for $u$ and so we find the injective Lipschitz curve $\gamma$ such that $\gamma([0,1])\cap \varphi([0,1]) = u^{-1}(\Psi) \cap \varphi([0,1])$. By the definition of the three curve property we may assume that $\gamma([0,1/2])\cap \varphi([0,1]) = u^{-1}(\tilde{\Psi}) \cap \varphi([0,1])$ and $\gamma([1/2,1])\cap \varphi([0,1]) = u^{-1}(\hat{\Psi}) \cap \varphi([0,1])$.

	\step{3}{Replace $\phi([0,1])$ with $\Gamma_k$ again}{Step3of3Cproof}
	
	Let $0<s<t<1/2$ be a subsequent pair of points of $\gamma^{-1}(\Gamma_k)$ (which is a finite set). We want to prove that there is some rectangle $R = [x_i, x_{i+1}]\times[y_j, y_{j+1}]$ of $\Gamma_k$ with $\gamma(s),\gamma(t) \in \partial R$ and it is possible by homotopy in $\er^2\setminus \phi([0,1])$ to require that $\gamma([s,t])\subset R$. We suggest the reader refer to Figure~\ref{fig:MakePhi} where we take $R$ to be the shaded rectangle and $\gamma(s)$ to be the black point on its boundary. Because  $\gamma((s,t))\cap \phi(0,1) = \emptyset$ we have that $\gamma(s), \gamma(t)$ lie in the closure of the component of $[-1,1]^2\setminus [\varphi([0,1]) \cup \gamma([1/2,1])]$ that contains $\gamma(s/2+t/2)$. The point $\gamma(t)$ must lie on $\varphi([0,1])$ (the red curve in Figure~\ref{fig:MakePhi}), not on $\gamma([1/2,1])$ (therefore not on blue curves in the picture). If the point $\gamma(t)$ does not lie on the boundary of $R$, then it must lie on part of $\varphi([0,1])\setminus R$ which is not disconnected from $R$ by $\gamma([1/2, 1])$. By the construction of $\Psi$ (especially the circular arcs around points in $u(\{(x_i,y_j)\})$) we have for every $(x,y) \in \varphi([0,1])$ which is not disconnected from $R$ by $\gamma([1/2,1])$ that $u(x,y)$ lies in $B((X,Y), \eta)$ for some $(X,Y) \in \{u(x_i,y_j); i,j=1,\dots M\}$. On the other hand, however, $\gamma(s) \in \tilde{\Psi}$ and by the choice of $\eta$ (point 2), $u(\gamma(s)) \notin B((X,Y), \eta)$. Therefore $\gamma(s) \in \partial R \cap \varphi([0,1])$.
	
	The last fact means that $\gamma(s)$ and $\gamma(t)$ lie in $\partial R \setminus \bigcup_{i,j}B((x_i,y_j), \delta)$ and after some simple homotopic correction to $\gamma$ we can assume that $\gamma([0,1/2])$ does not intersect $\overline{B((x_i,y_j), \delta)}$ at all, i.e. $\gamma([s,t])$ lies entirely inside $R$. Effectively this means we can replace $\varphi([0,1])$ with $\Gamma_k$, i.e. $\gamma([0,1/2]) \cap \Gamma_k = F = \Gamma_k \cap u^{-1}(\G)$. Then, to finish the proof, it remains to construct a continuous injective $g$ on $\G$ with $g(\G) \cap \Gamma_k = \gamma([0,1/2])\cap \phi([0,1]) = \Gamma_k\cap u^{-1}(\G)$.
	
	\step{4}{Replacing $\psi([0,1])$ with $\G$ again}{Step4of3Cproof}
	
	The overall idea is simple: for each $(w_n,z_m)$ cross point of $\G$ we subtract two small arcs from $\gamma([0,1/2])$ corresponding to the $NE$ and $SW$ segments that we added in the construction of $\tilde{\Psi}$, then we choose a point $g(w_n,z_m)$ and connect it up by four disjoint injective Lipschitz curve with the four ends where we subtracted the two arcs. The things we need to guarantee when we do this are:
	\begin{enumerate}
		\item the new added curves do not intersect $\Gamma_k$,
		\item the added curves do not intersect each other (except in the case that they share a common endpoint, then they intersect only at that point),
		\item the added curves do not intersect $\gamma([0,1/2])$ except at their respective endpoints where they connect up.
	\end{enumerate}
	
	We only need to deal with the rectangles $R$ of $\Gamma_k$ that are not contained completely in $P$ since for those $R\subset P$ the claim is obvious and we can put $g(w,z) = (w,z)$ for $(w,y) \in \G\cap P$.
	
	Firstly, we make the following observation using  Lemma~\ref{DegreesAreUseful}. Because $\deg((w,z), u, (-1,1)^2)=1$ for all $(w,z)\in (-1,1)^2\setminus u(\Gamma_k)$, there is exactly one rectangle $R$ of $\Gamma_k$ such that $\deg((w,z), u, \inter R)=1$. The mapping that assigns this $R$ to each $(w,z)\in (-1,1)^2\setminus u(\Gamma_k)$ is constant on components of $(-1,1)^2\setminus u(\Gamma_k)$ by the homotopic properties of the degree. As a shorthand  for the previous, when we write $R_{(w,z)}$ we mean the rectangle of $\Gamma$ such that $\deg((w,z), u, \inter R)=1$. By the choice of $\eta$ we have that $R_{(w,z)} = R_{(w_n,z_m)}$ for all $(w,z)$ on the $SW$ and $NE$ segments of $\tilde{\Psi}$ close to $(w_n,z_m)$.
	
	Let $L = NE$ (or $L=SW$) be a segment in $\tilde{\Psi}$. There exists exactly one pair $\xi_L, \zeta_L\in [0,1/2]\cap \psi^{-1}\circ u\circ\phi([0,1])$ of subsequent points (i.e. $\xi_L<\zeta_L$ and $\psi((\xi_L,\zeta_L))\cap u\circ\phi([0,1]) = \emptyset$) such that $L\subset \psi((\xi_L,\zeta_L))$. From \eqref{IntervalOrder} we have a pair of disjoint closed intervals $I_{\xi_L}$ and $I_{\zeta_L}$ and we call $I_{L}$ the open interval lying between them, i.e. $I_{L} = \{\lambda\in \er: s<\lambda<t, s\in I_{\xi_L}, t\in I_{\zeta_L} \}$. Note that there may be several $L =NE$ (or $SW$) segments which have the same interval $I_L$. Let $\xi, \zeta\in [0,1/2]\cap \psi^{-1}\circ u\circ\phi([0,1])$ be any subsequent pair, then for each of the $NE$ (or $SW$) segments $L\subset \psi((\xi,\zeta))$ let us choose a closed interval $J_L \subset I_{L}$ such that the intervals $J_L$ are pairwise disjoint and in the correct order, i.e. if $L\neq L'$ then for all $s\in J_{L}$ and all $t\in J_{L'}$ we have $s<t$ exactly when $a<b$ for all $a\in \psi^{-1}(L)$ and $b\in \psi^{-1}(L')$. 
	
	Let us prove that we can find curves satisfying (1). Let $SW$ and $NE$ be the pair of segments close to $(w_n,z_m)$, a cross point of $\G$. From the construction of $\tilde{\Psi}$ (see Figure~\ref{fig:wavy}) there are a pair of disjoint closed intervals $T_{SW} \supset \psi^{-1}(SW)$ and $T_{NE}\supset\psi^{-1}(NE)$ such that on each $T_L$ the curve $\psi$ starts close to $\partial(-1,1)^2$ and ends close to another side of $\partial(-1,1)^2$. In Figure~\ref{fig:wavy} $\psi(T_{NE})$ is depicted in blue and $\psi(T_{SW})$ is depicted in red. We split each $T_L$ into two parts, the part before $L$ and the part after $L$ i.e. the two subintervals $T_{L}^-\{\lambda \in T_L: \lambda <\sigma, \sigma \in \psi^{-1}(L)\}$ and $T_{L}^+\{\lambda \in T_L :\lambda >\sigma, \sigma \in \psi^{-1}(L)\}$  of $T_{L}$. Using Lemma~\ref{DegreesAreUseful} and the equivalence of the degree and the winding number we know that there are an even number of points in $\partial R\cap u^{-1}\circ\psi(T_{L}^{\pm})$ for every rectangle $R$ which is not completely contained in $P$ except for $R_{w_n,z_m}$ where there is an odd number of intersections. From the very definition of the three-curve property, the same is true for the number of intersections of $\gamma$ on the two intervals of $[0,1/2]\setminus J_L$ with $\partial R$. This means that the endpoints of $J_L$ lie inside the rectangle $R_L$. Since $R_{SW} = R_{NE}$ we see that the four points of $\gamma(\partial J_{SW} \cup \partial J_{NE})$ all lie in $\inter R_{(w_n,z_m)}$ and so can be connected inside $\inter R_{(w_n,z_m)}$, i.e. without intersecting $\Gamma_k$. 
	
	Now we turn our attention to property (2). Let us have vertexes $(w_{n},z_m)$ and $(w_{n+o},z_{m-o})$ of $\G$ whose $NE$ and whose $SW$ segments each lie on a component of $\psi([0,\tfrac{1}{2}])\setminus u(\Gamma_k)$, i.e. $I_{SW(n,m)} = I_{SW(n+o,m-o)}$ and $I_{NE(n,m)} = I_{NE(n+o,m-o)}$. We want to prove that a Lipschitz injective curve from $\gamma(J_{SW(n,m)})$ to $\gamma(J_{NE(n,m)})$ does not disconnect $\gamma(J_{SW(n+o,m-o)})$ from $\gamma(J_{NE(n+o,m-o)})$. By $(a_{NE},b_{NE})$ let us denote the component of $\gamma^{-1}(\er^2\setminus \Gamma_k)$ containing $I_{NE(n,m)}$ and by $(a_{SW},b_{SW})$ let us denote the component of $\gamma^{-1}(\er^2\setminus \Gamma_k)$ containing $I_{SW(n,m)}$. Let $\Omega$ be the bounded component of $\er^2\setminus u(\Gamma_k)$ containing $(w_n,z_m)$ (and $(w_{n+o}, z_{m-o})$). It is possible to find a continuous injective $v:\partial R_{(w_n,z_m)} = \partial R_{(w_{n+k},z_{m-k})} \to \partial\Omega$ such that
	$$
	v(x,y) = u(x,y) \quad \text{ and }\quad  v^{-1}(u(x,y)) =\{(x,y)\} 
	$$
	for all $(x,y) \in \{\gamma(a_{SW}),\gamma(a_{NE}),\gamma(b_{SW}),\gamma(b_{NE})\}$ and, by homotopy,
	$$
	\deg((w,z),v, \partial R_{w_n,z_m}) = \deg((w,z),u, \partial R_{w_n,z_m})
	$$
	for all $(w,z)$ in $\Omega$. Then clearly the order of the points $\gamma(a_{SW})$, $\gamma(a_{NE})$, $\gamma(b_{SW})$, $\gamma(b_{NE})$ on $\partial R_{(w_n,z_m)}$ is the same as the order of $v\big(\gamma(a_{SW})\big)$, $v\big(\gamma(a_{NE})\big)$, $v\big(\gamma(b_{SW})\big)$, $v\big(\gamma(b_{NE})\big)$ in $\partial\Omega$ because $v$ is sense-preserving (i.e. the degree is $+1$ in $\Omega$). Therefore clearly, because injective Lipschitz curves in $\inter v(R_{(w_n,z_m)}) \setminus \psi([0,1/2])$ connecting $SW(n,m)$ to $NE(n,m)$ do not disconnect $SW(n+o,m-o)$ from $NE(n+o,m-o)$, also injective Lipschitz curves in $\inter R_{(w_n,z_m)}\setminus \gamma([0,1/2])$ from $\gamma(J_{SW(n,m)})$ to $\gamma(J_{NE(n,m)})$ do not disconnect $\gamma(J_{SW(n+o,m-o)})$ from $\gamma(J_{NE(n+o,m-o)})$.
	
	We show condition (3) by contradiction. Suppose that there is a pair $SW$ and $NE$ and a component $C$ of $[0,1/2]\setminus\gamma^{-1}( \phi([0,1]))$ such that $\gamma(J_{SW})$ is disconnected from $\gamma(J_{NE})$ inside $\inter R$ by $\gamma(\overline{C})$, where $R$ is the rectangle containing $\gamma(SW)$ and $\gamma(NE)$. Following the argument of the previous paragraph there must be a component of $\psi([0,1/2]) \setminus u(\Gamma_k)$ which disconnects $SW$ from $NE$ inside $\{(w,z): \deg((w,z), u,\partial R_{w_n,z_m)}) = 1\}$ which is clearly impossible from the construction of $\tilde{\Psi}$.
	
	Thus, it is possible to construct the $g:\G\to\er^2$ which we wanted.
\end{proof}

\section{The equivalence of the QM condition}
We use the QM condition to express the preimage of closed connected sets which do not disconnect the preimage, which allows us to prove the $\INV$ condition. Further we show that the QM condition implies the NCL condition.
\begin{lemma}\label{OneRing}
	Let $u\in W_{\id}^{1,p}((-1,1)^2, \er^2)$ satisfy the QM condition (Definition~\ref{defQM}), then $u$ satisfies the NCL condition (Definition~\ref{defNCL}).
\end{lemma}

\begin{proof}
	
	Let us assume that $u$ does not satisfy the NCL condition. Then we can find two Lipschitz continuous injective curves $\phi_1,\phi_2:[0,1] \to \er^2$ with disjoint images and a Jordan curve $\beta:\mathbb{S} \to \er^2$ such that $u\circ\phi_i((0,1))$ lies in the interior of $\beta(\mathbb{S})$, but $u\circ\phi_1$ crosses $u\circ\phi_2$, i.e.
	$$
	\begin{aligned}
		u(\phi_1(0)) = \beta(1,0), & \quad u(\phi_1(1)) = \beta(-1,0),\\
		u(\phi_2(0)) = \beta(0,1),& \quad u(\phi_2(1)) = \beta(0,-1). 
	\end{aligned}
	$$ 
	We let $C_1$ be the image of $[0,\frac{3}{2}\pi]$ in the mapping $\beta(\cos(\cdot), \sin(\cdot))$ and we let $C_2$ be a curve in the interior of $\beta(\mathbb{S})$ and stays very close to $\beta$ as shown in Figure~\ref{fig:hhh}. It is possible to choose $\phi_i$ and $C_i$ so that $u^{-1}(C_i) \cap \phi_j([0,1])$ is exactly two points, one near $\phi_j(0)$ and the other near $\phi_j(1)$. A necessary condition for the existence of the disjoint Jordan domains $V^1_{k^{-1},m^{-1}}$ and $V^2_{k^{-1},m^{-1}}$ is that there are a pair of injective continuous curves $\gamma_1$, $\gamma_2$ with disjoint images and such that $\gamma_i([0,1]) \cap \phi_i([0,1]) = \phi_i^{-1}\circ u^{-1}(C_i)$. We proceed to follow that such curves cannot exist.
	
	\begin{figure}
		\begin{tikzpicture}[line cap=round,line join=round,>=triangle 45,x=0.3cm,y=0.3cm]
			\clip(-20.847376750036133,-1.1295955883538953) rectangle (27.103770934687542,13.413527886841413);
			\fill[line width=0.7pt,fill=black!10] (4.84,4.94) -- (2.72,8.24) -- (5.,12.) -- (13.64,11.76) -- (11.14,9.78) -- (16.82,7.24) -- (18.84,12.5) -- (26.,8.32) -- (22.,1.) -- (15.684,3.83) -- (11.196,0.904) -- (8.49,4.512) -- cycle;
			\draw [line width=0.7pt] (4.84,4.94)-- (2.72,8.24);
			\draw [line width=0.7pt] (2.72,8.24)-- (5.,12.);
			\draw [line width=0.7pt] (5.,12.)-- (13.64,11.76);
			\draw [line width=0.7pt] (13.64,11.76)-- (11.14,9.78);
			\draw [line width=0.7pt] (11.14,9.78)-- (16.82,7.24);
			\draw [line width=0.7pt] (16.82,7.24)-- (18.84,12.5);
			\draw [line width=0.7pt] (18.84,12.5)-- (26.,8.32);
			\draw [line width=0.7pt] (26.,8.32)-- (22.,1.);
			\draw [line width=0.7pt] (22.,1.)-- (15.684,3.83);
			\draw [line width=0.7pt] (15.684,3.83)-- (11.196,0.904);
			\draw [line width=0.7pt] (11.196,0.904)-- (8.49,4.512);
			\draw [line width=0.7pt] (8.49,4.512)-- (4.84,4.94);
			\draw [line width=0.7pt,color=ffqqqq] (4.086019031857675,10.492733140256517)-- (7.,8.);
			\draw [line width=0.7pt,color=ffqqqq] (7.,8.)-- (10.98,7.3);
			\draw [line width=0.7pt,color=ffqqqq] (18.96,4.52)-- (21.46,3.3);
			\draw [line width=0.7pt,color=ffqqqq] (21.46,3.3)-- (21.140010898421124,1.3853339388011745);
			\draw [line width=0.7pt,color=yqqqyq] (10.98,7.3)-- (13.42,5.02);
			\draw [line width=0.7pt,color=yqqqyq] (13.42,5.02)-- (15.3,6.36);
			\draw [line width=0.7pt,color=yqqqyq] (15.3,6.36)-- (18.96,4.52);
			\draw [line width=0.7pt,color=qqqqff] (5.802764780780162,4.827105938034546)-- (7.78,6.66);
			\draw [line width=0.7pt,color=qqqqff] (7.78,6.66)-- (10.98,7.3);
			\draw [line width=0.7pt,color=qqqqff] (18.96,4.52)-- (22.7,7.1);
			\draw [line width=0.7pt,color=qqqqff] (22.7,7.1)-- (22.445478425325145,10.395125723762694);
			\draw [line width=0.7pt] (4.086019031857675,10.492733140256517)-- (2.72,8.24);
			\draw [line width=0.7pt] (2.72,8.24)-- (4.84,4.94);
			\draw [line width=0.7pt] (4.84,4.94)-- (5.802764780780162,4.827105938034546);
			\draw [line width=0.7pt] (5.802764780780162,4.827105938034546)-- (8.49,4.512);
			\draw [line width=0.7pt] (8.49,4.512)-- (11.196,0.904);
			\draw [line width=0.7pt] (11.196,0.904)-- (15.684,3.83);
			\draw [line width=0.7pt] (15.684,3.83)-- (22.,1.);
			\draw [line width=0.7pt] (22.,1.)-- (26.,8.32);
			\draw [line width=0.7pt] (26.,8.32)-- (22.445478425325145,10.395125723762694);
			\draw [line width=0.7pt,color=zzttqq] (4.442136480355891,10.188096364862266)-- (3.2227361732564836,8.237668587313);
			\draw [line width=0.7pt,color=zzttqq] (3.2227361732564836,8.237668587313)-- (5.07137392810552,5.290190208952269);
			\draw [line width=0.7pt,color=zzttqq] (5.07137392810552,5.290190208952269)-- (8.797784952291813,4.886570607576248);
			\draw [line width=0.7pt,color=zzttqq] (8.797784952291813,4.886570607576248)-- (11.280673417855791,1.4494740258347119);
			\draw [line width=0.7pt,color=zzttqq] (11.280673417855791,1.4494740258347119)-- (15.672130632039908,4.2884385394697775);
			\draw [line width=0.7pt,color=zzttqq] (15.672130632039908,4.2884385394697775)-- (21.900062898676165,1.5857509520615645);
			\draw [line width=0.7pt,color=zzttqq] (21.900062898676165,1.5857509520615645)-- (25.442094457329127,8.149420807195797);
			\draw [line width=0.7pt,color=zzttqq] (25.442094457329127,8.149420807195797)-- (22.49224210685591,9.789706673740241);
			\draw [line width=0.7pt] (-20.,12.)-- (-20.,0.);
			\draw [line width=0.7pt] (-20.,0.)-- (-8.,0.);
			\draw [line width=0.7pt] (-8.,0.)-- (-8.,12.);
			\draw [line width=0.7pt] (-8.,12.)-- (-20.,12.);
			\draw [line width=0.7pt,color=ffqqqq] (-16.80864861475176,8.43449719049069)-- (-11.217508362842993,8.488780105557765);
			\draw [line width=0.7pt,color=qqqqff] (-16.781507157218226,3.9561566974569145)-- (-11.163225447775918,3.9832981549904525);
			\draw [line width=0.7pt] (-16.80864861475176,8.43449719049069)-- (-16.781507157218226,3.9561566974569145);
			\draw [line width=0.7pt] (-16.781507157218226,3.9561566974569145)-- (-16.737963009767444,1.9724834760420693);
			\draw [line width=0.7pt] (-16.737963009767444,1.9724834760420693)-- (-10.,2.);
			\draw [line width=0.7pt] (-10.,2.)-- (-10.,10.);
			\draw [line width=0.7pt] (-10.,10.)-- (-11.217508362842993,8.488780105557765);
			\draw [line width=0.7pt] (-11.217508362842993,8.488780105557765)-- (-11.163225447775918,3.9832981549904525);
			\draw [line width=0.7pt,color=zzttqq] (-16.44693193450557,8.438009002920264)-- (-16.38047470156619,3.9580940523151367);
			\draw [line width=0.7pt,color=zzttqq] (-11.99989978864716,3.979256249962089)-- (-11.987163617282516,8.481307724446703);
			\draw [line width=0.7pt,color=zzttqq] (-11.987163617282516,8.481307724446703)-- (-11.352152175023669,10.030885206190504);
			\draw [line width=0.7pt,dotted,color=zzttqq] (-11.352152175023669,10.030885206190504)-- (-10.72092246749564,9.081165524103262);
			\draw [line width=0.7pt,dotted,color=zzttqq] (-10.72092246749564,9.081165524103262)-- (-10.575553521437282,2.780443831178698);
			\draw [line width=0.7pt,dotted,color=zzttqq] (-10.575553521437282,2.780443831178698)-- (-16.376716694930145,2.799101272182173);
			\draw [line width=0.7pt,dotted,color=zzttqq] (-16.376716694930145,2.799101272182173)-- (-16.38047470156619,3.9580940523151367);
			\draw [->,line width=0.7pt] (-5.337105330359512,5.452934410647067) -- (0.827664307172034,5.452934410647067);
			\draw (2.5,4.6) node[anchor=north west] {$\psi_1([0,1])$};
			\draw [color=zzttqq] (-2.5,11.5) node[anchor=north west] {$\psi_2([0,1])$};
			\draw [color=black] (-3.3,7) node[anchor=north west] {$u$};
			\draw [color=red] (-15,10) node[anchor=north west] {$\phi_1$};
			\draw [color=blue] (-15,5.5) node[anchor=north west] {$\phi_2$};
			\draw [color=black] (-15,2.2) node[anchor=north west] {$\gamma_1$};
			\draw [color=zzttqq] (-16.5,7.5) node[anchor=north west] {$\gamma_2$};
			\draw [color=red] (4,12) node[anchor=north west] {$u\circ\phi_1([0,1])$};
			\draw [color=blue] (13.5,8) node[anchor=north west] {$u\circ\phi_2([0,1])$};
			\begin{scriptsize}
				\draw [fill=ffqqqq] (-10.72092246749564,9.081165524103262) circle (2.5pt);
			\end{scriptsize}
		\end{tikzpicture}
		\caption{The curves $\gamma_1$ and $\gamma_2$ must cross and therefore it is impossible to find a pair of disjoint Jordan domains $V^1_{k^{-1},m^{-1}}$ and $V^2_{k^{-1},m^{-1}}$.}\label{fig:hhh}
	\end{figure}
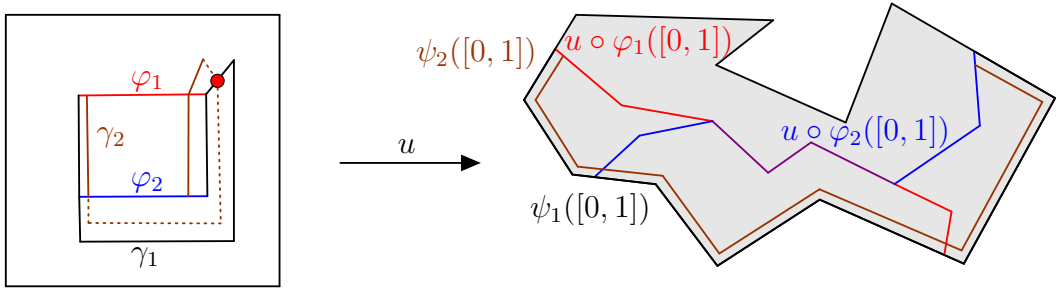
	
	Up to renumbering and a change of orientation of $\phi_i$ we have $\gamma_1(0) = \phi_1(0)$, $\gamma_1(\tfrac{1}{3}) = \phi_2(0)$, $\gamma_1(\tfrac{2}{3}) = \phi_1(1)$ and $\gamma_1(1) = \phi_2(1)$ and $\gamma_1$ has no other intersections with $\phi_1([0,1])$ or $\phi_2([0,1])$. Then  $\phi_2((0,1))$ lies in the bounded component of $\er^2 \setminus \gamma_1([0,1])\cup \phi_1([0,1])$ (see Figure~\ref{fig:hhh}). The curve $\gamma_2$ starts near $\gamma_1(0)$, then it intersects $\phi_2((0,1))$ once. After this it lies in the bounded component of $\er^2 \setminus \gamma_1([\tfrac{1}{3},1])\cup \phi_2([0,1])$ and then it needs to intersect $\phi_1$ without intersecting $\phi_2([0,1])$ or $\gamma_1([0,1])$, which is impossible.  
\end{proof}

\begin{thm}\label{PointwiseChar}
	Let $u\in W_{\id}^{1,p}((-1,1)^2, \er^2)$. Then $u\in \overline{H\cap W_{\id}^{1,p}((-1,1)^2, \er^2)}$ if and only if $u$ satisfies the QM condition.
\end{thm}	

\begin{proof}[Proof of the necessity of the QM condition]
	Let us prove that each $u\in \overline{H\cap W_{\id}^{1,p}((-1,1)^2, \er^2)}$ satisfies the QM property. Note that the $\INV$ condition is immediate from Theorem~\ref{HomImpINV}. Let $\delta >0$, let $C_1,C_2$ be disjoint closed connected sets such that $(-1,1)^2\setminus (C_1 \cup C_2)$ is connected and let $r>0$  be such that $C_1+B(0,4r)\cap C_2+B(0,4r) = \emptyset$. Using Theorem~\ref{QuasiMonotne} and Theorem~\ref{Young} we find a homeomorphism $h$ with $\|u-h\|_{L^{\infty}((-1,1)^2\setminus U_{\delta})}<r$.  Then $\overline{h^{-1}(C_1 +B(0,2r))}\cap \overline{h^{-1}(C_2 +B(0,2r))} = \emptyset$. Further 
	$$
	u^{-1}(C_i)\setminus U_{\delta} \subset h^{-1}(C_i +B(0,2r)) \subset \overline{h^{-1}(C_i +B(0,2r))}\subset u^{-1}(C_i+B(0,3r))\cup U_{\delta}.
	$$
	Then it suffices to put $V_i =h^{-1}(C_i +B(0,2r))$.
\end{proof}

Before starting the proof of the sufficiency of the condition, let us add a short comment about the strategy of our proof.  We want to prove that $u$ satisfies the three-curve property, so we take $\phi$ and $\psi$, a good test pair for $u$. For any closed interval $I\subset[0,1]$, $C = \psi(I)$ is a closed set that does not disconnect the plane. Then by the $\INV$ condition and Lemma~\ref{IDidThisOne} also $u^{-1}(C)$ is closed and connected. We can find curves $\gamma_I$ inside $V_{k^{-1},m^{-1}}$ which cross $\phi$ exactly at the points $\phi([0,1])\cap u^{-1}\circ \psi(I)$.

We may assume that $\phi([0,1])$ disconnects $[-1,1]^2$ into two parts a left and right side. The NCL condition allows us to say that any point in $[-1,1]^2\setminus u\circ\phi([0,1])$ lies on one of two `sides' of $u\circ\phi([0,1])$ even if $[-1,1]^2 \setminus u\circ\phi([0,1])$ has many components, not just two. This allows us to prove that we can connect the curves $\gamma_I$ found using the $\INV$ condition without making any additional intersections with $\phi([0,1])$.

It is the QM condition that guarantees that the curves constructed do not intersect each other and therefore our final $\gamma$ is injective.
\begin{proof}[Proof the of sufficiency of the QM condition]			
	Let $\delta>0$ be fixed, let $\phi,\psi$ be a good test pair for $u$, call $F = u^{-1}\circ\psi([0,1])\cap \phi([0,1])$ and let $r>0$ be such that $B((w,z), 5r)$ are pairwise disjoint for $(w,z)\in u(F)$. By $K\in \en$ let us denote the number of points in $u(F)$.
	
	Because $(u\circ\phi)'(t)$ and $\psi'(s)$ are linearly independent whenever $u\circ\phi(t)=\psi(s)$, we have
	$$
	\alpha :=\frac{1}{3}\min_{u\circ\phi(t)=\psi(s)}\Big\langle \frac{(u\circ\phi)'(t)}{|(u\circ\phi')(t)|}, \frac{[\psi'(s)]^{\bot}}{|\psi'(s)|} \Big\rangle \in (0,1)
	$$
	where $(v_1,v_2)^{\bot} = (-v_2,v_1)$. Assuming that $r>0$ is sufficiently small, we can find exactly $K+1$ closed intervals $I_1, I_2,\dots, I_{K+1}$, components of $[0,1]\setminus \psi^{-1}\big(\bigcup_{i=1}^K B((w_i,z_i),2r\alpha^{-1})\big)$, such that
	$$\psi(I_i) \cap \overline{B((w_{i-1},z_{i-1}),2r\alpha^{-1})} \neq \emptyset \neq \psi(I_i) \cap\overline{B((w_{i},z_{i}),2r\alpha^{-1})} $$ for $i=2,\dots, K$ and 
	$$
	B((w_i,z_i),4r\alpha^{-1}) \cap \psi(I_j) = \emptyset  \ \ \text{ for all $j\neq i-1$ and $j\neq i$.}
	$$
	Similarly,  $\psi(I_1)$ intersects $\overline{B((w_{1},z_{1}),2r\alpha^{-1})}$ but not $B((w_i,z_i),4r\alpha^{-1})$, $i\neq 1$ and $\psi(I_{K+1})$ intersects $\overline{B((w_{K},z_{K}),2r\alpha^{-1})}$ but not $B((w_i,z_i),4r\alpha^{-1})$, $i\neq K$. Because $\psi'(s)$ exists and $\psi'(s)\neq 0$ whenever $\psi(s)\in u(F)$, we may assume that $r$ is so small that $\psi(I_i)+B(0,r)$ are pairwise disjoint. Again, because $\psi'(s)$ exists and $\psi'(s)\neq 0$ whenever $\psi(s)\in u(F)$,  we may assume that $r$ is so small that $\psi(\hat{I}_i) \subset (B((w_i,z_i),3r\alpha^{-1}))$, where $\hat{I}_i = [\max I_{i}, \min I_{i+1}]$ for $i=1,2,\dots, K$.  Moreover, for $r$ small enough we have that 
	\begin{equation}\label{ACDC}
		[\psi(I_i) + B(0, 2r) ]\cap u\circ\phi([0,1]) = \emptyset   \ \ \text{ for $i=1,\dots K+1$. }
	\end{equation}
	
	Having fixed $r$ we fix $\delta>0$ and use the QM condition (see Remark~\ref{QMRemarks} (4)) to find domains $V_1, V_2, \dots V_K, \tilde{V}_1,\tilde{V}_2,\dots, \tilde{V}_{K+1}$ with pairwise disjoint closures such that $V_i \supset u^{-1}(w_i,z_i)  \setminus U_{\delta}$ and $\tilde{V}_i \supset u^{-1}(\psi(I_i))\setminus U_{\delta}$. From the QM condition we have that $\overline{\tilde{V}_i}\subset u^{-1}(\psi(I_i) + B(0,r))\cup U_{\delta}$, the choice of $\phi$ ensures that $\phi([0,1]) \cap U_{\delta} = \emptyset$ so by \eqref{ACDC} we see that $\overline{\tilde{V}_i} \cap \phi([0,1]) = \emptyset$. Thus, we have $\dist(\tilde{V}_i, \phi([0,1]))>0$. Similarly we find domains $\hat{V}_1, \hat{V}_2,\dots, \hat{V}_K$ with pairwise disjoint closures such that $\hat{V}_i\supset\overline{u^{-1}(B((w_i,z_i), 3r\alpha^{-1}))}\setminus U_{\delta}$, $\hat{V}_i\supset V_i$ and $\hat{V}_i \cap \tilde{V}_j = \emptyset$ exactly when $j\neq i-1$ and $j\neq i$.
	
	Without loss of generality we may assume that $\phi^{-1}(\partial(-1,1)^2) = \{0,1\}$ and so $\phi([0,1])$ disconnects $[-1,1]^2$ into two components, further we may assume that $|\phi(0) - \phi(1)|\geq 1$. We call the two components of $[-1,1]^2\setminus \phi([0,1])$ the left hand and right-hand side respectively in the obvious way. Using the NCL condition (thanks to Lemma~\ref{OneRing}), for any $\epsilon>0$ we find a continuous injective $g:[0,1]\to [-1,1]^2$ such that $\|u\circ\phi - g\|_{L^{\infty}((0,1),\er^2)}<\epsilon$. Because $u$ is the identity at $\partial(-1,1)^2$ we may assume that $g(0) = u\circ\phi(0)$ and $g(1) = u\circ\phi(1)$ (at least for all $\epsilon$ sufficiently small) and that there are exactly two components of $[-1,1]^2\setminus g([0,1])$ which we call the left hand and right hand side of $g([0,1])$ in the obvious way. Note that 
	\begin{equation}\label{BoundarySideUniqueness}
		\begin{aligned}
			&\text{$(x,y)\in \partial(-1,1)^2\setminus \phi([0,1])$ is on the left side of $\phi([0,1])$ exactly when }\\
			&\text{$(x,y)$ is on the left side of $g([0,1])$.}
		\end{aligned}
	\end{equation} In the course of the first half of the proof of Theorem~\ref{ThreeCurveChar} we showed that, if $r$ and $\epsilon$ are sufficiently small, then we can modify $g$ so that $g^{-1}(\psi([0,1])) = \phi^{-1}(F)$ holds and there exists $g'$ at each point $\phi^{-1}(F)$ and $g'(t)$ is linearly independent of $\psi'(s)$ whenever $g(t) = \psi(s)$. In fact, we may assume that $g$ is linear on each interval in $g^{-1}B((w,z),r \alpha^{-1})$ for each $(w,z)\in u(F)$.
	
	Let  $(w,z)\in (-1,1)^2\setminus u\circ\phi([0,1])$ and let $g_1,g_2:[0,1]\to \er^2$ be continuous and injective satisfying $\|g_i-u\circ\phi\|_{\infty}< \dist\big((w,z), u\circ\phi([0,1])\big)$, then 
	\begin{equation}\label{SideUniqueness}
		\text{$(w,z)$ is on the left of $g_1$ exactly when $(w,z)$ is on the left side of $g_2$.} 
	\end{equation}
	Especially, the side of $g([0,1])$ which contains $\psi(I_i)$ is invariant of the choice of continuous injective $g$ satisfying $\|u\circ\phi - g\|_{\infty}< r$, because $\dist(\psi(I_i), u\circ \phi([0,1])) \geq 2 r$. 
	
	Further, let us prove that $u^{-1}(\psi(I_i))$ lies on the left side of $\phi([0,1])$ exactly when $\psi(I_i)$ lies on the left side of any injective continuous $g([0,1])$ satisfying $\|g-u\circ\phi\|_{\infty}< r$. We can find a Lipschitz continuous injective path $\sigma$ parametrized from $[0,1]$ which on $[0,1/3]$ goes from some point $(x,y) \in \tilde{V}_i \setminus S_u$ to $\partial(-1,1)^2$ inside $[-1,1]^2\setminus \phi([0,1])$, on $[1/3,2/3]$ the path $\sigma$ follows $\partial(-1,1)^2$ till it gets to $\phi(0)$ and on $[2/3,1]$ we define $\sigma(t) = \phi(3t-2)$. We may assume that $u\circ\sigma\in W^{1,1}((0,1),\er^2)$. We use the NCL condition on $\sigma$ and get a $\|g_{\sigma}-u\circ\sigma\|_{\infty}< r$, we may assume that $g_{\sigma}(t) = \sigma(t)$ on $[1/3,2/3]$. Because $(x,y)\in (-1,1)^2\setminus S_u$ the value $u(x,y)$ is well defined. By \eqref{SideUniqueness} we have that $u(x,y)$ is on the left side of $g_{\sigma}([2/3,1])$ exactly when $u(x,y)$ is on the left side of any $g([0,1])$ where $g$ is injective continuous satisfying $\|g-u\circ\phi\|_{\infty}< r$. Because $g_{\sigma}$ is injective $u(x,y)$ is on the left side of $g_{\sigma}([2/3,1])$ exactly when $g_{\sigma}(1/2)$ is on the left side of $g_{\sigma}([2/3,1])$. By \eqref{BoundarySideUniqueness} $\sigma(1/2)$ is on the left side of $\phi([0,1])$ exactly when $g_{\sigma}(1/2)$ is on the left side of $g_{\sigma}([2/3,1])$. Because $\sigma([0,1/2])\cap \phi([0,1]) = \emptyset$, the point $\sigma(0) = (x,y)$ lies on the left of $\phi([0,1])$ exactly when $\sigma(1/2)$ lies on the left of $\phi([0,1])$. Collectively $u(x,y)$ lies on the left side of any continuous injective $g$ satisfying $\|g-u\circ\phi\|_{\infty}< r$ exactly when $(x,y)$ lies on the left side of $\phi([0,1])$. Because $u^{-1}(\psi(I_i))\subset \tilde{V}_i\cup U_{\delta}$, there is a component of $\tilde{V}_i\cup U_{\delta}$ containing $u^{-1}(\psi(I_i))$. Because this component does not intersect $\phi([0,1])$ it lies entirely on one side of $\phi([0,1])$ where $(x,y)$ lies.

	Let $M_{i}$ denote the number of points $(x,y) \in F$ such that $u(x,y) = (w_i,z_i)\in u(F)$ for $i=1,\dots , K$. By our adaption of $g$ there are $M_{i}$ intervals $J_{i,1}, \dots, J_{i,M_{i}}$ such that $g$ is linear on each $J_{i,o}$, $g(J_{i,o})$ disconnects $B((w_i,z_i),\alpha^{-1}r)$ and $g^{-1}(B((w_i,z_i),r)\cap \psi([0,1])) \subset \bigcup_{k=1}^{M_{i}}J_{i,k}$. Then the number of intersections of $g([0,1])$ with $\psi([0,1])\cap B((w_i,z_i),r)$ is also $M_i$. This means that any continuous $\gamma_{i}([0,1])$ that starts at the side of $\phi([0,1])$ containing $\tilde{V}_i$, crosses $\phi([0,1])$ exactly once at every point of $\phi([0,1])\cap u^{-1}(w_i,z_i)$ and has no other intersections with $\phi([0,1])$ ends on the side of $\phi([0,1])$ which contains $\tilde{V}_{i+1}$. Thus we find a continuous $\gamma_i:[0,1] \to \er^2$ with $\gamma_i([0,1])\subset V_i$ such that $\gamma_i(0)$ lies on the same side of $\phi([0,1])$ as $\tilde{V}_i$,  $\gamma_i(1)$ lies on the same side of $\phi([0,1])$ as $\tilde{V}_{i+1}$, $\gamma_i$ crosses $\phi([0,1])$ exactly once at each point of $\phi([0,1])\cap u^{-1}(w_i,z_i)$ and has no other intersections with $\phi([0,1])$. Since $V_i$ is open we may assume that $\gamma_i$ is finitely piecewise linear, i.e. $[0,1]$ is the union of a finite number of non-overlapping closed intervals each of which is mapped to a segment at constant, non-zero speed. By slightly moving the endpoints of the segments in the image we can ensure that no three segments intersect at the same point and if any pair of segments intersect, then their intersection is a singleton (which we call a crossing point of $\gamma_i$). 
	
	\begin{figure}
		\begin{tikzpicture}[line cap=round,line join=round,>=triangle 45,x=1.0cm,y=1.0cm]
			\clip(-0.1803666981981683,-0.42848439255843424) rectangle (4.1713425399820485,3.3201958508791223);
			\draw [line width=0.7pt] (4.,0.)-- (0.,0.);
			\draw [line width=0.7pt,dash pattern=on 1pt off 1pt] (2.,2.) circle (0.5cm);
			\draw [line width=0.7pt,color=qqqqff] (1.5527864045000421,2.223606797749979)-- (1.646446609406726,1.646446609406726);
			\draw [line width=0.7pt,color=qqqqff] (2.353553390593274,1.646446609406726)-- (2.447213595499958,2.223606797749979);
			\draw [line width=0.7pt] (0.,3.)-- (1.5527864045000421,2.223606797749979);
			\draw [line width=0.7pt] (2.447213595499958,2.223606797749979)-- (4.,3.);
			\draw [line width=0.7pt] (0.,0.)-- (1.646446609406726,1.646446609406726);
			\draw [line width=0.7pt] (2.353553390593274,1.646446609406726)-- (4.,0.);
			\draw [line width=0.7pt,dotted] (1.5527864045000421,2.223606797749979)-- (2.353553390593274,1.646446609406726);
			\draw [line width=0.7pt,dotted] (1.646446609406726,1.646446609406726)-- (2.447213595499958,2.223606797749979);
		\end{tikzpicture}
		\caption{If we have a piece-wise linear curve that crosses itself we can replace it with a curve that does not cross itself by replacing the dotted segments with the blue segments and reversing the orientation on the ``disconnected part'' of the curve.}\label{fig:Uncross}
	\end{figure}
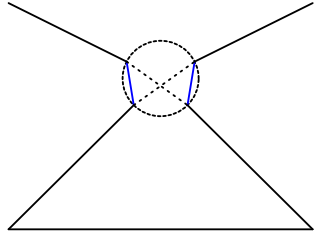
	
	Let us prove that it is possible to assume that $\gamma_i$ is injective (our approach is depicted in Figure~\ref{fig:Uncross}). There is a finite number of crossing points of $\gamma_i$ and so we find finitely many pairwise disjoint balls $B_j$ each centered at a crossing point and $\gamma_i([0,1]) \cap B_j$ is a pair of oriented segments $S_1 = ab$ and $S_2 = cd$ with $\gamma_i^{-1}(a)< \gamma_i^{-1}(b)<\gamma_i^{-1}(c)<\gamma_i^{-1}(d)$. On $[\gamma_i^{-1}(a),\gamma_i^{-1}(b)]$ we replace $ab$ with the segment $ac$ on $[\gamma_i^{-1}(b),\gamma_i^{-1}(c)]$ we reverse the orientation of $\gamma_i$ and on $[\gamma_i^{-1}(c),\gamma_i^{-1}(d)]$ we replace $cd$ with $bd$. This gives a curve $\gamma_i$, which is finitely piecewise affine, continuous,  $\gamma_i([0,1])\subset V_i$, $\gamma_i(0)$ lies on the same side of $\phi([0,1])$ as $\tilde{V}_i$,  $\gamma_i(1)$ lies on the same side of $\phi([0,1])$ as $\tilde{V}_{i+1}$, $\gamma_i$ crosses $\phi([0,1])$ exactly once at each point of $\phi([0,1])\cap u^{-1}(w_i,z_i)$ and has no other intersections with $\phi([0,1])$ but does not have a crossing point inside $B_j$. We do this for all the crossing points and we have that $\gamma_i$ is injective.
	
	For each $i=1,\dots K$ let us find a point $(x_i^-,y_i^-) \in \tilde{V}_{i-1}\cap\hat{V}_i \neq \emptyset$ and a point $(x_i^+,y_i^+) \in \tilde{V}_i\cap\hat{V}_i \neq \emptyset$. It holds that $\gamma_i([0,1])$ does not disconnect $\hat{V}_i$, $\gamma_i(0)$ is on the same side of $\phi([0,1])$ as $(x_i^-,y_i^-)$ and $\gamma_i(1)$ is on the same side of $\phi([0,1])$ as $(x_i^+,y_i^+)$ we can extend $\gamma_i$ onto $[-1,2]$ in such a way that each $\gamma_i$ is continuous injective and  $\gamma_i(-1) = (x_i^-,y_i^-)$ and $\gamma_i(2) = (x_i^+,y_i^+)$. See Figure~\ref{fig:Connecting}.
	
	\begin{figure}
		\begin{tikzpicture}[line cap=round,line join=round,>=triangle 45,x=0.3cm,y=0.3cm]
			\clip(-7.678286068693612,-4.927321455562875) rectangle (28.967724581267998,22.598491773849407);
			\fill[line width=0.7pt,color=sqsqsq,fill=sqsqsq!10] (27.81200457833495,-8.858627502902186) -- (21.190330231294695,-7.773107118141484) -- (15.556972727460344,-2.898783063662781) -- (14.919691483576846,2.2873677486305253) -- (17.029168756378688,2.64788857556125) -- (17.783771740181518,-0.6163829600240025) -- (23.723211129069657,-3.6119456432254617) -- cycle;
			\fill[line width=0.7pt,color=sqsqsq,fill=sqsqsq!10] (7.018942621152939,23.525317512034597) -- (6.402905154957156,16.30887862231254) -- (2.618675005468776,15.164809042234655) -- (-3.8057157134301014,19.2130552486641) -- (-2.5736407810385358,37.43016317759661) -- (4.412781175304548,42.757805945552676) -- (6.3695115018719095,35.517903737253405) -- cycle;
			\fill[line width=0.7pt,color=sqsqsq,fill=sqsqsq!10] (3.5990432699173804,17.320146817467275) -- (4.635893155247703,14.313282150009323) -- (12.96525390073463,13.691172218811126) -- (12.343143969536436,12.930815636235552) -- (7.262579531417855,12.377829030726044) -- (5.8109896919554025,6.329538032965793) -- (7.1588945428848225,2.907933411375708) -- (12.896130575045943,2.0093301774227568) -- (17.147215104900265,1.0070419549367722) -- (15.16296582886637,6.489359518627789) -- (10.476814175941856,6.329538032965793) -- (9.820142581899317,9.336402700423745) -- (15.868433579659536,10.373252585754075) -- (18.7370515957401,14.313282150009323) -- (16.42142018516904,17.804010097288096) -- cycle;
			\draw [line width=0.7pt,dotted] (12.,-4.927321455562875) -- (12.,22.598491773849407);
			\fill[line width=0.7pt,color=sqsqsq,fill=sqsqsq!20] (10.739524899359601,14.337522518459508) -- (10.,16.) -- (14.985205574188992,16.230866603180722) -- (16.534305279869987,13.993278139419289) -- (13.895098373894958,11.153262012337471) -- (8.874867846225072,10.5221473174304) -- (10.195034885060778,5.461127281889825) -- (13.263983678987888,5.301107568653727) -- (14.29671681610855,3.8093819261461057) -- (12.891052268360982,3.1208931680656655) -- (10.423967551906067,3.4077634839325155) -- (7.55526439323756,6.247779611014333) -- (7.469203298477505,11.526193422964377) -- (13.005800394707721,12.501552496911668) -- (14.640961195148769,14.366209550046193) -- cycle;
			\draw [line width=0.7pt] (11.203862180634122,15.3221514630554)-- (14.907299809830421,15.309597437193718);
			\draw [line width=0.7pt] (14.907299809830421,15.309597437193718)-- (15.547555128776223,14.116964980333883);
			\draw [line width=0.7pt] (15.547555128776223,14.116964980333883)-- (14.,12.);
			\draw [line width=0.7pt] (14.,12.)-- (8.291328180723813,10.97845851491327);
			\draw [line width=0.7pt] (8.291328180723813,10.97845851491327)-- (8.680502982435966,6.459009204707583);
			\draw [line width=0.7pt] (8.680502982435966,6.459009204707583)-- (9.93693893467331,4.311427139254739);
			\draw [line width=0.7pt] (9.93693893467331,4.311427139254739)-- (13.057553607539958,4.123720993518399);
			\draw [line width=0.7pt] (26.,-8.)-- (18.62458336760644,-3.82794748986379);
			\draw [line width=0.7pt] (18.62458336760644,-3.82794748986379)-- (16.313013413235893,-1.6523522386915042);
			\draw [line width=0.7pt] (16.313013413235893,-1.6523522386915042)-- (16.,2.);
			\draw [line width=0.7pt] (1.969995552421627,31.540556573904766)-- (0.9992164620058273,22.28348453315405);
			\draw [line width=0.7pt] (0.9992164620058273,22.28348453315405)-- (4.639638051065076,16.59749271786147);
			\draw [line width=0.7pt,dash pattern=on 4pt off 4pt,color=qqqqff] (11.203862180634122,15.3221514630554)-- (7.239939186107397,15.210665445838892);
			\draw [line width=0.7pt,dash pattern=on 4pt off 4pt,color=qqqqff] (7.239939186107397,15.210665445838892)-- (4.639638051065076,16.59749271786147);
			\draw [line width=0.7pt,dash pattern=on 4pt off 4pt,color=qqqqff] (13.057553607539958,4.123720993518399)-- (15.226085198101115,3.2224515151484394);
			\draw [line width=0.7pt,dash pattern=on 4pt off 4pt,color=qqqqff] (15.226085198101115,3.2224515151484394)-- (16.,2.);
			\draw (11.6,21) node[anchor=north west] {$\phi([0,1])$};
			\draw (16.2,15) node[anchor=north west] {$V_{i}$};
			\draw [color=ffqqqq](-1.3,17.7) node[anchor=north west] {$(x_i^-,y_i^-)$};
			\draw [color=ffqqqq](16.2,3.4) node[anchor=north west] {$(x_i^+ y_i^+)$};
			\draw (18.4,15.3) node[anchor=north west] {$\hat{V}_i$};
			\draw (6.5,21) node[anchor=north west] {$\tilde{V}_{i-1}$};
			\draw (13.6,0) node[anchor=north west] {$\tilde{V}_i$};
			\begin{scriptsize}
				\draw [fill=ffqqqq] (16.,2.) circle (2.5pt);
				\draw [fill=ffqqqq] (4.639638051065076,16.59749271786147) circle (2.5pt);
			\end{scriptsize}
		\end{tikzpicture}
		\caption{We find curves that connect $(x_{i}^-, y^-_{i})$ to $(x_{i}^+, y^+_{i})$ by extending $\gamma_i$ from $[0,1]$. The added image is displayed in blue. The blue curves can be constructed so they do not cross $\phi([0,1])$ and the final curve is injective.}\label{fig:Connecting}
	\end{figure}
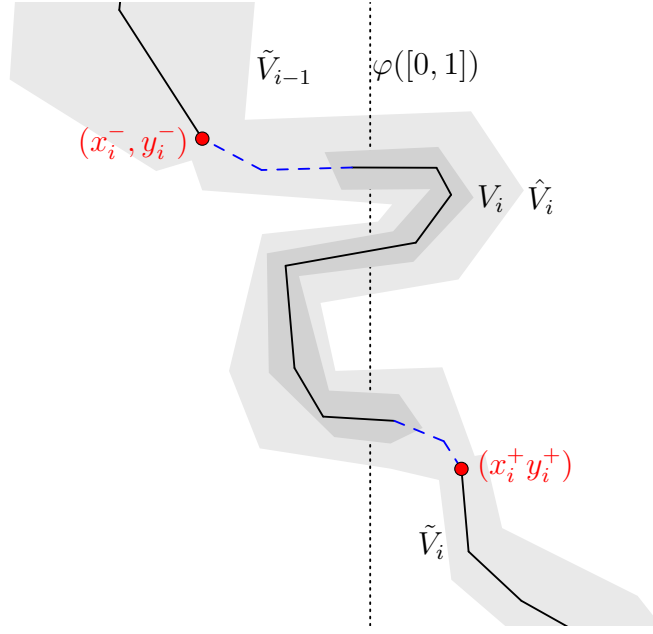
	
	Because $\bigcup_{j=1}^K\gamma_{j}([-1,2])$ does not disconnect any $\tilde{V}_j$ we can find continuous injective curves $\tilde{\gamma}_i:[0,1] \to \er^2$ with $\tilde{\gamma}_i(0) = (x_{i-1}^+, y^+_{i-1})$ , $\tilde{\gamma}_i(1) = (x_{i}^-, y^-_{i})$ and $\tilde{\gamma}_i((0,1)) \subset \tilde{V}_j\setminus \bigcup_{j=1}^K\gamma_{j}([-1,2])$ for each $i=2,\dots, K$. We define $\gamma:[0,4K-1]\to \er^2$ as
	$$
	\gamma_1 + \tilde{\gamma}_1 + \gamma_2 + \dots + \tilde{\gamma}_{K-1} +\gamma_K.
	$$
	Let us prove that $\gamma$ is injective. Assume that $\gamma(s) = \gamma(t)$. Then we find a $\tilde{V}_i \ni \gamma(s)$ or a $\hat{V}_i \ni \gamma(s)$. The mappings $\gamma_i$ and $\tilde{\gamma}_i$ are injective and have disjoint images except for the common endpoint and therefore $s=t$. The map $\gamma_i$ was made to cross $\phi([0,1])$ exactly once at every point of $u^{-1}(w_i,z_i)\cap \phi([0,1])$. Therefore, we see that $u$ satisfies the three-curve property and by Theorem~\ref{ThreeCurveChar} $u\in  \overline{H\cap W_{\id}^{1,p}((-1,1)^2, \er^2)}$.
\end{proof}

\section{Examples}
\subsection{The set $S_u$ can be dense and the very weak monotonicity property is strictly weaker than the weak monotonicity property}
\begin{example}\label{I'mDense}
	There is a $u\in  \overline{H\cap W_{\id}^{1,p}((-1,1)^2, \er^2)}$ for all $p<2$ such that the set $S_u$ of discontinuity points of $u$ is dense. Further, $u$ satisfies the very weak monotonicity property, but not the weak monotonicity property.
\end{example}
\begin{lemma}\label{HaHaJoke}
	Let $1\geq R>r>0$. The mapping $u_{r, R}:(-1,1)^2 \to \er^2$ defined as 
	$$
	u(x,y) = \begin{cases}
		r \frac{(x,y)}{|(x,y)|} + (R-r)(x,y) \quad &(x,y)\in Q(0,R)\\
		(x,y) \quad &(x,y)\in (-1,1)^2 \setminus Q(0,R)\\
	\end{cases}
	$$
	belongs to $\overline{H\cap W_{\id}^{1,p}((-1,1)^2, \er^2)}$ for all $p<2$. Further,
	\begin{equation}\label{LilJoke}
		\int_{Q(0,R)}|Du_{r,R}|^p \leq   4R^{2-p}(r^p + R^p(R-r)^p)
	\end{equation}
	and $u_{r,R}$ satisfies the very weak monotonicity property, but not the weak monotonicity property.
\end{lemma}
\begin{proof}
	Let $1\geq R>r>0$. We define
	$$
	u_k(x,y) = \begin{cases}
		k(x,y) \quad & (x,y) \in Q(0,\tfrac{r}{k})\\
		r\frac{(x,y)}{|(x,y)|} + (R-r)(x,y) \quad &(x,y)\in Q(0,R) \setminus Q(0,\tfrac{r}{k})\\
		(x,y) \quad &(x,y)\in (-1,1)^2 \setminus Q(0,R).\\
	\end{cases}
	$$
	Then each $u_k$ is a Lipschitz homeomorphism. A direct computation easily gives
	$$
	\int_{(-1,1)^2}|Du_k|^p \leq  4\frac{r^2}{k^{2-p}} + 4r^pR^{2-p} + 4R^2(R-r)^p + (4-2R^2) < \infty
	$$
	and sending $k\to \infty$ we have
	$$
	\limsup_{k\to\infty} \int_{Q(0,R)}|Du_k|^p \leq   4R^{2-p}(r^p + R^p(R-r)^p)
	$$
	Since $u_k$ is a sequence of bounded functions converging in $L^1$ with $\int_{(-1,1)^2}|Du_k|^p$ bounded, we have that its limit is in $W^{1,p}((-1,1)^2,\er^2)$ for all $p\in (1,2)$.
	
	For each point $(w,z)\in \overline{Q(0,r)}$ we have $u^{-1}(w,z) = (0,0)$. For each point $(w,z)\in Q(0,R)\setminus \overline{Q(0,r)}$ we have
	$u^{-1}(w,z) = \frac{(w,z)}{|(w,z)|} \frac{|(w,z)| - r}{R-r}$ and $u^{-1}(w,z) = (w,z)$ for $|(w,z)|_{\infty}\geq R$. We define
	$\psi:[0,4]\to \er^2$ as follows
	$$
	\psi(t) = \begin{cases}
		(r(1-t) + Rt, 0) \quad & t\in [0,1]\\
		(R, R(1-t)) \quad & t\in [1,2]\\
		(R(3-t), R) \quad & t\in [2,3]\\
		(0, r(t-3) + R(4-t)) \quad & t\in [3,4].
	\end{cases}
	$$ 
	Clearly $\psi([0,4])$ is a closed set which does not disconnect $(-1,1)^2$. On the other hand, $u^{-1}(\psi[0,4])$ is the boundary of a square with corners $(0,0)$, $(R,0),$ $(R,R)$, $(0,R)$ which disconnects $(-1,1)^2$. Thus $u_{r,R}$ does not satisfy the weak monotonicity property.
\end{proof}
\textit{Construction of $u$ from Example~\ref{I'mDense}}

We construct a sequence of finite sets $A_k$, $k\in \en_0$ such that $\bigcup_{k=0}^{\infty} A_k$ is dense in $(-1,1)^2$ inductively. Let $A_0 = \{0\}$, $A_k = \big[(-1,1)^2 \cap 2^{-k}\zet^2\big] \setminus \bigcup_{j=0}^{k-1}A_j$. The number of points in $A_k$ is estimated by $2^{3k}$. Assume that $R_k \leq 2^{-2k-2}$ $r_k \leq 2^{-2k-4}$. For $k\in \en_0$ we define
$$
u_{k}(x,y) = \begin{cases}
	(a,b) + u_{r_k,R_k}((x,y) - (a,b)) \quad &(x,y)\in Q((a,b), R_k), (a,b)\in A_k\\
	(x,y) \quad &\text{elsewhere}.
\end{cases}
$$
We define $v_k = u_k\circ u_{k-1}\circ \dots \circ u_0$. Because $u_{j}$ is locally bi-Lipschitz on $(-1,1)^2\setminus A_j$ and $\dist(A_i, A_j) \geq 2^{-\min\{i,j\}}$ for $i\neq j$ we observe that $v_k\in W^{1,p}((-1,1)^2, \er^2)$ as soon as
$\int_{(-1,1)^2} |\nabla v_k|^p < \infty$. Let us calculate inductively 
$v_0 \in W^{1,p}((-1,1)^2, \er^2)$. By the choice of $R_k$ (especially 
$\dist (A_k ,\bigcup{j=0}^{k-1} A_j ) \geq 2R_k$), we find some $C_k$ such that $|\nabla v_{k-1}| \leq C_k$ on $A_k + Q(0,R_k)$. Then we can calculate
$|\nabla v_k(x,y)| \leq C_k|\nabla u_k(v_{k-1}(x,y))|$ for $(x,y) \in A_k + Q(0,R_k)$ and  $|\nabla v_k(x,y)| = |\nabla v_{k-1}(x,y)|$ elsewhere. By choosing $r_k$ and $R_k$ small enough (see \eqref{LilJoke}) we have
$$
\int|\nabla v_{k} - \nabla v_{k-1}|^p < 2^{-k}
$$ 
Thus $v_k$ is bounded in $W^{1,p}$. Convergence of $v_k$ in $L^1$ is easily observed and so we define $u = \lim_k v_k \in W^{1,p}((-1,1)^2, \er^2)$. Since each $u_k$ can be approximated by a Sobolev homeomorphism in $W^{1,p}$ we see that $u \in \overline{H\cap W_{\id}^{1,p}((-1,1)^2, \er^2)}$ for all $p<2$. The arguments that $u$ satisfy the very weak monotonicity property but not the monotonicity property are analogous to Lemma~\ref{HaHaJoke}. 
\qed

\subsection{A map satisfying $\INV$ and the NCL condition that does not satisfy the QM condition}
\begin{example}\label{NoWindingConditionFail} 
	There exists a $u\in W^{1,p}_{\id}((-1,1)^2,\er^2)$ for all $p\in[1,2)$ which satisfies the conditions $\INV$ and NCL, but does not satisfy the QM condition
\end{example}
The map here is from \cite{DPP}[Section 5.2]. We define $\omega:((0,4)\setminus\{1\})\times(-\pi,\pi)\to \er$ as follows
$$
\omega(r,\theta) = \begin{cases}
	\theta \quad &r\geq 2\\
	\frac{\theta}{r-1}  \quad &r\in (1,2), \theta\in ((1-r)\pi, (r-1)\pi)\\
	\pi \quad &r\in (1,2), \theta\in (-\pi, \pi)\setminus [(1-r)\pi, (r-1)\pi]\\
	\pi + \frac{\theta}{1-r}  \quad &r\in (\tfrac{1}{2},1), \theta\in (2(r-1)\pi, 2(1-r)\pi)\\
	\pi \quad &r\in (\tfrac{1}{2},1),  \theta\in (-\pi, \pi)\setminus [2(1-r)\pi, 2(r-1)\pi]\\
	\pi + 2\theta \quad &r\in (0,\tfrac{1}{2}], \theta\in (-\pi, \pi).
\end{cases}
$$
By $\Phi$ we denote standard polar coordinates, i.e. $\Phi(r,\theta) = (r\cos\theta, r\sin \theta)$, $r>0, \theta\in(-\pi,\pi)$. We define $u$ as
$$
u(x,y) = (r\cos\omega(r,\theta), 2\max\{0,r-2\}\sin\omega(r,\theta)) \quad \text{where } (x,y) = \Phi(r,\theta)
$$ 
for $|(x,y)|<4$ and $u(x,y) = (x,y)$ for $|(x,y)|\geq 4$.

Let us note some basic properties of $u$. Firstly $u$ is (or can be extended as) continuous on $\er^2\setminus \{(1,0)\}$. Secondly $u\in W^{1,p}_{\loc}(\er^2,\er^2)$ especially since the derivative is bounded everywhere away from the point $(1,0)$ and close to $(1,0)$ we can estimate
$$
|\nabla u(\Phi(r,\theta))| \approx |\nabla \omega(r,\theta)|\approx |1-r|^{-1}
$$
on the set $|\theta|<2|1-r|$. Integrating we have
$$
\int_{B((1,0),1)}|\nabla u|^p\leq C\int_{0}^{1}\int_{0}^s\tfrac{1}{s^{p-1}}\ dt\ ds < \infty
$$
for $p<2$.

From \cite{DPP}[Lemma 5.2], we have that $u$ satisfies the $\INV$ condition (so the preimage of a point is a closed connected set). By \cite{CPR}[Remark 3.2] it is known that this map satisfies the NCL condition. Let us now prove that the QM condition  is not satisfied.

The set $U_{\delta}$ can be taken to be $Q((1,0), \delta)$ for any $\delta >0$ It suffices to consider the preimage of the points $(-\tfrac{1}{2},0)$ and $(\tfrac{1}{2},0)$. From the definition of $u$ we easily check that 
$$
u^{-1}(-\tfrac{1}{2},0) \cap \overline{B(0,\tfrac{1}{2})} = \{(-\tfrac{1}{2},0), (\tfrac{1}{2},0)\} \text{ and } u^{-1}(\tfrac{1}{2},0) \cap \overline{B(0,\tfrac{1}{2})} = \{(0, -\tfrac{1}{2}), (0,\tfrac{1}{2})\}. 
$$
It is simple to check that $(1,0)\in u^{-1}(-\tfrac{1}{2},0)$ and $(1,0)\in u^{-1}(\tfrac{1}{2},0)$. By the continuity of $u$ on $\er^2\setminus\{(1,0)\}$ we have $u^{-1}(-\tfrac{1}{2},0)\cap u^{-1}(\tfrac{1}{2},0) = \{(1,0)\}$. Then $u^{-1}(-\tfrac{1}{2},0)$ and $u^{-1}(\tfrac{1}{2},0)$ are a pair of continua which are disjoint outside $Q((1,0), \delta)$. Note that any continuous curve from $(-\tfrac{1}{2},0)$ to $(\tfrac{1}{2},0)$ in $\er^2\setminus B(0,\tfrac{1}{2})$ must necessarily disconnect $(0, -\tfrac{1}{2})$ from $(0,\tfrac{1}{2})$ in $\er^2\setminus B(0,\tfrac{1}{2})$.

Let assume that $V^{1}_{4^{-1},4^{-1}}, V^{2}_{4^{-1},4^{-1}}$ are a pair of disjoint Jordan domains with $V^{1}_{4^{-1},4^{-1}}\subset u^{-1}(B((-\tfrac{1}{2},0), \tfrac{1}{4})) \cup Q((1,0), \tfrac{1}{4})$ containing $u^{-1}(-\tfrac{1}{2},0)\setminus Q((1,0), \tfrac{1}{4})$ and similarly for $V^{1}_{4^{-1},4^{-1}}$. Then there is a pair of curves, one going from $(-\tfrac{1}{2},0)$ to $(\tfrac{1}{2},0)$ in $V^{1}_{4^{-1},4^{-1}}$ and the other going from $(0, -\tfrac{1}{2})$ to $(0,\tfrac{1}{2})$ in $V^{2}_{4^{-1},4^{-1}}$. Since $V^{1}_{4^{-1},4^{-1}}\cap V^{2}_{4^{-1},4^{-1}} = \emptyset$, these curves do not intersect. Thus there is a continuous curve from $(-\tfrac{1}{2},0)$ to $(\tfrac{1}{2},0)$ in $\er^2\setminus B(0,\tfrac{1}{2})$ which does not disconnect $(0, -\tfrac{1}{2})$ from $(0,\tfrac{1}{2})$ in $\er^2\setminus B(0,\tfrac{1}{2})$, which is a contradiction.
\qed

\vskip 5pt
\noindent
{\bf Acknowledgments.} The author would like to express his deep gratitude to Aldo Pratelli for his kind tutorship and the insight he shared with the author, without which this paper would not have come to be. The author was supported by the grants ERC-CZ Grant LL2105 CONTACT and GA\v CR grant P201/21-01976.

\end{document}